\documentclass[a4paper,reqno]{amsart}
\usepackage{amsmath}
\usepackage{amssymb}
\usepackage{cite}
\allowdisplaybreaks

\theoremstyle{plain}
\newtheorem{theorem}{\bf Theorem}[section]
\newtheorem{proposition}[theorem]{\bf Proposition}
\newtheorem{lemma}[theorem]{\bf Lemma}

\theoremstyle{definition}

\newcommand{\N}{\mathbb N}

\newcommand{\R}{\mathbb R}

\newcommand{\Div}{\mathrm{div}}

\numberwithin{equation}{section}

\begin{document}
\title[]{Sign-changing concentration phenomena of an anisotropic sinh-Poisson type equation with a Hardy or H\'{e}non term}

\thanks{$^\dag$ This work was supported by the National Natural Science Foundation of China (no. 12001298) and Inner Mongolia Autonomous Region Natural Science Foundation(no. 2024LHMS01001)}

\author[]{Qiang Ren}
\address{School of Mathematical Sciences \\ Inner Mongolia University \\ Hohhot 010021 \\ People's Republic of China} \email{renq@imu.edu.cn}



\begin{abstract}
We consider the following anisotropic sinh-Poisson type equation with a Hardy or H\'{e}non term:
\begin{equation}\label{115}
\left\{\begin{array}{ll}
-\Div (a(x)\nabla u)+ a(x)u=\varepsilon^2a(x)|x-q|^{2\alpha}(e^u-e^{-u}) &\mbox{in $\Omega$,} \\
\frac{\partial u}{\partial n}=0, &\mbox{on $\Omega$,}
\end{array}\right.
\end{equation}
where $\varepsilon>0$, $q\in \bar\Omega\subset \R^2$, $\alpha \in(-1,\infty)\char92 \N$, $\Omega\subset \R^2$ is a smooth bounded domain, $n$ is the unit outward normal vector of $\partial \Omega$ and anisotropic coefficient $a(x)$ is a smooth positive function defined on $\bar\Omega$. From finite dimensional reduction method, we proved that the problem \eqref{115} has a sequence of sign-changing solutions with arbitrarily many interior spikes accumulating to $q$, provided $q\in \Omega$ is a local maximizer of $a(x)$. However, if $q\in \partial \Omega$ is a strict local maximum point of $a(x)$ and satisfies $\langle \nabla a(q),n \rangle=0$, we proved that  \eqref{115} has a family of sign-changing solutions with arbitrarily many mixed interior and boundary spikes accumulating to $q$.

Under the same condition, we could also construct a sequence of blow-up solutions for the following problem
\begin{equation*}
\left\{\begin{array}{ll}
-\Div (a(x)\nabla u)+ a(x)u=\varepsilon^2a(x)|x-q|^{2\alpha}e^u &\mbox{in $\Omega$,} \\
\frac{\partial u}{\partial n}=0, &\mbox{on $\partial\Omega$.}
\end{array}\right.
\end{equation*}
\end{abstract}
\maketitle

\section{Introduction}

Standard sinh-Poisson equation is of the following form:
\[-\Delta u=\varepsilon^2\left(e^u-e^{-u}\right).\]
It was used to describe the vortex-type configurations for 2D turbulent Euler flow and the geometry of constant mean curvature surfaces( c.f. \cite{Bartolucci_Pistoia2007,Gurarie_Chow2004,Chow_Ko_Leung_Tang1998,Kuvshinov_Schep2000, Mallier_Maslowe1993,Wente1986}). In the recent decades, the Dirichlet problem of the sinh-Poisson equation and its variant attracted widespread attention(see \cite{Figueroa2021, Pistoia_Ricciardi2016,D'Aprile2015,Wei2011} and references therein).

The following problem was first considered by Spruck:
\begin{equation}\label{101}
\left\{\begin{array}{ll}
-\Delta u=\varepsilon^2(e^u-e^{-u}) &\mbox{in $\Omega$,} \\
u=0, &\mbox{on $\partial\Omega$.}
\end{array}\right.
\end{equation}
Under the condition that $\Omega\subset \R^2$ is a rectangle containing the origin, Spruck \cite{Spruck1988}  constructed a family of positive solutions of \eqref{101} which blow up at the origin as $\varepsilon\to 0$.
Later, Jost \textit{et al.}\cite{Jost_Wang_Ye_Zhou2008} studied the blow-up behaviours of the solutions of \eqref{101} as $\varepsilon \to 0$. They proved that the number of blow-up points of each solution to \eqref{101} is finite and the mass of each blow-up point of any solution to \eqref{101} is a multiple of $8\pi$. In the case that the domain $\Omega\subset \R^2$  is symmetric with respect to the origin and $0\in \Omega$, Grossi and Pistoia \cite{Grossi_Pistoia2013} constructed a sequences of sign-changing radial solutions of \eqref{101} with a multiple blow-up point at the origin. Bartsch \textit{et al.} \cite{Bartsch_Pistoia_Weth2010,Bartsch_Pistoia_Weth2015} considered the relationship between the blow-up behaviours of the solutions to \eqref{101} and the critical points of the following general Kirchhoff-Routh path function
\[H_{\Omega}(x_1, x_2,\cdots, x_N)=\sum_{i=1}^N \Gamma_i^2h(x_i)+\sum_{i\not=j}\Gamma_i\Gamma_j G(x_i,x_j),\]
where $h(\cdot)$ and $G(\cdot,\cdot)$ are the Robin function and Green function of the Dirichlet Laplacian on $\Omega$, respectively. They found some sign-changing blow-up solutions of \eqref{101}, and the locations of these blow-up points constitute a critical point of $H_{\Omega}$.

The following sinh-Poisson equation with H\'{e}non term arise from the study of Chern-Simons vortex theory(see \cite{Yang2001,Bethuel_Brezis_Helein1994})
\begin{equation}\label{118}
\left\{\begin{array}{ll}
-\Delta u=\varepsilon^2|x|^{2\alpha}(e^u-e^{-u}) &\mbox{in $\Omega$,} \\
u=0, &\mbox{on $\partial\Omega$.}
\end{array}\right.
\end{equation}
On the smooth domain $\Omega$ containing the origin, D'Aprile \cite{D'Aprile2015} constructed a family of sign-changing blow-up solutions of \eqref{118}, whose blow-up points are uniformly away from the origin and the boundary $\partial\Omega$. Moreover, the number of blow-up points of the solutions constructed in \cite{D'Aprile2015} is less than 4. However, for $\Omega=B_1(0)\subset \R^2$ and  any integer $m>0$, Zhang and Yang \cite{Zhang_Yang2013} constructed a family of sign-changing solutions to \eqref{118} which blow up at the origin and $m$ points on a circle in $B_1(0)$.

Only few results available considering the following Neumann problem of sinh-Poisson equation:
\begin{equation*}
\left\{\begin{array}{ll}
-\Delta u=\varepsilon^2(e^u-e^{-u}) &\mbox{in $\Omega$,} \\
\frac{\partial u}{\partial n}=0, &\mbox{on $\partial\Omega$.}
\end{array}\right.
\end{equation*}
Given any non-negative numbers $k$ and $l$ satisfying $k+l\geq 1$, Wei \textit{et al.}\cite{Wei_Wei_Zhou2011,Wei2009} found a sequence of solutions which blow up at $k$ points in $\Omega$ and $l$ points on $\partial \Omega$ as $\varepsilon\to 0$.

Motivated by these work above, Wang and Zheng \cite{Wang_Zheng2022} considered the following problem:
\begin{equation}\label{100}
\left\{\begin{array}{ll}
-\Delta u+ u=\varepsilon^2|x-q_1|^{2\alpha_1}\cdots|x-q_s|^{2\alpha_s}(e^u-e^{-u}) &\mbox{in $\Omega$,} \\
\frac{\partial u}{\partial n}=0, &\mbox{on $\partial\Omega$,}
\end{array}\right.
\end{equation}
where $q_1, q_2,\cdots, q_s\in \Omega$, and $\alpha_1, \cdots, \alpha_s\in (0,+\infty)\char92 \N$.
Via Lyapunov-Schmidt reduction method, they proved that given any two nonnegative integer $k$ and $l$ satisfying $k+l\geq 1$, there exists a solution $u_\varepsilon$ to \eqref{100} satisfying
\[\varepsilon^2|x-q_1|^{2\alpha_1}\cdots|x-q_s|^{2\alpha_s}(e^{u_\varepsilon}-e^{-u_\varepsilon})\rightharpoonup \sum_{j=1}^k 8\pi \delta_{\xi_j}+\sum_{j=1}^l4\pi \delta_{\xi_{k+j}}+8\pi \sum_{i=1}^s(1+\alpha)\delta_{q_i},\quad \mathrm{as} \quad \varepsilon\to 0,\]
where $\xi_1, \cdots, \xi_l \in \Omega$ and $\xi_{l+1}, \cdots, \xi_{k+l}\in \partial \Omega$.

In the case of $s=1$, problem \eqref{100} is equivalent to the following one:
\begin{equation}\label{102}
\left\{\begin{array}{ll}
-\Delta v+ v=\varepsilon^2|x-q|^{2\alpha}(e^v-e^{-v}) &\mbox{in $\Omega$,} \\
\frac{\partial v}{\partial n}=0, &\mbox{on $\partial \Omega$.}
\end{array}\right.
\end{equation}
However, in the case of $s=1$ and  $\alpha_1=0$, problem \eqref{100} is reduced to
\begin{equation}\label{123}
\left\{\begin{array}{ll}
-\Delta u+ u=\varepsilon^2(e^u-e^{-u}) &\mbox{in $ D$,} \\
\frac{\partial u}{\partial n}=0, &\mbox{on $\partial D$,}
\end{array}\right.
\end{equation}
which is similar to
\begin{equation}\label{113}
\left\{\begin{array}{ll}
-\Delta u+ u=\varepsilon^2e^u &\mbox{in $\Omega$,} \\
\frac{\partial u}{\partial n}=0, &\mbox{on $\partial\Omega$.}
\end{array}\right.
\end{equation}
The equation \eqref{113} is used to represent steady-states of the Keller-Segal system(c.f. \cite{delPino_Wei2006, Agudelo_Pistoia2016, Keller_Segel1970} and references therein). The blow-up behaviours of the solutions to \eqref{113} were analysed in \cite{Nagasaki_Suzuki1990}. Given any nonnegative integers $k$ and $l$ such that $k+l\geq1$, del Pino and Wei \cite{delPino_Wei2006} found a sequence of solutions to \eqref{113} which blow up at $k$ points in $\Omega$ and $l$ points on $\partial \Omega$ as $\varepsilon\to 0$. In the case of $\Omega=B_{r_0}(0)\subset \R^2$, Pistoia and Vaira \cite{Pistoia_Vaira2015} found a sequence of solutions to \eqref{113} which are concentrated near the whole boundary $\partial B_{r_0}(0)$. Later, del Pino \textit{et al.}\cite{delPino_Pistoia_Vaira2016} extended the results in \cite{Pistoia_Vaira2015} to more general domain in $\R^2$. To explore the boundary concentration phenomena of \eqref{113} in high dimensional case, Agudelo and Pistoia \cite{Agudelo_Pistoia2016} considered the following problem
\begin{equation}\label{114}
\left\{\begin{array}{ll}
-\Div (a(x)\nabla u)+ a(x)u=\varepsilon^2a(x)e^u &\mbox{in $\Omega$,} \\
\frac{\partial u}{\partial n}=0, &\mbox{on $\partial\Omega$,}
\end{array}\right.
\end{equation}
where $\Omega\subset \R^2$. For any integer $m\geq1$, they found solutions of \eqref{114} with $m$ interior layers or $m$ boundary layers which blow up at $m$ strict local minimum points or strict local maximum points of $a(x)$ restricted to $\partial \Omega$. They also found a sequence of solutions to \eqref{114} with $m$ boundary layers which blow up at the same strict local maximum point of $a(x)$ restricted to $\partial \Omega$.

Equation \eqref{114} is called isotropic if $a(x)\equiv 1$ and it is called anisotropic if the function $a(x)\not\equiv 1$( see \cite{Wei_Ye_Zhou2007}). Anisotropic problems attract much attention in recent years.  Using a parabolic approach, Ciani \textit{et al.} \cite{Ciani_Skrypnik_Vespri2023} establish some elliptic estimate of the following more general singular anisotropic elliptic equations:
\[\sum_{i=1}^s \partial_{ii}u+\sum_{i=s+1}^N \partial_i(A_i(x,u,\nabla u))=0, \qquad in \;\; \Omega\subset \R^N,\]
where $1\leq s\leq N-1$. However,  Clop \textit{et al.} \cite{Clop_Gentile_PassarellidiNapoli2023} studied the minimizer of the following the following non-autonomous non-homogeneous functional
\[\mathcal{F}(w, \Omega)=\int_\Omega \left[F(x, Dw(x))-f(x)w(x)\right]dx.\]
Under some sub-quadratic growth condition of $F$, they proved the minimizer of the functional $\mathcal{F}(w, \Omega)$ belongs to $W^{2,p}(\Omega)$, where $p\in(0,1)$. Papageorgiou \textit{et al.} \cite{Papageorgiou_Radulescu_Sun2024} studied the following more general anisotropic singular Dirichlet problem
\begin{equation}\label{120}
\left\{\begin{array}{ll}
-\Div a(z, Du)=\beta(z)u(z)^{-\eta(z)}+f(z, u(z)), \qquad \mbox{in $\Omega$}, \\
u|_{\partial \Omega}=0, \qquad u>0,
\end{array}\right.
\end{equation}
where $a: \Omega \times \R^N\to \R^N$ is measurable in $z\in \Omega$, continuous and monotone in $x\in \R^N$, positive function $\beta\in L^{+\infty}(\Omega)$ and $\eta\in C(\bar(\Omega))$ with range in $(0,1)$. From variational method, they proved \eqref{120} has at least two smooth solutions.

Eigenvalue problem is also a class of interesting problem in partial differential equations. In \cite{Du_Mao_Wang_Xia_Zhao2023}, Du \textit{et al.} established Li-Yau-Kr\"{o}ger-type estimates of the eigenvalue of the following Neumann-type eigenvalue problem
\begin{equation*}
\left\{
\begin{array}{ll}
\Delta^2 u - \tau \Delta u=\Lambda u, &\mbox{in $\Omega$}, \\
(1-\sigma)\frac{\partial^2 u}{\partial n^2}+\sigma \Delta u=0,  &\mbox{on $\partial\Omega$}, \\
\tau \frac{\partial u}{\partial n}-(1-\sigma)\Div_{\partial \Omega}(D^2 u\cdot n)_{\partial \Omega}-\frac{\partial \Delta u}{\partial n}=0,  &\mbox{on $\partial\Omega$}.
\end{array}
\right.
\end{equation*}
In the same paper, they also give some Reilly-type estimates of the following Steklov-type eigenvalue problem
\begin{equation*}
\left\{
\begin{array}{ll}
\Delta^2 u - \tau \Delta u=0, &\mbox{in $\Omega$}, \\
(1-\sigma)\frac{\partial^2 u}{\partial n^2}+\sigma \Delta u=0,  &\mbox{on $\partial\Omega$}, \\
\tau \frac{\partial u}{\partial n}-(1-\sigma)\Div_{\partial \Omega}(D^2 u\cdot n)_{\partial \Omega}-\frac{\partial \Delta u}{\partial n}=\lambda u,  &\mbox{on $\partial\Omega$}.
\end{array}
\right.
\end{equation*}

Now we consider the following anisotropic problem
\begin{equation}\label{112}
\left\{\begin{array}{ll}
-\Div (a(x)\nabla u)+ a(x)u=\varepsilon^2a(x)|x-q|^{2\alpha}e^u &\mbox{in $\Omega$,} \\
\frac{\partial u}{\partial n}=0, &\mbox{on $\partial\Omega$,}
\end{array}\right.
\end{equation}
where $\Omega\subset \R^2$ is a smooth bounded domain, $q\in \bar\Omega$, $\alpha \in(-1,\infty)\char92 \N$, and $a(x)$ is a smooth positive function defined on $\bar\Omega$.
Here we impose a Hardy or H\'{e}non weight on the right hand side of the first equation in \eqref{114}. It is natural to ask whether we could find some blow-up solutions to \eqref{112} as $\varepsilon \to 0$. Similarly, we could also  consider the following anisotropic sinh-Poisson type equation with a Hardy or H\'{e}non weight:
\begin{equation}\label{1}
\left\{\begin{array}{ll}
-\Div (a(x)\nabla u)+ a(x)u=\varepsilon^2a(x)|x-q|^{2\alpha}(e^u-e^{-u}) &\mbox{in $\Omega$,} \\
\frac{\partial u}{\partial n}=0, &\mbox{on $\partial\Omega$,}
\end{array}\right.
\end{equation}
where $\Omega\subset \R^2$ is some smooth bounded domain, $q\in \bar\Omega$, $\alpha \in(-1,\infty)\char92 \N$, and the anisotropic coefficient $a(x)$ is a smooth positive function defined on $\bar\Omega$.

Equation \eqref{1} is a natural generalization of \eqref{102}. In the isotropic case  $a(x)\equiv 1$, \eqref{1} is reduced to  \eqref{102}. One may expect similar result hold in the anisotropic case(where $a(x)$ is a function). In the case that $\alpha=0$, problem \eqref{1} arise from the problem of looking for the high-dimensional concentration phenomenon of the solution to \eqref{123} under the condition $D\subset \R^N$( $N\geq3$).

Let $b\in \{1,2\}$ to be fixed and $\Omega$ be a smooth bounded domain in $\R^2$ such that
\[\bar{\Omega}\subset \{(x_1,x_2)\in \R^b\times \R^{2-b}: x_i>0, i=1,b\}.\]
Denote $M:=\sum_{i=1}^b M_i$, where $M_i\geq 2$ are positive integers. Then
\[D:=\{(y_1,y_2, x')\in \R^{M_1}\times \R^{M_b}\times \R^{2-b}:(|y_1|, |y_2|, x')\in \Omega\}\]
is a smooth bounded domain in $\R^N$, where $N=M+2-b$. Now we consider the solution of \eqref{102} which is invariant under the action of the group $\Gamma:=\mathcal{O}(M_1)\times \mathcal{O}(M_b)$. Here $\mathcal{O}(M_i)$ is the group of linear isometries of $\R^{M_i}$. From direct computation, we know  that if the function $v$  defined by
\[v(y_1, y_b, x')=u(|y_1|, |y_b|, x')\]
solves \eqref{123}, then $u$ solves the following problem
\begin{equation}\label{103}
\left\{\begin{array}{ll}
-\Delta u+\sum_{i=1}^b \frac{M_i-1}{x_i}\frac{\partial u}{\partial x_i}+u=\varepsilon^2(e^u-e^{-u}), &\mbox{in $\Omega$;} \\
\frac{\partial u}{\partial n}=0, &\mbox{on $\partial\Omega$.}
\end{array}\right.
\end{equation}
By choosing the function $a(x_1,x_b):=x_1^{M_1-1}\cdot x_b^{M_b-1}$, we see that problem \eqref{103} is just \eqref{1} in the case of $\alpha=0$. However, in the case of $\alpha\not=0$, it is necessary to point out that \eqref{1} can not be seen as a special version of \eqref{102} in high dimensional case, due to the presence of the Hardy-H\'{e}non term.

In this paper, we will show that \eqref{1} has some non-simple blow-up solution due to the presence of Hardy-H\'{e}non term.

Since the classical sinh-Poisson Dirichlet problem \eqref{101} and \eqref{118} have a plenty of sign-changing solutions, it is natural to ask if \eqref{1} also has some sign-changing ones? We will deal with this problem in this paper. We also show that \eqref{1} has some solutions which are different from that of the solutions to \eqref{102}.

Notice the energy functional to \eqref{1} is
\begin{equation}\label{6}
J(u)=\frac12\int_\Omega a(x)\left[|\nabla u|^2+u^2\right]-\varepsilon^2 \int_\Omega a(x)|x-q|^{2\alpha}(e^u+e^{-u}).
\end{equation}

Let $G(x,y)$ be the Green function with respect to the Neumann operator $-\Delta-\nabla(\log a(x))\cdot \nabla +1$ on $\Omega$, \textit{i.e.} $G(x,y)$ is the solution of the following problem:
\begin{equation}
\left\{\begin{array}{ll}
-\Delta G(x,y)-\nabla(\log a(x))\cdot \nabla  G(x,y)+G(x,y)=\delta_y(x), &\mbox{$x\in \Omega$,} \\
\frac{\partial G}{\partial n}=0,  &\mbox{$x\in \partial\Omega$.}
\end{array}\right.
\end{equation}
The regular part of $G(x,y)$ is defined by
\begin{equation}\label{104}
H(x,y)=\left\{\begin{array}{ll}
G(x,y)+\frac1{2\pi}\log |x-y|, &\mbox{$y\in \Omega$,} \\
G(x,y)+\frac1\pi \log |x-y|,   &\mbox{$y\in \partial\Omega$.}
\end{array}\right.
\end{equation}
Let $\kappa=m(3m+\alpha+1)$. Via the Lyapunov-Schmidt reduction method, we get the following two theorems.
\begin{theorem}\label{th1}
Let $\alpha \in (-1, +\infty)\char92\N$ and $q\in \Omega$ be a strict local maximum point of $a(x)$. For any integer $m\geq 1$ and $b_i\in\{-1,1\}$, $i=0,1\cdots,m$, there exists $\varepsilon_m>0$ such that for any constant $\varepsilon\in (0,\varepsilon_m)$,  \eqref{1} has a family of solutions $u_\varepsilon$  of the form
\begin{eqnarray*}
u_\varepsilon(x)&=& \sum_{i=1}^m b_i \left[\log\frac1{[(\varepsilon \mu_i)^2+|x-\xi_i^\varepsilon|^2]^2} +8\pi H(x, \xi_i^\varepsilon)\right]+b_0\log\frac1{[(\varepsilon\mu_0)^{2(1+\alpha)}+|x-q|^{2(1+\alpha)}]^2} \\
 &&+ 8\pi(\alpha +1)b_0 H(x,q)+o(1),
\end{eqnarray*}
where  $o(1)\to 0$ on any compact subset of $\bar{\Omega}\char92\{q, \xi_1^\varepsilon, \xi_2^\varepsilon, \cdots, \xi_m^\varepsilon\}$ as $\varepsilon \to 0$ and $\xi_i^\varepsilon \to q$   as $\varepsilon \to 0$. The constants $\mu_i$'s satisfy
\begin{equation*}
 \frac{C_0}{|\log\varepsilon|^{C_2}}\leq \mu_0^{2(1+\alpha)}\varepsilon^{2\alpha}\leq C_1|\log\varepsilon|^{C_2} \quad \mathrm{and} \quad \frac{C_3}{|\log\varepsilon|^{C_5}}\leq \frac{\mu_j^2}{|\xi_j^\varepsilon-q|^{2\alpha}}\leq C_4|\log\varepsilon|^{C_5}, \quad j=1,2,\cdots, m.
\end{equation*}
The points $\xi_i^\varepsilon$'s satisfy
\[|\xi_i^\varepsilon-\xi_j^\varepsilon|\geq \frac1{|\log\varepsilon|^\kappa}\quad \mathrm{for} \quad i\not=j \quad \mathrm{and} \quad|\xi_i^\varepsilon-q|\geq \frac1{|\log\varepsilon|^\kappa}\quad \mathrm{for} \quad i=1,2,\cdots,m.\]
\end{theorem}

\begin{theorem}\label{th2}
Let $\alpha \in (-1, +\infty)\char92 \N$. $q\in \partial \Omega$ is a strict local maximum point of $a(x)$ on $\bar{\Omega}$ and satisfies $\partial_{n} a(q)=0$. For any integers $m\geq 1$, $l\in[0,m]\cap\N$ and $b_i\in\{-1,1\}$, $i=0,1\cdots,m$, there exists constant $\varepsilon_m>0$ such that for any constant $\varepsilon\in (0,\varepsilon_m)$, problem \eqref{1} has a family of solutions $u_\varepsilon$ with $m-l+1$ different boundary spikes and $l$ different interior spikes which accumulate to $q$ as $\varepsilon \to 0$. More precisely,
\begin{eqnarray*}
  u_\varepsilon(x) &=&\sum_{i=1}^lb_i\left[ \log\frac1{[(\varepsilon\mu_i)^2+|x-\xi_i|^2]^2} +8\pi H(x, \xi_i)\right] \\
   &&+ \sum_{i=l+1}^m b_i\left[\log \frac1{[(\varepsilon\mu_i)^2+|x-\xi_i|^2]^2} +4\pi H(x, \xi_i) \right]\\
   && +b_0\log\frac1{[(\varepsilon \mu_0)^{2(1+\alpha)}+|x-q|^{2(1+\alpha)}]^2} +4\pi(1+\alpha)b_0H(x,q)+o(1),
\end{eqnarray*}
where $(\xi_1^\varepsilon, \xi_2^\varepsilon, \cdots, \xi_m^\varepsilon)\in \Omega^l \times (\partial \Omega)^{m-l}$ satisfy
\[|\xi_i^\varepsilon-\xi_j^\varepsilon|\geq \frac1{|\log\varepsilon|^\kappa}\quad  \mathrm{for} \quad i\not=j, \qquad \mathrm{and} \qquad|\xi_i^\varepsilon-q|\geq \frac1{|\log\varepsilon|^\kappa}\quad  \mathrm{for} \quad i=1,2,\cdots,m, \]
the term $o(1)\to 0$ as $\varepsilon \to 0$, on any compact subset of $\bar{\Omega}\char92\{q, \xi_1^\varepsilon, \xi_2^\varepsilon, \cdots, \xi_m^\varepsilon\}$, $\xi_i^\varepsilon \to q$ as $\varepsilon \to 0$, and the constants $\mu_i$'s satisfy
\begin{equation*}
 \frac{C_0}{|\log\varepsilon|^{C_2}}\leq \mu_0^{2(1+\alpha)}\varepsilon^{2\alpha}\leq C_1|\log\varepsilon|^{C_2} \quad \mathrm{and} \quad
\frac{C_3}{|\log\varepsilon|^{C_5}}\leq \frac{\mu_j^2}{|\xi_j^\varepsilon-q|^{2\alpha}}\leq C_4|\log\varepsilon|^{C_5}, \quad j=1,2,\cdots, m.
\end{equation*}
\end{theorem}
Notice that if $\{b_i:i=0,1,\cdots,m\}=\{-1,1\}$, blow-up solutions  we found in Theorem \ref{th1} and Theorem \ref{th2} are sign-changing ones. From Theorem \ref{th1} and Theorem \ref{th2}, we know that the interaction between the anisotropic function $a(x)$ and $q$ will make the concentration point $q$ of solutions to \eqref{1} to be non-simple. There are arbitrarily many spiky sequences of points converging to $q$. However, In the isotropic case $a(x)\equiv 1$, we cannot find solutions with these property.

For \eqref{112}, we also get the following theorems.
\begin{theorem}\label{th3}
Let $\alpha \in (-1, +\infty)\char92\N$ and $q\in \Omega$ be a strict local maximum point of $a(x)$. For any integer $m\geq 1$, there exists $\varepsilon_m>0$ such that for any $\varepsilon$ saitisfying $0<\varepsilon<\varepsilon_m$,  problem \eqref{112} has a family of solutions $u_\varepsilon$ which satisfies
\begin{eqnarray*}
u_\varepsilon(x)&=& \sum_{i=1}^m \left[\log\frac1{[(\varepsilon \mu_i)^2+|x-\xi_i^\varepsilon|^2]^2} +8\pi H(x, \xi_i^\varepsilon)\right]+\log\frac1{[(\varepsilon\mu_0)^{2(1+\alpha)}+|x-q|^{2(1+\alpha)}]^2} \\
 &&+ 8\pi(\alpha +1) H(x,q)+o(1),
\end{eqnarray*}
where  $o(1)\to 0$ as $\varepsilon \to 0$, on any compact subset of $\bar{\Omega}\char92\{q, \xi_1^\varepsilon, \xi_2^\varepsilon, \cdots, \xi_m^\varepsilon\}$ and $\xi_i^\varepsilon \to q$ as $\varepsilon \to 0$. Moreover,  the constants $\mu_i$'s satisfy
\begin{equation*}
 C_0\leq \mu_0^{2(1+\alpha)}\varepsilon^{2\alpha}\leq C_1|\log\varepsilon|^{C_2}, \quad \mathrm{and} \quad C_3\leq \frac{\mu_j^2}{|\xi_j^\varepsilon-q|^{2\alpha}}\leq C_4|\log\varepsilon|^{C_5}, \quad j=1,2,\cdots, m.
\end{equation*}
The points $\xi_i^\varepsilon$'s satisfy
\[|\xi_i^\varepsilon-\xi_j^\varepsilon|\geq \frac1{|\log\varepsilon|^\kappa}\quad \mathrm{for} \quad i\not=j\]
and
\[|\xi_i^\varepsilon-q|\geq \frac1{|\log\varepsilon|^\kappa}\quad \mathrm{for} \quad i=1,2,\cdots,m.\]
\end{theorem}

\begin{theorem}\label{th4}
Let $\alpha \in (-1, +\infty)\char92 \N$. $q\in \partial \Omega$ is a strict local maximum point of $a(x)$ on $\bar{\Omega}$ and satisfies $\partial_{n} a(q)=0$. For any integers $m\geq 1$ and $l\in[0,m]\cap\N$, there exists constant $\varepsilon_m>0$ such that for $0<\varepsilon<\varepsilon_m$,  problem \eqref{112} has a family of solutions $u_\varepsilon$ with $m-l+1$ different boundary spikes and $l$ different interior spikes which accumulate to $q$ as $\varepsilon \to 0$. More precisely
\begin{eqnarray*}
  u_\varepsilon(x) &=&\sum_{i=1}^l\left[ \log\frac1{[(\varepsilon\mu_i)^2+|x-\xi_i|^2]^2} +8\pi H(x, \xi_i)\right] + \sum_{i=l+1}^m \left[\log \frac1{[(\varepsilon\mu_i)^2+|x-\xi_i|^2]^2} +4\pi H(x, \xi_i) \right]\\
   && +\log\frac1{[(\varepsilon \mu_0)^{2(1+\alpha)}+|x-q|^{2(1+\alpha)}]^2} +4\pi(1+\alpha)H(x,q)+o(1),
\end{eqnarray*}
where  $(\xi_1^\varepsilon, \xi_2^\varepsilon, \cdots, \xi_m^\varepsilon)\in \Omega^l \times (\partial \Omega)^{m-l}$, $o(1)\to 0$ as $\varepsilon \to 0$, on any compact subset of $\bar{\Omega}\char92\{q, \xi_1^\varepsilon, \xi_2^\varepsilon, \cdots, \xi_m^\varepsilon\}$.  The constants $\mu_i$'s satisfy
\begin{equation*}
 C_0\leq \mu_0^{2(1+\alpha)}\varepsilon^{2\alpha}\leq C_1|\log\varepsilon|^{C_2}, \quad \mathrm{and} \quad C_3\leq \frac{\mu_j^2}{|\xi_j^\varepsilon-q|^{2\alpha}}\leq C_4|\log\varepsilon|^{C_5}, \quad j=1,2,\cdots, m,
\end{equation*}
And $\xi_i^\varepsilon \to q$ as $\varepsilon \to 0$. What is more, the points $\xi_i^\varepsilon$'s  satisfy
\[|\xi_i^\varepsilon-\xi_j^\varepsilon|\geq \frac1{|\log\varepsilon|^\kappa}\quad  \mathrm{for} \quad i\not=j\]
and
\[|\xi_i^\varepsilon-q|\geq \frac1{|\log\varepsilon|^\kappa}\quad  \mathrm{for} \quad i=1,2,\cdots,m.\]
\end{theorem}

We use the finite dimensional reduction method in \cite{Agudelo_Pistoia2016,Zhang2021,Zhang2022} to prove these theorems. Since the proof of Theorem \ref{th3} and Theorem \ref{th4} is similar to that of Theorem \ref{th1} and Theorem \ref{th2}. So we only prove Theorem \ref{th1} and Theorem \ref{th2} for simplicity.

In the proof of Theorem \ref{th1} and Theorem \ref{th2}, we will frequently use the estimates of Green function and its regular part established in \cite[Section 2]{Agudelo_Pistoia2016}. Readers can also refer this fact in \cite[Appendix A]{Zhang2022}.

This paper is organized as follows. An approximate solution of \eqref{1} is constructed in Section \ref{sect1} and its error term is estimated in Section \ref{sect7}. Then the Lyapunov-Schmidt reduction is conducted, where the linear problem and nonlinear problem are investigated in Section \ref{sect2} and Section \ref{sect3}, respectively. Section \ref{sect5} and Section \ref{sect6} is devoted to prove Theorem \ref{th1} and Theorem \ref{th2}. The expansion of the energy functional and some technical lemmata which are involved are proved in Section \ref{sect4}.

In this paper, $C$ denotes a constant which is varying from line to line.

\section{Approximate solution}\label{sect1}
It is well known(c.f. \cite{Prajapat_Tarantello2001}) that for $\alpha \in (-1, +\infty)\char92 \N$, all the solutions of
\begin{equation*}
\left\{\begin{aligned}
&-\Delta u=|x|^{2\alpha}e^u, \quad \mbox{in $\R^2$,} \\
&\int_{\R^2} |x|^{2\alpha}e^u<\infty
\end{aligned}\right.
\end{equation*}
are of the form
\begin{equation*}
\tilde{V}_{\delta}(x)=\log \frac{8(1+\alpha)^2\delta^{2(1+\alpha)}}{(\delta^{2(1+\alpha)}+|x|^{2(1+\alpha)})^2}.
\end{equation*}
However, all the solutions of
\begin{equation*}
\left\{\begin{aligned}
&-\Delta u=e^u, \quad \mbox{in $\R^2$,} \\
&\int_{\R^2} e^u<\infty
\end{aligned}\right.
\end{equation*}
are of the form
\[V_{\delta, \xi}(x)=\log \frac{8\delta^2}{(\delta^2+|x-\xi|^2)^2},\]
where $\delta>0$ and $\xi\in \R^2$( see \cite{Chen_Li1991,Chen_Lin2002}). We will use these functions to construct approximate solution to \eqref{1}.

For $m\geq 1$ and $l\in [0,m]\cap \N$, we define
\begin{equation}\label{12}
\begin{aligned}
\mathcal{O}_\varepsilon=&\left\{\xi=(\xi_1, \cdots, \xi_m)\in(B_d(q)\cap \Omega)^l\times(B_d(q)\cap \partial\Omega)^{m-l}\;:\;\min_{1\leq i\leq l} \mathrm{dist}(\xi_i, \partial \Omega)>\frac1{|\log \varepsilon|^\kappa}\right., \\
&\qquad\qquad\qquad\qquad\qquad \left.\min_{i=1,\dots, m}|\xi_i-q|>\frac1{|\log \varepsilon|^\kappa}, \; \min_{i,j=1,\cdots, m; i\not=j}|\xi_i-\xi_j|>\frac1{|\log \varepsilon|^\kappa}\right\}.
\end{aligned}
\end{equation}
where $d>0$ is a small constant defined in Section \ref{sect2}.

It is apparent that for  $\xi=(\xi_1, \cdots, \xi_m)\in \mathcal{O}_\varepsilon$, the functions
\[U_0(x):=\log\frac{8(1+\alpha)^2\mu_0^{2(1+\alpha)}\varepsilon^{2\alpha}}{\left[(\varepsilon \mu_0)^{2(1+\alpha)}+|x-q|^{2(1+\alpha)}\right]^2}\]
and
\[U_i(x):=\log \frac{8\mu_i^2}{\left[(\varepsilon \mu_i)^2+|x-\xi_i|^2\right]^2|\xi_i-q|^{2\alpha}},\qquad i=1,2,\cdots,m,\]
solve the following problems
\begin{equation}\label{84}
-\Delta U_0=\varepsilon^2|x-q|^{2\alpha}e^{U_0} \quad \mathrm{in}\quad \R^2
\end{equation}
and
\begin{equation}\label{85}
-\Delta U_i=\varepsilon^2 |\xi_i-q|^{2\alpha}e^{U_i} \quad \mathrm{in}\quad \R^2, \qquad i=1,2,\cdots,m,
\end{equation}
respectively. In the expressions above,  $\mu_i$'s are constants which are defined in \eqref{4}.

Now we define an approximate solution of \eqref{1} by
\begin{equation}\label{5}
U_\xi(x)=\sum_{i=0}^m b_iPU_i(x)=\sum_{i=0}^mb_i\left[U_i(x)+H_i(x)\right],
\end{equation}
where $b_i\in\{-1,1\}$ for $i=0,1,\cdots,m$ and $H_i(x)$ is the unique solution of
\begin{equation}\label{86}
\left\{\begin{array}{ll}
-\Delta H_i-\nabla(\log a(x))\cdot\nabla H_i+H_i=\nabla(\log a(x))\cdot \nabla U_i-U_i, &\mbox{in $\Omega$,} \\
\frac{\partial H_i}{\partial n}=-\frac{\partial U_i}{\partial n}, &\mbox{on $\partial\Omega$.}
\end{array}\right.
\end{equation}
For notation simplicity, we denote
\[c_0=\left\{\begin{array}{ll}
8\pi(1+\alpha), &\mbox{$q\in \Omega$,} \\
4\pi(1+\alpha), &\mbox{$q\in \partial\Omega$,}
\end{array}\right.\quad \mathrm{and} \quad
c_i=\left\{\begin{array}{ll}
8\pi, &\mbox{$\xi_i\in \Omega$,} \\
4\pi, &\mbox{$\xi_i\in \partial\Omega$.}
\end{array}\right.\]
From the definition of regular part of Green function $H(\cdot,\cdot)$ in \eqref{104}, we know the function $H(x, q)$ solves the following problem
\[
\left\{\begin{array}{ll}
-\Delta H(x,q) -\nabla (\log a(x))\cdot \nabla H(x, q)+ H(x, q)=&\frac{4(1+\alpha)}{c_0}\log|x-q| \\
&-\frac{4(1+\alpha)}{c_0}\frac{\nabla(\log a(x))\cdot (x-q)}{|x-q|^2}, \qquad\mbox{in $\Omega$,} \\
\frac{\partial H(x,q)}{\partial n}=\frac{4(1+\alpha)}{c_0}\frac{(x-q,n)}{|x-q|^2}, \qquad\qquad\mbox{on $\partial \Omega$,}
\end{array}\right.\]
and $H(x,\xi_i)$ solves
\[
\left\{\begin{array}{ll}
-\Delta H(x,\xi_i) -\nabla (\log a(x))\cdot \nabla H(x, \xi_i)+ H(x, \xi_i)=-\frac{4}{c_i}\frac{\nabla(\log a(x))\cdot (x-\xi_i)}{|x-\xi_i|^2}+ \frac{4}{c_i}\log|x-\xi_i|, &\mbox{in $\Omega$,} \\
\frac{\partial H(x,\xi_i)}{\partial n}=\frac{4}{c_i}\frac{(x-\xi_i,n)}{|x-\xi_i|^2}, &\mbox{on $\partial \Omega$.}
\end{array}\right.\]
In order to get some estimate of $U_\xi$, we need to find the relationship between $H_i$ and the regular part of Green functions $H(\cdot, \cdot)$.
\begin{lemma}\label{lm1}
For $\forall p\in(1,2)\char92\{\frac1{1+\alpha},\frac2{3+2\alpha}, \frac1{2(1+\alpha)}\}$ and $\forall \xi=(\xi_1, \xi_2, \cdots, \xi_m)\in \mathcal{O}_\varepsilon$, the following estimates hold in $C^1(\Omega)$ sense:
\begin{equation}\label{2}
H_0(x)=c_0H(x,q)-\log8(1+\alpha)^2\mu_0^{2(1+\alpha)}\varepsilon^{2\alpha}+O\left((\varepsilon \mu_0)^{\min\left\{\frac2p-1, 2(1+\alpha)\right\}}\right),
\end{equation}
and
\begin{equation}\label{3}
H_i(x)=c_iH(x,\xi_i)-\log \frac{8\mu_i^2}{|\xi_i-q|^{2\alpha}}+O((\varepsilon \mu_i)^{\frac2p-1}), \quad i=1,2,\cdots, m.
\end{equation}
\end{lemma}

\begin{proof}
Let
\[Y_0(x):=H_0(x)-c_0H(x, q)+\log 8(1+\alpha)^2\mu_0^{2(1+\alpha)}\varepsilon^{2\alpha}.\]
It is a solution of the following problem
\begin{equation*}
\left\{\begin{array}{ll}
-\Delta Y_0-\nabla(\log a(x))\cdot \nabla Y_0 +Y_0=&\frac{4(1+\alpha)(\varepsilon \mu_0)^{2(1+\alpha)}\nabla(\log a)\cdot (x-q)}{|x-q|^2[(\varepsilon\mu_0)^{2(1+\alpha)}+|x-q|^{2(1+\alpha)}]} \\
& -2\log \frac{|x-q|^{2(1+\alpha)}}{(\varepsilon \mu_0)^{2(1+\alpha)}+|x-q|^{2(1+\alpha)}}, \qquad \mbox{in $\Omega$,} \\
\frac{\partial Y_0}{\partial n}=-\frac{4(1+\alpha)(\varepsilon\mu_0)^{2(1+\alpha)}(x-q,n)}{|x-q|^2[(\varepsilon\mu_0)^{2 (1+\alpha)}+|x-q|^{2(1+\alpha)}],} &\qquad\qquad\qquad\qquad\qquad\qquad\qquad\mbox{on $\partial\Omega$.}
\end{array}\right.
\end{equation*}
It is apparent
\begin{eqnarray*}
  \int_{\Omega}\left|\frac{4(1+\alpha)(\varepsilon \mu_0)^{2(1+\alpha)}\nabla(\log a)\cdot (x-q)}{|x-q|^2[(\varepsilon\mu_0)^{2(1+\alpha)}+|x-q|^{2(1+\alpha)}]}\right|^pdx  &\leq& C(\varepsilon \mu_0)^{2-p}\int_{\tilde{\Omega}_{\varepsilon \mu_0}}\left|\frac1{|t|[1+|t|^{2(1+\alpha)}]}\right|^pdt \\
   &\leq& C(\varepsilon \mu_0)^{\min\left\{2-p,2(1+\alpha)p\right\}},
\end{eqnarray*}
where $\tilde{\Omega}_{\varepsilon \mu_0}=(\varepsilon\mu_0)^{-1}\left(\Omega-\{q\}\right)$.
Then we get
\[
\left\|\frac{4(1+\alpha)(\varepsilon \mu_0)^{2(1+\alpha)}\nabla(\log a)\cdot (x-q)}{|x-q|^2[(\varepsilon\mu_0)^{2(1+\alpha)}+|x-q|^{2(1+\alpha)}]} \right\|_{L^p(\Omega)}\leq C(\varepsilon \mu_0)^{\min\left\{\frac2p-1, 2(1+\alpha)\right\}}.
\]
From the same method, we also get
\[
\left\|2\log \frac{|x-q|^{2(1+\alpha)}}{(\varepsilon \mu_0)^{2(1+\alpha)}+|x-q|^{2(1+\alpha)}}\right\|_{L^p(\Omega)}\leq C(\varepsilon \mu_0)^{\min\left\{\frac2p, 2(1+\alpha)\right\}} .
\]

Now we estimate the boundary terms. In the case of $q\in \Omega$, we get
\[\frac{\partial Y_0}{\partial n}=O\left((\varepsilon \mu_0)^{2(1+\alpha)}\right).\]
However in the case of $q\in \partial\Omega$, we have
\begin{eqnarray*}
\int_{\partial \Omega\cap B_{\delta_1}(q)}\left|\frac{4(1+\alpha)(\varepsilon\mu_0)^{2(1+\alpha)}(x-q,n)}{|x-q|^2[ (\varepsilon\mu_0)^{2(1+\alpha)}+|x-q|^{2(1+\alpha)}]}\right|^pdx&\leq& C\varepsilon \mu_0 \int_{\tilde{\Omega}_{\varepsilon \mu_0}} \left|\frac1{1+|t|^{2(1+\alpha)}}\right|^p dt \\
&\leq& C (\varepsilon\mu_0)^{\min\{1, 2(1+\alpha)p\}},
\end{eqnarray*}
with the help of \cite[Lemma 3]{Agudelo_Pistoia2016}, where $\delta_1>0$ is a small constant. It is obvious that
\[\int_{\partial \Omega\cap B^c_{\delta_1}(q)}\left|\frac{4(1+\alpha)(\varepsilon\mu_0)^{2(1+\alpha)}(x-q,n)}{|x-q|^2[ (\varepsilon\mu_0)^{2(1+\alpha)}+|x-q|^{2(1+\alpha)}]}\right|^pdx\leq C(\varepsilon \mu_0)^{2(1+\alpha)p}.\]
Then
\[\left\|\frac{\partial Y_0}{\partial n}\right\|_{L^p(\partial\Omega)}\leq C(\varepsilon \mu_0)^{\min\left\{\frac1p, 2(1+\alpha)\right\}}.\]
From elliptic estimate and Sobolev embedding theorem,  we get
\[\|Y_0\|_{C^\tau(\bar{\Omega})}\leq C\|Y_0\|_{W^{1+\theta, p}}\leq C(\varepsilon \mu_0)^{\min\{\frac2p-1, 2(1+\alpha)\}}\]
where $0<\theta<\frac1p$ and $0<\tau<\frac12+\frac1p$.
Hence \eqref{2} holds.

To prove \eqref{3}, we define
\[Y_i(x):=H_i(x)-c_iH(x,\xi_i)+\log\frac{8\mu_i^2}{|\xi_i-q|^{2\alpha}}.\]
It solves the following boundary value problem
\[
\left\{\begin{array}{ll}
-\Delta Y_i-\nabla(\log a(x))\cdot \nabla Y_i+Y_i=4\frac{\varepsilon^2\mu_i^2\nabla (\log a)\cdot (x-\xi_i)}{|x-\xi_i|^2(\varepsilon^2\mu_i^2+|x-\xi_i|^2)}-2\log \frac{|x-\xi_i|^2}{\varepsilon^2 \mu_i^2+|x-\xi_i|^2}, &\mbox{in $\Omega$,} \\
\frac{\partial Y_i}{\partial n}=-\frac{4\varepsilon^2\mu_i^2(x-\xi_i,n)}{|x-\xi_i|^2(\varepsilon^2\mu_i^2+|x-\xi_i|^2)}, &\mbox{on $\partial\Omega$.}
\end{array}\right.
\]
From direct computation, we have
\[\left\|4\frac{\varepsilon^2\mu_i^2\nabla (\log a)\cdot (x-\xi_i)}{|x-\xi_i|^2(\varepsilon^2\mu_i^2+|x-\xi_i|^2)}\right\|_{L^p(\Omega)}\leq C(\varepsilon \mu_i)^{\frac2p-1}\]
and
\[\left\|\log \frac{|x-\xi_i|^2}{\varepsilon^2 \mu_i^2+|x-\xi_i|^2}\right\|_{L^p(\Omega)}\leq C(\varepsilon \mu_i)^{\frac2p}.\]
For $1\leq i\leq l$, we get $\mathrm{dist}(\xi_i, \partial \Omega)\geq \frac1{|\log\varepsilon|^\kappa}$. There holds that
\[\left\|\frac{\partial Y_i}{\partial n}\right\|_{L^\infty(\partial\Omega)}\leq C\varepsilon^2\mu_i^2|\log \varepsilon|^{3\kappa}.\]
However, in the case of $l+1\leq i\leq m$, $\xi_i\in \partial \Omega$. We get
\[\int_{\partial \Omega\cap B_{\delta_1}(\xi_i)}\left|\frac{4\varepsilon^2\mu_i^2(x-\xi_i,n)}{|x-\xi_i|^2(\varepsilon^2\mu_i^2 +|x-\xi_i|^2)}\right|^pdx\leq C\varepsilon \mu_i,\]
with the help of \cite[Lemma 3]{Agudelo_Pistoia2016}. It is obvious that
\[\int_{\partial \Omega\cap B^c_{\delta_1}(\xi_i)}\left|\frac{4\varepsilon^2\mu_i^2(x-\xi_i,n)}{|x-\xi_i|^2(\varepsilon^2 \mu_i^2+|x-\xi_i|^2)}\right|^pdx\leq C(\varepsilon \mu_i)^{2p}.\]
Hence
\begin{equation}\label{74}
\left\|\frac{\partial Y_i}{\partial n}\right\|_{L^p(\partial\Omega)}\leq C(\varepsilon\mu_i)^{\frac1p}.
\end{equation}
From elliptic estimate, we get
\[\|Y_i\|_{W^{1+\theta, p}}\leq C(\varepsilon \mu_0)^{\frac2p-1}, \qquad \mathrm{where} \quad 1<p<2, \; 0<\theta<\frac1p.\]
Using Sobolev embedding theorem, we have
\[\|Y_i\|_{C^\tau(\bar{\Omega})}\leq C(\varepsilon \mu_0)^{\frac2p-1}, \qquad \mathrm{where}\quad 0<\tau<\frac12+\frac1p.\]
Then \eqref{2} holds.
\end{proof}

From now on, we always choose
\[p\in(1,2)\char92\{\frac1{1+\alpha},\frac2{3+2\alpha}, \frac1{2(1+\alpha)}\}.\]
With Lemma \ref{lm1}, we will derive some local expressions of $U_\xi(x)$.

In the case of $|x-q|<\frac1{|\log \varepsilon|^{2\kappa}}$, we get the following expansion with the help of Lemma \ref{lm1}:
\begin{align}\label{73}
 \nonumber U_\xi(x)= & \;b_0U_0(x)+c_0b_0 H(x,q) +\sum_{i=1}^mb_i \left[\log \frac1{[(\varepsilon \mu_i)^2+|x-\xi_i|^2]^2}+ c_i H(x, \xi_i)\right] \\
 \nonumber & \; -b_0\log 8(1+\alpha)^2 \mu_0^{2(1+\alpha)}\varepsilon^{2\alpha}+O\left(\sum_{i=1}^m \left(\varepsilon \mu_i\right)^{\frac2p-1}\right)+O\left((\varepsilon \mu_0)^{\min\left\{\frac2p-1, 2(1+\alpha)\right\}}\right) \\
  \nonumber= & \;b_0U_0(x) +b_0c_0H(q,q) -b_0\log 8(1+\alpha)^2 \mu_0^{2(1+\alpha)}\varepsilon^{2\alpha}+\sum_{i=1}^mb_i\left[-4\log|q-\xi_i|+c_i H(q, \xi_i)\right] \\
   &\;+O(|x-q|^\beta)+O\left(\sum_{i=1}^m \left(\varepsilon \mu_i\right)^{\frac2p-1}\right)+ O\left((\varepsilon \mu_0)^{\min\left\{\frac2p-1, 2(1+\alpha)\right\}}\right).
\end{align}
Hence
\begin{eqnarray*}
U_\xi(x)&=&b_0U_0(x)+b_0c_0H(q,q)-b_0\log 8(1+\alpha)^2\mu_0^{2(1+\alpha)}\varepsilon^{2\alpha}+\sum_{i=1}^mb_i c_iG(q, \xi_i)+O\left(|x-q|^\beta\right) \\
&&+O\left(\sum_{i=1}^m(\varepsilon\mu_i)^{\frac2p-1}\right)+O\left((\varepsilon \mu_0)^{\min\left\{\frac2p-1, 2(1+\alpha)\right\}}\right),
\end{eqnarray*}
where $1<p<2$ and $0<\beta<\frac12$.

However, in the case of $|x-\xi_j|<\frac1{|\log\varepsilon|^{2\kappa}}$, where $j\in\{1, \cdots, m\}$, we get
\begin{eqnarray*}
U_\xi(x)&=&b_jU_j(x)+b_jc_j H(\xi_j, \xi_j)-b_j\log \frac{8\mu_j^2}{|\xi_j-q|^{2\alpha}}+\sum_{i=1, i\not= j}^m b_ic_i G(\xi_j,\xi_i)+b_0c_0G(\xi_j, q)\\
&&+O\left(|x-\xi_j|^\beta\right)+O\left(\sum_{i=1}^m(\varepsilon\mu_i)^{\frac2p-1}\right)+O\left((\varepsilon \mu_0)^{\min\left\{\frac2p-1, 2(1+\alpha)\right\}}\right)
\end{eqnarray*}
from the same method as in \eqref{73}.

To construct a good approximate solution, we choose  $\{\mu_i\}_{i=0}^m$ to be the solution the following problem:
\begin{equation}\label{4}
\left\{
\begin{aligned}
&b_0\log8(1+\alpha)^2\mu_0^{2(1+\alpha)}\varepsilon^{2\alpha}=b_0c_0 H(q,q)+\sum_{i=1}^m b_ic_i G(q, \xi_i), \\
&b_j\log\frac{8\mu_j^2}{|\xi_j-q|^{2\alpha}}=b_jc_j H(\xi_j, \xi_j)+b_0c_0G(\xi_j, q)+\sum_{i=1, i\not=j}^m b_ic_i G(\xi_j, \xi_i).
\end{aligned}\right.
\end{equation}
From \eqref{4} and \cite[Appendix A]{Zhang2022}, we get that $\mu_i$'s satisfy
\begin{equation}\label{51}
\left\{
\begin{aligned}
& \frac{C_0}{|\log\varepsilon|^{C_2}}\leq \mu_0^{2(1+\alpha)}\varepsilon^{2\alpha}\leq C_1|\log\varepsilon|^{C_2} \\
&\frac{C_3}{|\log\varepsilon|^{C_5}}\leq \frac{\mu_j^2}{|\xi_j-q|^{2\alpha}}\leq C_4|\log\varepsilon|^{C_5}, \quad j=1,2,\cdots, m.
\end{aligned}\right.
\end{equation}
Here we notice that $\mu_k$'s are not always bounded away from $0$. In particular, we have
\[C_0\frac{\varepsilon^2}{|\log\varepsilon|^{C_2}}\leq (\varepsilon\mu_0)^{2(1+\alpha)}\leq C_1 \varepsilon^2 |\log\varepsilon|^{C_2}.\]
For $j=1,2,\cdots, m$ and any constant $\gamma>0$, we also get
\[\varepsilon^\gamma \mu_j\to0, \quad \mathrm{as}\quad \varepsilon\to 0.\]
Then $U_\xi$ satisfies the following estimates:
\begin{itemize}
  \item For $|x-q|<\frac1{|\log \varepsilon|^{2\kappa}}$,
  \begin{equation}\label{8}
  U_\xi(x)=b_0U_0(x)+O\left(|x-q|^\beta\right)+O\left(\sum_{i=0}^m(\varepsilon\mu_i)^{\frac2p-1}\right);
  \end{equation}
  \item For $|x-\xi_j|<\frac1{|\log \varepsilon|^{2\kappa}}$,
  \begin{equation}\label{75}
  U_\xi(x)=b_jU_j(x)+O\left(|x-\xi_j|^\beta\right)+O\left(\sum_{i=0}^m(\varepsilon\mu_i)^{\frac2p-1}\right);
  \end{equation}
  \item For any $x$ satisfying $|x-q|>\frac1{|\log \varepsilon|^{2\kappa}}$ and $|x-\xi_j|>\frac1{|\log \varepsilon|^{2\kappa}}$, $j=1,2,\cdots, m$,
  \begin{equation}\label{53}
  U_\xi(x)=b_0c_0G(x,q)+\sum_{i=1}^m b_ic_i G(x,\xi_i)+O\left(\sum_{i=0}^m(\varepsilon\mu_i)^{\frac2p-1}\right),
  \end{equation}
\end{itemize}
where we get \eqref{53} from Lemma \ref{lm1} and the same method as in \eqref{73}.

If $u$ is a solution of \eqref{1}, the function
\[v(y)=u(\varepsilon y)+4\log\varepsilon\]
solves the following problem:
\begin{equation}\label{15}
\left\{\begin{aligned}
&-\Delta v-\nabla(\log a(\varepsilon y))\cdot\nabla v +\varepsilon^2(v-4\log\varepsilon)=|\varepsilon y-q|^{2\alpha}(e^v-\varepsilon^8e^{-v}), \qquad\mbox{in $\Omega_\varepsilon$} \\
&\frac{\partial v}{\partial n}=0, \qquad \mathrm{on}\quad \partial \Omega; \qquad v>0,  \qquad \mathrm{in}\quad \Omega,
\end{aligned}\right.
\end{equation}
where $\Omega_\varepsilon=\Omega/\varepsilon$. From the relationship between \eqref{1} and \eqref{15}, we define an approximate solution of \eqref{15} by
\begin{equation}\label{17}
V(y):=U_\xi(\varepsilon y)+4\log\varepsilon=\sum_{i=0}^m b_i\left[U_i(\varepsilon y)+H_i(\varepsilon y)\right]+4\log\varepsilon
\end{equation}

If $v(y)=V(y)+\phi(y)$  is a solution of \eqref{15},  $\phi$ solves
\begin{equation}\label{124}
\left\{\begin{array}{ll}
L(\phi)=R+N(\phi),  & \mathrm{in}\quad  \Omega_\varepsilon, \\
 \frac{\partial \phi}{\partial n}=0, & \mathrm{on}\quad \partial \Omega_\varepsilon,
\end{array}\right.
\end{equation}
where
\begin{equation}\label{87}
L(\phi):=-\Delta \phi-\nabla(\log a(\varepsilon y))\cdot \nabla \phi+\varepsilon^2\phi-W\phi,
\end{equation}
\begin{equation}\label{16}
W(y)=|\varepsilon y-q|^{2\alpha}(e^V+\varepsilon^8e^{-V}),
\end{equation}
\begin{equation}\label{49}
R=\Delta V+\nabla(\log a(\varepsilon y))\cdot \nabla V-\varepsilon^2(V-4\log\varepsilon)+|\varepsilon y-q|^{2\alpha}(e^V-\varepsilon^8e^{-V})
\end{equation}
and
\begin{equation}\label{50}
N(\phi)=|\varepsilon y-q|^{2\alpha}\left[e^{V+\phi}-e^V-e^V\phi-\varepsilon^8\left(e^{-V-\phi}-e^{-V}+e^{-V} \phi\right)\right].
\end{equation}
To prove \eqref{15} has a solution of the form $v(y)=V(y)+\phi(y)$,  we only need to show \eqref{124} has a  solution $\phi$.

\section{Some estimate of the error}\label{sect7}

Denote $q'=q/\varepsilon$ and $\xi_i'=\xi_i/\varepsilon$ where $i=1,2\cdots,m$. In this section, we will give some estimate of the error term $R$ and its derivatives with respect to $\xi_{kl}'$. To finish this task, we first estimate the derivatives of $\mu_i$'s and $H_i$'s with respect to $\xi_{kl}'$.

Taking derivatives with respect to $\xi_{kl}'$ on the both sides of \eqref{4}, we get
\[\partial_{\xi_{kl}'} \mu_0=\frac{\varepsilon c_k b_k\mu_0}{2(1+\alpha)b_0}\partial_{\xi_{kl}}G(q, \xi_k),\]
\[\partial_{\xi_{kl}'} \mu_j=\frac{\varepsilon b_k c_k \mu_j}{2b_j}\partial_{\xi_{kl}}G(\xi_j, \xi_k);\qquad \mathrm{for}\qquad k\not=j \]
and
\[\partial_{\xi_{kl}'} \mu_k =  \frac{\varepsilon\mu_k}{2b_k}\left\{b_kc_k \partial_{\xi_{kl}}\left[H(\xi_k, \xi_k)\right] +b_0c_0 \partial_{\xi_{kl}}G(\xi_k, q)+\sum_{i=1, i\not=k}^m b_ic_i \partial_{\xi_{kl}}G(\xi_k, \xi_i)+2\alpha b_k\frac{(\xi_k-q)_l}{|\xi_k-q|^2}\right\}.
\]
From the properties of Green function \cite{Zhang2022}, we get
\[\left|\partial_{\xi_{kl}'}\mu_i\right|\leq C\varepsilon \mu_i |\log\varepsilon|^\kappa\qquad \mathrm{for} \qquad i=0,1, \cdots, m;\quad k=1,2, \cdots, m;\quad l=1,J_k .\]
Using the estimates above and the same method as in Lemma \ref{lm1}, we get the following lemma.
\begin{lemma}\label{lm7}
For $k=1,2,\cdots, m$ and $l=1, J_k$, we get
\begin{equation}\label{61}
\|\partial_{\xi_{kl}'}H_0\|_{C^1(\Omega)}\leq C\varepsilon|\log\varepsilon|^\kappa, \qquad \|\partial_{\xi_{kl}'}H_i\|_{C^1(\Omega)}\leq C\varepsilon|\log\varepsilon|^\kappa, \quad \mathrm{for} \quad k\not=i
\end{equation}
and
\begin{equation}\label{55}
\|\partial_{\xi_{kl}'}H_k\|_{C^1(\Omega)}\leq \frac{C}{\mu_k}(\varepsilon \mu_k)^{\frac2p-1}.
\end{equation}
\end{lemma}

\begin{proof}
A direct computation yields
\begin{equation}\label{122}
\partial_{\xi_{kl}'}U_i(x)=\frac{4\varepsilon (x-\xi_i)_l}{(\varepsilon \mu_i)^2+|x-\xi_i|^2}\delta_{ik} +\frac2{\mu_i}\partial_{\xi_{kl}'}\mu_i-\frac{4\varepsilon^2\mu_i\partial_{\xi_{kl}'}\mu_i}{ (\varepsilon\mu_i)^2+|x-\xi_i|^2}-2\alpha\frac{(\xi_i'-q')_l}{|\xi_i'-q'|^2}\delta_{ik}
\end{equation}
and
\begin{eqnarray*}
\nabla \partial_{\xi_{kl}'}U_i(x)&=&\frac{4\varepsilon\delta_{ik}}{[(\varepsilon \mu_i)^2+|x-\xi_i|^2]^2}\left\{\left[(\varepsilon \mu_i)^2 +|x-\xi_i|^2\right]\vec{e}_l-2(x-\xi_i)_l(x-\xi_i)\right\} \\
&& +\frac{8\varepsilon^2\mu_i\partial_{\xi_{kl}'}\mu_i(x-\xi_i)}{\left[(\varepsilon \mu_i)^2+|x-\xi_i|^2\right]^2},
\end{eqnarray*}
where $\vec{e}_l$ is a vector whose $l$-component is $1$ and other components equal to $0$. Moreover, $\partial_{\xi_{kl}'} H_i$ is the unique solution of the following problem
\begin{equation*}
\left\{\begin{array}{ll}
-\Delta \partial_{\xi_{kl}'} H_i-(\nabla \log a(x))\cdot \nabla(\partial_{\xi_{kl}'} H_i) +\partial_{\xi_{kl}'} H_i=\nabla (\log a(x))\cdot \nabla \partial_{\xi_{kl}'} U_i -\partial_{\xi_{kl}'} U_i &\mbox{in $\Omega$} \\
\frac{\partial}{\partial n}(\partial_{\xi_{kl}'}H_i)= -\frac{\partial}{\partial n}(\partial_{\xi_{kl}'}U_i) &\mbox{on $\partial \Omega$}.
\end{array}\right.
\end{equation*}
For $k\not=i$, we get
\[\left\|\partial_{\xi_{kl}'} U_i\right\|_{L^\infty(\Omega)}\leq \frac{C}{\mu_i}|\partial_{\xi_{kl}'} \mu_i|\leq C\varepsilon |\log\varepsilon|^\kappa\]
and
\[\|\nabla \partial_{\xi_{kl}'} U_i\|_{L^p(\Omega)}\leq \frac{C}{\mu_i}(\varepsilon \mu_i)^{\frac2p}|\log\varepsilon|^\kappa.\]
For $i\in \{1,2,\cdots, l\}$, we have
\[\left\|\frac{\partial}{\partial n}(\partial_{\xi_{kl}'} U_i)\right\|_{L^\infty(\partial \Omega)}\leq \frac{C}{\mu_i}(\varepsilon \mu_i)^3|\log\varepsilon|^{4\kappa}.\]
Whereas, for $i\in\{l+1, \cdots,m\}$, we get the following estimate via the same method as in \eqref{74}:
\[\left\|\frac{\partial}{\partial n}(\partial_{\xi_{kl}'} U_i)\right\|_{L^p(\partial \Omega)}\leq \frac{C}{\mu_i}(\varepsilon \mu_i)^{\frac1p+1}|\log\varepsilon|^{\kappa}.\]
From elliptic estimate, we get
\[\left\|\partial_{\xi_{kl}'}H_i\right\|_{W^{1+\theta,p}(\Omega)}\leq C\varepsilon|\log\varepsilon|^\kappa.\]
Using Sobolev embedding theorem, we get the second inequality in  \eqref{61}.

Now we estimate \eqref{55}. From direct computation, we have
\[\|\nabla \partial_{\xi_{kl}'} U_k(x)\|_{L^p(\Omega)}\leq \frac{C}{\mu_k}(\varepsilon \mu_k)^{\frac2p-1}\]
and
\[\left\|\partial_{\xi_{kl}'} U_k(x)\right\|_{L^p(\Omega)}\leq C\varepsilon|\log\varepsilon|^{\kappa}.\]

Then we estimate the boundary terms. For $k=1,2, \cdots,l$, we get
\[\left\|\frac{\partial}{\partial n}\left(\partial_{\xi_{kl}'} U_k(x)\right)\right\|_{L^\infty(\partial\Omega)}\leq C \varepsilon|\log\varepsilon|^{2\kappa}.\]
However, in the case of $k=l+1, \cdots, m$, it holds that $\xi_k\in \partial \Omega$. In this case, the boundary $\partial \Omega$ is parameterized by $y_2-\xi_{k2}=G(x_1-\xi_{k1})$ in some neighborhood of $\xi_k$, where the function $G$ satisfies $G(0)=0$ and $G'(0)=0$.
Since the outward normal vector of $\partial \Omega$ near $\xi_k$ is
\[\vec{n}=\frac{(G'(x_1-\xi_{k1}), -1)}{\sqrt{1+|G'(x_1-\xi_{k1})|^2}},\]
we get
\[\left\|\frac{\partial}{\partial n}(\partial_{\xi_{k1}'} U_k)\right\|_{L^p(\partial \Omega)}\leq \frac{C}{\mu_k}(\varepsilon \mu_k)^{\frac1p},\]
from the same method as in \eqref{74}. Then we get \eqref{55} from elliptic estimate and Sobolev embedding theorem.

\end{proof}
From definition of \eqref{16} and \eqref{17} we get
\[W(y)=\varepsilon^4|\varepsilon y-q|^{2\alpha}\left(e^{U_\xi(\varepsilon y)}+e^{-U_\xi(\varepsilon y)}\right),\]
which is a positive function. We will derive some estimate of $W$ in Lemma \ref{lm10}, which is technical and is also used in constructing linear theory in Section \ref{sect2}.
\begin{lemma}\label{lm10}
For $y\in \Omega_\varepsilon$ satisfying $|\varepsilon y-q|<\frac1{|\log\varepsilon|^{2\kappa}}$, we get
\begin{equation}\label{105}
W(y)=\frac{8(1+\alpha)^2\mu_0^{2(1+\alpha)}|y-q'|^{2\alpha}}{[\mu_0^{2(1+\alpha)} +|y-q'|^{2(1+\alpha)}]^2}\left(1+O(|\varepsilon y-q|^\beta)+O\left(\sum_{k=0}^m(\varepsilon \mu_k)^{\frac2p-1}\right)\right) +O(\varepsilon^4|\log\varepsilon|^C)
\end{equation}
and
\begin{equation}\label{76}
W(y)\leq C\frac{\mu_0^{2(1+\alpha)}|y-q'|^{2\alpha}}{(\mu_0+|y-q'|)^{4(1+\alpha)}}+C\varepsilon^4|\log\varepsilon|^C.
\end{equation}

For $y\in \Omega_\varepsilon$ satisfying $|\varepsilon y-\xi_j|<\frac1{|\log\varepsilon|^{2\kappa}}$, where $j\in\{1,2,\cdots,m\}$, we get
\begin{equation}\label{77}
W(y)=\frac{8\mu_j^2}{[\mu_j^2+|y-\xi_j'|^2]^2}\frac{|\varepsilon y-q|^{2\alpha}}{|\xi_j-q|^{2\alpha}}\left(1+O(|\varepsilon y-\xi_j|^\beta)+O\left(\sum_{k=0}^m(\varepsilon \mu_k)^{\frac2p-1}\right)\right) +O(\varepsilon^4|\log\varepsilon|^C)
\end{equation}
and
\begin{equation}\label{83}
W(y)\leq C\frac{\mu_j^2}{(\mu_j+|y-\xi_j'|)^4}+C\varepsilon^4|\log\varepsilon|^C.
\end{equation}

For $y\in \Omega_\varepsilon$ satisfying $|\varepsilon y-q|>\frac1{|\log\varepsilon|^{2\kappa}}$ and  $|\varepsilon y-\xi_j|>\frac1{|\log\varepsilon|^{2\kappa}}$ for all $j\in\{1,2,\cdots,m\}$, we get
\begin{equation}\label{38}
W(y)=O(\varepsilon^4 |\log\varepsilon|^C).
\end{equation}
\end{lemma}
\begin{proof}
We first consider the  case of $|\varepsilon y-q|<\frac1{|\log\varepsilon|^{2\kappa}}$. Using the estimate \eqref{8}, we get
\begin{equation*}
e^{U_\xi(\varepsilon y)}=\exp\left[b_0U_0(\varepsilon y)+O\left(|\varepsilon y-q|^\beta\right)+O\left(\sum_{i=0}^m (\varepsilon \mu_i)^{\frac2p-1}\right)\right].
\end{equation*}
From \eqref{51} and the fact $b_0\in\{-1,1\}$, we have
\begin{eqnarray*}
e^{U_\xi(\varepsilon y)}+e^{-U_\xi(\varepsilon y)} &=& \left(e^{U_0(\varepsilon y)}+e^{-U_0(\varepsilon y)}\right)\left[1+O\left(|\varepsilon y-q|^\beta\right)+O\left(\sum_{i=0}^m (\varepsilon \mu_i)^{\frac2p-1}\right)\right] \\
&=&\frac{8(1+\alpha)^2\mu_0^{2(1+\alpha)}\varepsilon^{2\alpha}}{[(\varepsilon \mu_0)^{2(1+\alpha)}+|\varepsilon y-q|^{2(1+\alpha)}]^2}\left[1+O\left(|\varepsilon y-q|^\beta\right)+O\left(\sum_{i=0}^m (\varepsilon \mu_i)^{\frac2p-1}\right)\right] \\
&&+O\left(|\log\varepsilon|^C\right).
\end{eqnarray*}
Then we obtain the estimate \eqref{105} and \eqref{76}.

In the case of $|\varepsilon y-\xi_j|<\frac1{|\log\varepsilon|^{2\kappa}}$, we get \eqref{77} and \eqref{83} via the same procedure above with the help of \eqref{51} and \eqref{75}.

The estimate \eqref{38} follows from \eqref{53} and Lemmata in \cite{Agudelo_Pistoia2016}(see also \cite[Appendix A]{Zhang2022}).

\end{proof}
To conduct the Lyapunov-Schmidt reduction procedure, we introduce the following weighted $L^\infty(\Omega_\varepsilon)$ norm:
\begin{equation}\label{24}
\|h\|_{*}=\sup_{y\in \Omega_\varepsilon}\left(\varepsilon^2+\frac{\mu_0^{2+2\hat{\alpha}}|y-q'|^{2\alpha}}{(\mu_0+|y-q'|)^{4+2\alpha+2\hat{\alpha}}} +\sum_{j=1}^m \frac{\mu_j^\sigma}{(\mu_j+ |y-\xi_j'|)^{2+\sigma}}\right)^{-1}|h(y)|,
\end{equation}
where $-1<\hat{\alpha}<\alpha$ and $\sigma>0$ is a small enough constant. Then we give some estimates of the error term $R$ and its derivatives with respect to the norm $\|\cdot\|_*$.
\begin{lemma}
For $\varepsilon>0$ small enough, and $\xi\in \mathcal{O}_\varepsilon$, we have
\begin{equation}\label{54}
\|R\|_*\leq C\left[(\varepsilon \mu_0)^{\min\{\beta, 2(\alpha-\hat{\alpha}), \frac2p-1\}}+\sum_{i=1}^m (\varepsilon \mu_i)^{\min\{\beta, \frac{2}{p}-1\}}\right].
\end{equation}
and
\begin{equation}\label{62}
  \|\partial_{\xi_{kl}'}R\|_*\leq \frac{C}{\mu_k}\left((\varepsilon \mu_k)^\beta+\sum_{i=0}^m (\varepsilon \mu_i)^{\frac2p-1}\right),
\end{equation}
where $k=1,2,\cdots,m$ and $l=1, J_k$.
\end{lemma}

\begin{proof}
From the definition of $V$, we get
\begin{align*}
  K_1:= & \Delta V+\nabla (\log a(\varepsilon y))\cdot \nabla V -\varepsilon^2 (V-4\log\varepsilon) \\
   =& -\frac{8(1+\alpha)^2 b_0\mu_0^{2(1+\alpha)}|y-q'|^{2\alpha}}{[\mu_0^{2(1+\alpha)} +|y-q'|^{2(1+\alpha)}]^2}-\sum_{k=1}^m \frac{8b_k\mu_k^2}{[\mu_k^2+|y-\xi_k'|^2]^2}.
\end{align*}
whereas
\begin{equation*}
  K_2:=|\varepsilon y-q|^{2\alpha}\left[e^V-\varepsilon^8e^{-V}\right]=\varepsilon^{4+2\alpha}|y-q'|^{2\alpha} \left[e^{U_\xi(\varepsilon y)}-e^{-U_\xi(\varepsilon y)}\right].
\end{equation*}
In the case of $|y-q'|\leq \frac1{\varepsilon|\log\varepsilon|^{2\kappa}}$, we get
\begin{equation}\label{52}
  K_2=\frac{8(1+\alpha)^2 b_0\mu_0^{2(1+\alpha)}|y-q'|^{2\alpha}}{[\mu_0^{2(1+\alpha)}+|y-q'|^{2(1+\alpha)}]^2}\left(1 +O\left(|\varepsilon y-q|^\beta\right)+O\left(\sum_{i=0}^m (\varepsilon \mu_i)^{\frac2p-1}\right)\right)+O\left(\varepsilon^4|\log\varepsilon|^C\right)
\end{equation}
from the same argument of \eqref{105}.
With the help of \eqref{51}, we have
\[K_1=-\frac{8(1+\alpha)^2b_0 \mu_0^{2(1+\alpha)}|y-q'|^{2\alpha}}{[\mu_0^{2(1+\alpha)}+|y-q'|^{2(1+\alpha)}]^2} +O\left((\varepsilon |\log\varepsilon|^\kappa)^4\sum_{i=1}^m\mu_i^2\right).\]
Hence we get the following estimate in this case
\[
|R|\leq C\left((\varepsilon\mu_0)^{\min\{\beta, 2(\alpha-\hat{\alpha})\}}+\sum_{i=0}^m (\varepsilon \mu_i)^{\frac2p-1}\right)\left[\varepsilon^2+\frac{\mu_0^{2+2\hat{\alpha}}|y-q'|^{2\alpha}}{(\mu_0+|y-q'|)^{4+2\alpha+2\hat{\alpha}}} \right].
\]

In the case of $|y-\xi_j'|\leq \frac1{\varepsilon|\log\varepsilon|^{2\kappa}}$, we get
\[K_2=\frac{8b_j\mu_j^2}{[\mu_j^2+|y-\xi_j'|^2]^2}\frac{|y-q'|^{2\alpha}}{|\xi_j'-q'|^{2\alpha}} \left(1+O\left(|\varepsilon y-\xi_j|^\beta\right)+O\left(\sum_{i=0}^m (\varepsilon \mu_i)^{\frac2p-1}\right)\right)+O\left(\varepsilon^4 |\log\varepsilon|^C\right)\]
via the same method as in \eqref{77}. Hence
\[K_1=-\frac{8b_j\mu_j^2}{[\mu_j^2+|y-\xi_j'|^2]^2}+O\left((\varepsilon|\log \varepsilon|^\kappa)^4\sum_{k=1,k\not=j}^m\mu_k^2\right)+O\left(\varepsilon^{4+2\alpha}\mu_0^{2(1+\alpha)}|\log\varepsilon|^{(4+2\alpha)\kappa}\right).\]
Then we arrive at
\begin{equation*}
  |R|\leq C\left[(\varepsilon \mu_j)^\beta+\sum_{i=0}^m(\varepsilon \mu_i)^{\frac2p-1}\right]\left(\frac{\mu_j^\sigma}{(\mu_j+|y-\xi_j'|)^{2+\sigma}} +\varepsilon^2\right).
\end{equation*}

For $y$ satisfying $|y-\xi_i'|\geq \frac1{\varepsilon |\log\varepsilon |^{2\kappa}}$ for $i=1,2, \cdots,m$ and $|y-q'|\geq \frac1{\varepsilon |\log\varepsilon|^{2\kappa}}$, we have
\[|K_1|\leq C\left(\varepsilon^{4+2\alpha}\mu_0^{2(1+\alpha)}|\log\varepsilon|^{ 4\kappa(2+\alpha)}+\sum_{k=1}^m \mu_k^2[\varepsilon|\log\varepsilon|^{2\kappa}]^4\right).\]
Using the expression \eqref{53}, we get there exists constant $C_2>0$ such that $|U_\xi(\varepsilon y)|\leq C_2\log|\log\varepsilon|$. Then $|K_2|\leq C\varepsilon^4|\log\varepsilon|^{C_2}$. Hence $|R|\leq C\varepsilon^4|\log\varepsilon|^{C_3}$.

Combining all these estimate above, we get \eqref{54}.

Now we prove \eqref{62}. It is apparent that
\[\partial_{\xi_{kl}'}R=\Delta (\partial_{\xi_{kl}'}V)+\nabla(\log a(\varepsilon y))\cdot \nabla (\partial_{\xi_{kl}'}V)-\varepsilon^2 \partial_{\xi_{kl}'}V+|\varepsilon y-q|^{2\alpha}(e^V+\varepsilon^8 e^{-V})\partial_{\xi_{kl}'}V=: K_3+K_4\]
where
\[K_3=\Delta (\partial_{\xi_{kl}'}V)+\nabla(\log a(\varepsilon y))\cdot \nabla (\partial_{\xi_{kl}'}V)-\varepsilon^2 \partial_{\xi_{kl}'}V\]
and
\[K_4=|\varepsilon y-q|^{2\alpha}(e^V+\varepsilon^8 e^{-V})\partial_{\xi_{kl}'}V= W\partial_{\xi_{kl}'}V.\]
From the definition of $V$, \eqref{122} and Lemma \ref{lm7}, we get
\begin{equation}\label{82}
\partial_{\xi_{kl}'}V(y)=\partial_{\xi_{kl}'}U_\xi(\varepsilon y)=\frac{4b_k(y-\xi_k')_l}{\mu_k^2+|y-\xi_k'|^2}+O\left(\frac1{\mu_k}(\varepsilon \mu_k)^{\frac2p-1}\right)
\end{equation}
and
\begin{align*}
  K_3= & -16\frac{(1+\alpha)^3b_0|y-q'|^{2\alpha}\mu_0^{1+2\alpha}\partial_{\xi_{kl}'}\mu_0}{ [\mu_0^{2(1+\alpha)}+|y-q'|^{2(1+\alpha)}]^3}\left(|y-q'|^{2(1+\alpha)}-\mu_0^{2(1 +\alpha)}\right) \\
   & -\sum_{i=1}^m \frac{16b_i\mu_i\partial_{\xi_{kl}'}\mu_i}{[\mu_i^2+|y-\xi_i'|^2]^3}\left( |y-\xi_i'|^2-\mu_i^2\right)-\frac{32b_k\mu_k^2(y-\xi_k')_l}{[\mu_k^2+|y-\xi_k'|^2]^3}.
\end{align*}
Then
\begin{align*}
  K_3= & -\frac{32b_k\mu_k^2(y-\xi_k')_l}{[\mu_k^2+|y-\xi_k'|^2]^3} +O\left(\frac{|\partial_{\xi_{kl}'}\mu_0|}{\mu_0}\frac{\mu_0^{2+2\alpha}|y-q'|^{2\alpha}}{(\mu_0 +|y-q'|)^{4+4\alpha}}\right)\\
   & +O\left(\sum_{i=1}^m \frac{|\partial_{\xi_{kl}'}\mu_i|}{\mu_i}\frac{\mu_i^\sigma}{(\mu_i+|y-\xi_k'|)^{2+\sigma}}\right) \\
   = & -\frac{32b_k\mu_k^2(y-\xi_k')_l}{[\mu_k^2+|y-\xi_k'|^2]^3} +O\left(\varepsilon |\log\varepsilon|^\kappa \left(\frac{\mu_0^{2+2\alpha}|y-q'|^{2\alpha}}{(\mu_0 +|y-q'|)^{4+4\alpha}}+ \sum_{i=1}^m \frac{\mu_i^\sigma}{(\mu_i+|y-\xi_k'|)^{2+\sigma}} \right)\right).
\end{align*}
In the case of $|y-\xi_k'|\geq \frac1{\varepsilon |\log\varepsilon|^{2\kappa}}$, we get
\begin{equation*}
  |K_3|\leq C\varepsilon|\log\varepsilon|^{2\kappa} \left(\frac{\mu_0^{2+2\alpha}|y-q'|^{2\alpha}}{(\mu_0 +|y-q'|)^{4+4\alpha}}+\sum_{i=1}^m \frac{\mu_i^\sigma}{(\mu_i+|y-\xi_i'|)^{2+\sigma}}\right)
\end{equation*}
and
\begin{equation}\label{116}
|K_4|\leq C\varepsilon|\log\varepsilon|^{2\kappa}\left(\frac{\mu_0^{2+2\alpha}|y-q'|^{2\alpha}}{(\mu_0 +|y-q'|)^{4+4\alpha}}+\sum_{i=1}^m \frac{\mu_i^\sigma}{(\mu_i+|y-\xi_i'|)^{2+\sigma}}+\varepsilon^2\right).
\end{equation}
The estimate \eqref{116} follows from \eqref{82} and the expression of $W$ under different cases( c.f. \eqref{76}, \eqref{83} and \eqref{38}).

For $|y-\xi_k'|\leq \frac1{\varepsilon |\log\varepsilon|^{2\kappa}}$, we have
\begin{align*}
  K_4= & \frac{32b_k\mu_k^2(y-\xi_k')_l}{[\mu_k^2+|y-\xi_k'|^2]^3}\frac{|\varepsilon y-q|^{2\alpha}}{|\xi_k-q|^{2\alpha}}+O\left(\frac1{\mu_k}\left((\varepsilon \mu_k)^{\beta}+\sum_{i=0}^m(\varepsilon \mu_i)^{\frac2p-1}\right)\left(\frac{\mu_k^\sigma}{(\mu_k+|y-\xi_k|)^{2+\sigma}}+\varepsilon^2\right)\right).
\end{align*}
Using these expressions above, we get \eqref{62}.

\end{proof}

\section{Linear theory}\label{sect2}
It is well known that (see \cite{Baraket_Pacard1998}) each bounded solution of
\[\Delta \rho+\frac{8}{[1+|y|^2]^2}\rho=0\quad \mathrm{in} \quad \R^2\]
is a linear combination of the following functions
\[Z_0(y)=\frac{|y|^2-1}{|y|^2+1}, \qquad Z_1(y)=\frac{y_1}{1+|y|^2}, \qquad Z_2(y)=\frac{y_2}{1+|y|^2}.\]
It is also proved in \cite{Esposito2005} that any solution of
\[\Delta \rho+\frac{8(1+\alpha)^2|y|^{2\alpha}}{[1+|y|^{2(1+\alpha)}]^2}\rho=0\quad \mathrm{in} \quad \R^2\]
is a multiple of
\[\tilde{Z}(y)=\frac{|y|^{2(1+\alpha)}-1}{|y|^{2(1+\alpha)}+1}.\]
Choose $R_0>0$ to be a constant large enough. We define a cutoff function $\chi$ which satisfies $\chi(r)=1$ if $r\in [0, R_0)$ and $\chi(r)=0$ if $r\in [R_0+1, \infty)$.

For $i=1,\cdots,l$, we define
\[\chi_i(y)=\chi\left(\frac{|y-\xi_i'|}{\mu_i}\right), \qquad Z_{ij}(y)=\frac1{\mu_i}Z_j\left(\frac{y-\xi_i'}{\mu_i}\right), \quad \mathrm{for} \;\; j=0,1,2.\]
For $i=l+1,\cdots,m$, $\xi_i\in \partial \Omega$, we define the rotation map $A_i:\R^2\to \R^2$ so that $A_i n_{\Omega}(\xi_i)=n_{\R_+^2}(0)$.
Let $G_i(x_1)$ be a defining function defined on a neighborhood of $(0,0)$ on $A_i(\partial \Omega-\xi_i)$. So there exist constants $R>0$ and $\delta_2>0$, and a smooth function $G_i:(-R,R)\to \R$ satisfying $G_i(0)=0$, $G_i'(0)=0$ such that
\[A_i(\Omega-\{ \xi_i\})\cap B_{\delta_2}(0)=\{(x_1,x_2)\;:\;-R<x_1<R,\; x_2>G_i(x_1)\}\cap B_{\delta_2} (0,0).\]
We define the function $F_i: B_{\delta_2}(0,0)\cap \overline{A_i(\Omega-\{\xi_i\})}\to \R^2$, where $F_i=(F_{i1}, F_{i2})$ and
\[F_{i1}(x_1,x_2)=x_1+\frac{x_2-G_i(x_1)}{1+|G_i'(x_1)|^2}G_i'(x_1), \qquad F_{i2}(x_1,x_2)=x_2-G_i(x_1).\]
It is obvious that $F_i$ is  a transformation which strengthen the boundary near $\xi_i$.
For $i=l+1, \cdots,m$, we define
\[F_i^\varepsilon(y)=\frac1\varepsilon F_i(A_i(\varepsilon y-\xi_i))\;: \; B_{\frac{\delta_2}{\varepsilon}}(\xi_i')\cap \Omega_\varepsilon\to \R^2\]
and
\[\chi_i(y)=\chi\left(\frac{|F_i^\varepsilon (y)|}{\mu_i}\right),\qquad Z_{ij}(y)=\frac1{\mu_i}Z_j\left(\frac{F_i^\varepsilon (y)}{\mu_i}\right), \qquad j=0,1.\]

As the same way, if $q\in \Omega$, we define
\[\chi_0(y)=\chi\left(\frac{|y-q'|}{\mu_0}\right), \qquad Z_{00}(y)=\frac1{\mu_0}\tilde{Z}\left(\frac{y-q'}{\mu_0}\right).\]

In the case $q\in \partial \Omega$, we define a rotation map $A_0$ such that $A_0 n_{\Omega}(q)=n_{\R^2_+}(0)$. Choose $G_0:(-R,R)\to \R$ to be a defining function on some neighborhood of $0$ in $A_0(\partial \Omega-\{q\})$, which satisfies $G_0(0)=0$, $G_0'(0)=0$ and
\[A_0(\Omega-\{q\})\cap B_{\delta_2}(0,0)=\{(x_1,x_2)\;:\; -R<x_1<R, x_2>G(x_1)\}\cap B_{\delta_2}(0,0).\]
Let $F_0: B_{\delta_2}(0,0)\cap \overline{A_0(\Omega-\{q\})}\to \R^2$ be a map which strengthens the boundary near $(0,0)$, and $F_0(x)=(F_{01}, F_{02})$, where
\[F_{01}(x_1,x_2)=x_1+\frac{x_2-G_0(x_1)}{1+|G_0'(x_1)|^2}G_0'(x_1), \qquad F_{02}(x_1,x_2)=x_2-G_0(x_1).\]
Define
\[F_0^\varepsilon (y)=\frac1\varepsilon F_0(A_0(\varepsilon y-q)): B_{\frac{\delta_2}\varepsilon}(q')\cap \Omega_\varepsilon\to \R^2\]
and
\[\chi_0(y)=\chi\left(\frac{|F_0^\varepsilon (y)|}{\mu_0}\right),\qquad Z_{00}(y)=\frac1{\mu_0}Z_0\left(\frac{F_0^\varepsilon (y)}{\mu_0}\right).\]

In this section, we consider the following linearized problem
\begin{equation}\label{18}
\left\{\begin{array}{ll}
L(\phi)=h+\frac1{a(\varepsilon y)}\sum_{i=1}^m \sum_{j=1}^{J_i} c_{ij}\chi_i Z_{ij}& \mbox{in $\Omega_\varepsilon$,} \\
\frac{\partial \phi}{\partial n}=0, &\mbox{on $\partial \Omega$,} \\
\int_{\Omega_\varepsilon} \chi_i Z_{ij}\phi dy=0, \quad i=1,\cdots,m;\;j=1, J_i.
\end{array}\right.
\end{equation}
where $J_i=2$ if $i=1,2,\cdots,l$ and $J_i=1$ if $i=l+1, \cdots,m$. In Subsection \ref{subs1} and Subsection \ref{subs2}, we will derive some prior estimates of some auxiliary problems. In Subsection \ref{subs3}, we will show \eqref{18} has a unique solution and also give the prior estimate of this solution.

\subsection{Prior estimate of the first auxiliary problem}\label{subs1}

In order to study \eqref{18}, we consider the following problem first
\begin{equation}\label{19}
\left\{\begin{array}{ll}
L(\phi)=h, &\mbox{in $\Omega_\varepsilon$,} \\
\frac{\partial \phi}{\partial n}=0, &\mbox{on $\partial \Omega_\varepsilon$,}  \\
\int_{\Omega_\varepsilon}\chi_i(y)Z_{ij}(y)\phi(y)dy=0, &\mbox{$i=0,1,\cdots m;$ $j=0,1,J_i$,}
\end{array}\right.
\end{equation}
where $J_0=0$.

\begin{lemma}\label{lm6}
For $\varepsilon>0$ small enough, there exist a constant $R_1>0$ and a smooth function $g: \tilde{\Omega}_\varepsilon\to (0,+\infty)$, where \[\tilde{\Omega}_\varepsilon=\Omega_\varepsilon\char92\left[\cup_{j=1}^mB_{R_1\mu_j}(\xi_j')\cup B_{R_1\mu_0}(q')\right],\]
so that the function $g$ satisfies
\begin{equation}\label{107}
-\Delta g-\nabla(\log a(\varepsilon y))\cdot \nabla g+\varepsilon^2 g-Wg\geq \varepsilon^2+\frac{\mu_0^{2+2\hat{\alpha}}}{|y-q'|^{4+2\hat{\alpha}}}+\sum_{j=1}^m \frac{\mu_j^\sigma}{|y-\xi_j'|^{2+\sigma}}, \qquad \mathrm{in}\quad \tilde{\Omega}_\varepsilon,
\end{equation}
\begin{equation}\label{108}
\frac{\partial g}{\partial n}\geq 0, \qquad \mathrm{on} \quad \partial\Omega_\varepsilon\char92 \left[\left(\bigcup_{i=1}^m B_{R_1\mu_i}(\xi_i')\right)\bigcup B_{R_1\mu_0}(q')\right]
\end{equation}
and
\begin{equation}\label{109}
g\geq 1,\qquad \mathrm{on} \quad  \Omega_\varepsilon\bigcap\left[\left( \bigcup_{i=1}^m \partial B_{R_1\mu_i}(\xi_i')\right)\bigcup \partial B_{R_1\mu_0}(q')\right].
\end{equation}
Moreover $g$ is a uniformly bounded positive function on $\tilde{\Omega}_\varepsilon$, \textit{i.e.} there exists a constant $C$ such that $0<g\leq C$.
\end{lemma}
\begin{proof}
We define the functions
\[g_0(y)=2-\frac1{(\hat{\alpha}+1)^2}\frac{\mu_0^{2+2\hat{\alpha}}}{r_0^{2+2\hat{\alpha}}}\qquad \mathrm{and} \qquad g_i(y)=2-\frac4{\sigma^2}\frac{\mu_i^\sigma}{r_i^\sigma} \quad \mathrm{for} \quad i=1,2,\cdots,m,\]
where $r_0=|y-q'|$ and $r_i=|y-\xi_i'|$.

Denote
\[M=\sup_{x\in \Omega}|\nabla \log a(x)|.\]
It is easy to verify that the function $g_0$ satisfies the following equation:
\begin{eqnarray*}
   && -\Delta g_0-\nabla(\log a(\varepsilon y))\cdot \nabla g_0 +\varepsilon^2 g_0 \\
   &=& 4\frac{\mu_0^{2+2\hat{\alpha}}}{r_0^{4+2\hat{\alpha}}}-\frac{2\varepsilon \mu_0^{2+2\hat{\alpha}}\langle (\nabla \log a)(\varepsilon y), y-q'\rangle}{(1+\hat{\alpha})r_0^{4+2\hat{\alpha}}}+\varepsilon^2\left(2-\frac{\mu_0^{2+2\hat{\alpha}}}{(\hat{\alpha}+1)^2 r_0^{2+2\hat{\alpha}}}\right).
\end{eqnarray*}
Since
\[\left|\frac{2\varepsilon \mu_0^{2+2\hat{\alpha}}\langle (\nabla \log a)(\varepsilon y), y-q'\rangle}{(1+\hat{\alpha})r_0^{4+2\hat{\alpha}}}\right|\leq \frac{2\varepsilon \mu_0^{2+2\hat{\alpha}}M}{(\hat{\alpha}+1)r_0^{3+2\hat{\alpha}}}\leq \mu_0^{2+2\hat{\alpha}}\left(\frac2{r_0^{4+2\hat{\alpha}}}+\frac{\varepsilon^2 M^2}{2(\hat{\alpha}+1)^2r_0^{2+2\hat{\alpha}}}\right),\]
we get
\[-\Delta g_0-\nabla(\log a(\varepsilon y))\cdot \nabla g_0 +\varepsilon^2 g_0\geq \frac{2\mu_0^{2+2\hat{\alpha}}}{r_0^{4+2\hat{\alpha}}}+\varepsilon^2\left(2-\frac{\mu_0^{2+2\hat{\alpha}}(2+M^2)}{2(1+\hat{\alpha})^2 r_0^{2+2\hat{\alpha}}}\right).\]
From the same method, we also get that for $i=1,2,\cdots,m$,
\[-\Delta g_i-\nabla(\log a(\varepsilon y))\cdot \nabla g_i +\varepsilon^2 g_i\geq \frac{2 \mu_i^\sigma}{r_i^{\sigma+2}}+\varepsilon^2 \left(2- \frac{2(2+M^2)\mu_i^\sigma}{\sigma^2 r_i^\sigma}\right).\]
Combining the estimates \eqref{76}, \eqref{77} and \eqref{38}, we have
\begin{equation*}
0<W\leq C_1 \left(\frac{\mu_0^{2+2\alpha}}{r_0^{4+2\alpha}}+\sum_{j=1}^m \frac{\mu_j^2}{r_j^4}+\varepsilon^3\right).
\end{equation*}
Choose $\tilde{g}$ to be the unique solution of the following problem:
\begin{equation*}
\left\{\begin{array}{ll}
-\Delta \tilde{g}-\nabla(\log a(\varepsilon y))\cdot \nabla \tilde{g}+\varepsilon^2 \tilde{g}=2\varepsilon^2 &\mbox{in $\Omega_\varepsilon$}, \\
\frac{\partial \tilde{g}}{\partial n}=\varepsilon &\mbox{on $\partial \Omega$}.
\end{array}\right.
\end{equation*}
It is apparent that $\tilde{g}$ is a uniformly bounded positive function on $\Omega_\varepsilon$, \textit{i.e.} there exists a constant $C_0$, such that $0<\tilde{g}\leq 2C_0$.

Now we estimate the boundary terms. For $i=1,2,\cdots,l$, we have $\mathrm{dist}(\xi_i, \partial \Omega)\geq \frac1{|\log\varepsilon|^\kappa}$. Hence
\[\frac{\partial g_i}{\partial n}=\frac{4 \mu_i^\sigma}{\sigma r_i^{\sigma+2}}\langle y-\xi_i',n \rangle=O(\mu_i^\sigma (\varepsilon |\log\varepsilon|^\kappa)^{1+\sigma})\quad \mathrm{on} \quad \partial\Omega_\varepsilon.\]
For $\varepsilon>0$ small enough, we get
\[\left|\frac{\partial g_i}{\partial n}\right|<\frac{\varepsilon}{m+1}.\]
However, for $i=l+1, \cdots, m$, we have $\xi_i\in \partial \Omega$. Then there holds that
\begin{equation}\label{106}
\left|\frac{\partial g_i}{\partial n}\right|\leq C_2 \frac{\varepsilon \mu_i^\sigma}{r_i^\sigma}.
\end{equation}
In fact, in some neighborhood of $\xi_i'$, \textit{i.e.} $|y-\xi_i'|\leq \delta_2/ \varepsilon$, $A_i(\partial \Omega_\varepsilon)$ is represented by
\[y_2-\xi_{i2}'=\frac1\varepsilon G_i(\varepsilon y_1-\xi_{i1}).\]
We get
\[\frac{\partial g_i}{\partial n}=\frac{4\mu_i^\sigma}{\sigma r_i^{\sigma +2}}\langle y-\xi_i',n \rangle =\frac{4\mu_i^\sigma}{\sigma r_i^{\sigma +2}}\frac{(y_1-\xi_{i1}')G_i'(\varepsilon y_1-\xi_{i1})-\frac1\varepsilon G_i((\varepsilon y_1-\xi_{i1}))}{\sqrt{1+|G_i'(\varepsilon y_1-\xi_{i1})|^2}}.\]
Then
\[\left|\frac{\partial g_i}{\partial n}\right|\leq C\frac{\varepsilon \mu_i^\sigma}{r_i^\sigma}\]
However if $y\in \partial \Omega_\varepsilon$ satisfying $|y-\xi_i'|\geq \delta_2/ \varepsilon$, it also holds that
\[\left|\frac{\partial g_i}{\partial n}\right|\leq \frac{4\varepsilon}{\sigma\delta_2}\frac{ \mu_i^\sigma}{r_i^\sigma}.\]
Hence \eqref{106} follows.

In the case of $q\in \Omega$
\[\frac{\partial g_0}{\partial n}=\frac{2\mu_0^{2+2\hat{\alpha}}}{(1+\hat{\alpha})r_0^{4+2\hat{\alpha}}}\langle y-q', n \rangle=O(\varepsilon^{3+2\hat{\alpha}}\mu_0^{2+2\hat{\alpha}})=O(\varepsilon\cdot \varepsilon^{2\frac{1+\hat{\alpha}}{1+\alpha}}|\log\varepsilon|^C) \quad \mathrm{on} \quad \partial \Omega_\varepsilon.\]
Hence for $\varepsilon>0$ small enough, we get
\[\left|\frac{\partial g_0}{\partial n}\right|<\frac{\varepsilon}{m+1}.\]
However, in the case of $q\in \partial \Omega$
\[\left|\frac{\partial g_0}{\partial n}\right|\leq C_3\frac{\varepsilon \mu_0^{2+2\hat{\alpha}}}{r_0^{2+2\hat{\alpha}}},\]
where we get this estimate from the same method as in \eqref{106}.

Choose the constant
\begin{eqnarray*}
R_1&>&\max\left\{\left[2C_1(C_0+m+1)\right]^{\frac1{2(\alpha-\hat{\alpha})}}, [2C_1(C_0+m+1)]^{\frac1{2-\sigma}}, \left[\frac1{1+\hat{\alpha}}\right]^{\frac1{\hat{\alpha}+1}}\right.\\
&&\left.\left(\frac4{\sigma^2}\right)^{\frac1\sigma}, \left[\frac{2+M^2}{4(\hat{\alpha}+1)^2}\right]^{\frac1{2(\hat{\alpha}+1)}},\left[\frac{2+M^2}{\sigma^2}\right]^{\frac1\sigma}, [C_2(m+1)]^{\frac1\sigma}, [C_3(m+1)]^{\frac1{2+2\hat{\alpha}}}\right\},
\end{eqnarray*}
and define the function $g$ by
\[g=g_0+\sum_{i=1}^m g_i +\tilde{g}.\]
From direct computation, we see the function $g$ satisfies the estimates \eqref{107}, \eqref{108} and \eqref{109}.

\end{proof}

\begin{lemma}\label{lm2}
There exist positive constants $\varepsilon_0$ and $C$ such that for $0<\varepsilon<\varepsilon_0$, any solution of \eqref{19} satisfies the estimate
\begin{equation}\label{72}
  \|\phi\|_{L^\infty(\Omega_\varepsilon)}\leq C\|h\|_*.
\end{equation}
\end{lemma}
\begin{proof}
From Lemma \ref{lm6}, we see that the following maximum principal hold:  if $\psi\in C^2(\tilde{\Omega}_\varepsilon)$ satisfies

\[-\Delta \psi -\nabla(\log a(\varepsilon y))\cdot \nabla \psi +\varepsilon^2\psi-W\psi\geq  0 \quad \mathrm{in} \quad \Omega_\varepsilon,\]
\[\frac{\partial \psi}{\partial n}\geq 0 \quad \mathrm{on} \quad \partial\Omega_\varepsilon\char92 \left(\bigcup_{i=1}^m B_{R_1\mu_i}(\xi_i')\cup B_{R_1\mu_0}(q')\right) \]
and
\[
  \psi\geq 0, \quad \mathrm{on} \quad \Omega_\varepsilon\cap \left(\bigcup_{i=1}^m \partial B_{R_1\mu_i}(\xi_i')\cup \partial B_{R_1\mu_0}(q')\right),
\]
we get $\psi\geq 0$ in $\tilde{\Omega}_\varepsilon$.

We first derive some prior estimate for \eqref{19}. Define
\[\|\phi\|_i=\sup_{y\in \Omega_\varepsilon \char92 \tilde{\Omega}_\varepsilon}|\phi(y)|.\]
We assert that the following estimate holds:
\begin{equation}\label{71}
  \|\phi\|_{L^\infty(\Omega_\varepsilon)}\leq C\left[\|h\|_*+\|\phi\|_i\right].
\end{equation}

Let $\bar{g}=\left(\|\phi\|_i+\|h\|_*\right)g$. On $\partial \Omega_\varepsilon \char92
\bigcup_{i=1}^m B_{R_1\mu_i }(\xi_i') \bigcup B_{R_1\mu_0}(q')$, it holds that
\[\frac{\partial \bar{g}}{\partial n}=\left(\|\phi\|_i+\|h\|_*\right)\frac{\partial g}{
\partial n} \geq 0=\frac{\partial \phi}{\partial n}.\]
Hence
\[\bar{g}\geq \|\phi\|_i\geq |\phi(y)|\qquad \mathrm{on} \qquad \Omega_\varepsilon\cap \left(\bigcup_{i=1}^m \partial B_{R\mu_i}(\xi_i')\bigcup \partial B_{R\mu_0}(q')\right).\]
However in $\Omega_\varepsilon$,
\[L(\bar{g})=\left(\|\phi\|_i+\|h\|_*\right)L(g)\geq \|h\|_*\left(\varepsilon^2 +\frac{\mu_0^{2+2\hat{\alpha}}}{|y-q'|^{4+2\hat{\alpha}}}+\sum_{j=1}^m \frac{\mu_j^\sigma}{|y-\xi_j'|^{2+\sigma}}\right)\geq |h(y)|=|L(\phi(y))|.\]
Using the maximum principal above, we get $|\phi(y)|\leq \bar{g}(y)$. Hence  \eqref{71} follows.

Suppose \eqref{72} does not hold. There exists a sequence $\{\varepsilon_k\}$ convergent to $0$ such that $\|\phi_k\|_{L^\infty(\Omega_\varepsilon)}=1$, but $\|h_k\|_*\to 0$. With the help of \eqref{72}, we find a constant $\kappa>0$ such that
\[\|\phi_k\|_i\geq \kappa.\]
If $\|\phi\|_i$ is archived in $B_{R_1\mu_0}(q')$, we have $\sup_{B_{R_1\mu_0}(q')}|\phi(y)|\geq \kappa$.

In the case of $q\in \Omega$, we define $\phi_k(y)=\tilde{\phi}_k(\mu_0^{-1}(y-q'))$. It holds that
\[-\Delta \tilde{\phi}_k-\varepsilon \mu_0 \nabla(\log a)(\varepsilon \mu_0 z+q)\cdot \nabla \tilde{\phi}_k +\varepsilon^2\mu_0^2\tilde{\phi}_k-\mu_0^2 W( \mu_0 z+q')\tilde{\phi}_k=\mu_0^2 h_k(\mu_0 z+q').\]
It is apparent that
\[\left|\mu_0^2 h_k(\mu_0 z+q')\right|\leq \mu_0^2\left(\varepsilon^2+\frac{|z|^{2\alpha}}{\mu_0^2(1+|z|)^{4+2\alpha+2\hat{\alpha}}} +\sum_{j=1}^m \frac{\mu_j^\sigma}{(\mu_j+|\mu_0 z+q'-\xi_j'|)^{2+\sigma}} \right)\|h_k\|_*\]
For $\alpha>0$, we get $\mu_0^2 h_k(\mu_0 z+q')\to 0$ on compact sets. However, for $\alpha \in (-1,0)$, we get the function $\mu_0^2 h_k(\mu_0 z+q')$ converges to 0 in $L^p_{loc}$ sense, where $1<p<\min\{-\frac1{\alpha}, 2\}$. From elliptic regularity, $\tilde{\phi}_k$ converges to the solution $\tilde{\phi}$ of the following equation in $C^{0, \tau}$ sense:
\[-\Delta \tilde{\phi}-\frac{8(1+\alpha)^2|z|^{2\alpha}}{(1+|z|^{ 2(1+\alpha)})^2}\tilde{\phi}=0 \qquad \mathrm{in} \qquad \R^2.\]
And the function $\tilde\phi$ also satisfies $\int_{\R^2}\chi(y)\tilde{Z}(y)\tilde{\phi}(y)dy=0$. Hence $\tilde\phi=0$.

In the case of $q\in \partial \Omega$, we define the function $\tilde{\phi}_k$ by $\phi_k(y)=\tilde{\phi}_k\left(\frac1{\mu_0}F_0^\varepsilon(y)\right)$. Then
\[\tilde{\phi}_k(z)= \phi_k\left[\frac1\varepsilon (F_0\circ A_0)^{-1}(\varepsilon \mu_0 z)+q'\right]\qquad \mathrm{where}\qquad z\in \mathcal{U}/(\varepsilon \mu_0).\]
Here $\mathcal{U}$ is some neighborhood of $0$. In  some neighborhood  of $q'$, we choose a local coordinate system by
\[\Phi(z)=\frac1\varepsilon (F_0\circ A_0)^{-1}(\varepsilon \mu_0 z)+q'.\]
Then the Laplacian-Beltrami operator is written into the following form
\[\Delta=a_{ij}^0 (\varepsilon \mu_0 z)\partial_{ij}+\varepsilon /\mu_0 b_j^0(\varepsilon \mu_0 z)\partial_j,\]
where $a_{ij}^0(\varepsilon \mu_0 z)=\mu_0^{-2}\left(\delta^{ij}+o(1)\right)$ and $b_j^0(\varepsilon \mu_0 z)=O(1)$. Hence $\tilde{\phi}_k$ solves the following problem
\begin{align*}
   & -\mu_0^{-2}(\delta^{ij}+o(1))\partial_{ij}\tilde{\phi}_k -\varepsilon \mu_0^{-1}b_j^0(\varepsilon \mu_0 z)\partial_j \tilde{\phi}_k -\varepsilon\mu_0^{-1}(\delta^{ij}+o(1))\partial_j(\log a)((F_0\circ A_0)^{-1}(\varepsilon \mu_0 z)+q)\partial_i\tilde{\phi}_k  \\
   & +\varepsilon^2 \tilde{\phi}_k -W\left(\frac1\varepsilon (F_0\circ A_0)^{-1}(\varepsilon \mu_0 z)+q'\right)\tilde{\phi}_k=\tilde{h}_k,\qquad \mathrm{in} \qquad \mathcal{U}/(\varepsilon \mu_0)
\end{align*}
where $\tilde{h}_k(z)=h_k\left(\frac1\varepsilon (F_0\circ A_0)^{-1}(\varepsilon \mu_0 z)+q'\right)$.
Then the function $\tilde{\phi}_k$ converges in $C^{0, \tau}$ sense to the solution of
\begin{equation*}
  \left\{\begin{array}{ll}
  \Delta \tilde{\phi} -\frac{8(1+\alpha)^2|z|^{2\alpha}}{(1+|z|^{ 2(1+\alpha)})^2}\tilde{\phi}=0, &\mbox{in $\R_+^2$,} \\
  \frac{\partial \tilde{\phi}}{\partial n}=0, &\mbox{on $\partial\R_+^2$.}
  \end{array}\right.
\end{equation*}
It also hold that $\int_{\R^2_+} \chi(y)\tilde{\phi}(y)\tilde{Z}dy=0$. Then  $\tilde{\phi}=0$. We get a contradiction.

If $\|\phi\|_i$ is achieved in $B_{R\mu_j}(\xi_j)$ for some $j\in\{1,2,\cdots,m\}$. Using the same procedure, we also get a contradiction. Hence this lemma follows.

\end{proof}

\subsection{Prior estimate of the second auxiliary problem}\label{subs2}
By changing the orthogonal condition in \eqref{19}, we we investigate the  prior estimate of the following problem
\begin{equation}\label{119}
\left\{\begin{array}{ll}
L(\phi)=h, &\mbox{in $\Omega_\varepsilon$,} \\
\frac{\partial \phi}{\partial n}=0, &\mbox{on $\partial \Omega_\varepsilon$,}  \\
\int_{\Omega_\varepsilon}\chi_i(y)Z_{ij}(y)\phi(y)dy=0, &\mbox{$i=1,\cdots m;$ $j=1,J_i$,}
\end{array}\right.
\end{equation}

In order to simplify the notations, we define $z_0=y-q'$ in the case of $q\in \Omega$ and $z_0=F_0^\varepsilon(y)$ in the case of $q\in \partial \Omega$. For $j=1,2,\cdots, l$, we define $z_j=y-\xi_j$ and for $j=l+1, \cdots,m$, we define $z_j=F_j^\varepsilon(y)$.

From the definition of $F_j^\varepsilon$, we get
\[z_j =A_j (y-\xi_j')(1+O(\varepsilon|y-\xi_j'|))\qquad \mathrm{for}\qquad j=l+1, \cdots,m.\]
In the case of $q\in \partial \Omega$, we also have
\[z_0 =A_0 (y-q')(1+O(\varepsilon|y-q'|)).\]
For any smooth function $\tilde{g}(y)=g(z_j)=g(F_j^\varepsilon(y))$, $j\in \{l+1,\cdots,m\}$, we have
\begin{equation}\label{88}
\nabla_y \tilde{g}=A_j^T\cdot \nabla_{z_j} g+O(\varepsilon |y-\xi_j'||\nabla_{z_j} g|)
\end{equation}
and
\begin{equation}\label{117}
\Delta_y \tilde{g}=\Delta_{z_j}g +O(\varepsilon |y-\xi_j'||\nabla_{z_j}^2 g|)+O(\varepsilon |\nabla_{z_j} g|).
\end{equation}

In the case of $q\in \partial\Omega$, we choose the constant $d\in (0, \frac1{20}\delta_2)$ small enough, so that
\[\frac12|y-\xi_j'|\leq |z_j|\leq 2|y-\xi_j'| \qquad \mathrm{for} \qquad y\in B_{\frac{20d}\varepsilon}(\xi_j')\]
and
\[\frac12|y-q'|\leq |z_0|\leq 2|y-q'| \qquad \mathrm{for} \qquad y\in B_{\frac{20d}\varepsilon}(q').\]
However in the case of $q\in \Omega$, we only choose the constant $d>0$ so that $B_{20d}(q)\subset \Omega$.

\begin{lemma}\label{lm4}
For all $\xi\in \mathcal{O}_\varepsilon$, there exists positive constants $\varepsilon_0$ and $C$, so that for $0<\varepsilon<\varepsilon_0$, any solution $\phi$ of problem \eqref{119} satisfies the following estimate
\[\|\phi\|_{L^\infty(\Omega)}\leq C|\log\varepsilon|\|h\|_*.\]
\end{lemma}

\begin{proof} In the proof of this lemma, we only consider the case $q\in \partial \Omega$, since in the case $q\in \Omega$ this lemma can be proven  by an easier method.
Let
\[a_0^q=\frac1{\mu_0\left[H(q,q)-\frac{4(1+\alpha)}{c_0}\log(\varepsilon \mu_0 R)\right]}, \qquad a_{j0}^\xi=\frac1{\mu_j\left[H(\xi_j,\xi_j)-\frac4{c_j}\log(\varepsilon \mu_j R)\right]},\]
where $j=1,2,\cdots,m$.
We define cutoff functions $\eta_1$ and $\eta_2$ by
\begin{equation*}
\eta_1(r)=\left\{\begin{array}{ll}
1, &\mbox{$r\leq R$}; \\
0, &\mbox{$r\geq R+1$},
\end{array}\right. \qquad
\eta_2(r)=\left\{\begin{array}{ll}
1, &\mbox{$r\leq 4d$,} \\
0, & \mbox{$r\geq 8d$.}
\end{array}\right.
\end{equation*}
Here we choose the constant $R\geq R_0+1$. Define
\[\eta_{j1}^{\xi}(y)=\left\{\begin{array}{ll}
\eta_1\left(\frac1{\mu_j}|y-\xi_j'|\right), &j=1,2,\cdots, l, \\
\eta_1\left(\frac1{\mu_j}|F_j^\varepsilon(y)|\right), &j=l+1, \cdots, m,
\end{array}\right.\]
and
\[\eta_1^q(y)=\eta_1\left(\frac1{\mu_0}|F_0^\varepsilon(y)|\right), \qquad \eta_2^q(y)=\eta_2(\varepsilon |F_0^\varepsilon (y)|).\]
Let
\begin{equation}\label{23}
\tilde{Z}_{00}=\eta_1^qZ_{00}+\eta_2^q(1-\eta_1^q)\hat{Z}_{00}
\end{equation}
and
\begin{equation}\label{22}
\tilde{Z}_{j0}=\eta_{j1}^{\xi}Z_{j0}+\eta_2^q(1-\eta_{j1}^\xi)\hat{Z}_{j0},
\end{equation}
where
\[\hat{Z}_{00}=Z_{00}-\frac1{\mu_0}+a_0^qG(\varepsilon y, q)\]
and
\[\hat{Z}_{j0}=\left\{\begin{array}{ll}
\mathring{Z}_{j0}-\frac1{\mu_j}+a_{j0}^{\xi}G(\varepsilon y, \xi_j), &j=1,2, \cdots,l, \\
Z_{j0}-\frac1{\mu_j}+a_{j0}^{\xi}G(\varepsilon y, \xi_j), &j=l+1, \cdots,m.
\end{array}\right.\]
Here $\mathring{Z}_{j0}$ is the unique solution of the following problem
\begin{equation*}
\left\{\begin{array}{ll}
-\Delta \mathring{Z}_{j0}-\nabla(\log a(\varepsilon y))\cdot \nabla \mathring{Z}_{j0}+\varepsilon^2 \mathring{Z}_{j0}= -\Delta Z_{j0}-\nabla(\log a(\varepsilon y))\cdot \nabla Z_{j0}+\varepsilon^2 Z_{j0} , &\mbox{in $\Omega_\varepsilon$,} \\
\frac{\partial \mathring{Z}_{j0}}{\partial n}=0, &\mbox{on $\partial\Omega_\varepsilon$.}
\end{array}\right.
\end{equation*}
From these definitions, we see $\tilde{Z}_{j0}$'s satisfy the Neumann boundary condition on $\partial \Omega_\varepsilon$, \textit{i.e.} $\frac{\partial \tilde{Z}_{j0}}{\partial n}=0$.
Let $\tilde{\phi}$ satisfy the following identities
\begin{equation*}
\left\{\begin{aligned}
&\phi=\tilde{\phi}+\sum_{k=0}^m d_k \tilde{Z}_{k0}+\sum_{k=1}^m \sum_{s=1}^{J_k} e_{ks}\chi_k Z_{ks} \\
&\int_{\Omega_\varepsilon} \tilde{\phi}\chi_iZ_{ij}dx=0, \qquad i=0, \cdots, m; j=0, \cdots, J_i.
\end{aligned}
\right.
\end{equation*}
Since $\xi\in \mathcal{O}_\varepsilon$, we get the constants $d_k$'s and $e_{ks}$'s satisfy the following identities
\begin{equation}\label{21}
e_{ij}\int_{\Omega_\varepsilon} \chi_i^2Z_{ij}^2dx=-\sum_{k=0, k\not=i}^m d_k \int_{\Omega_\varepsilon}\chi_i Z_{ij}\tilde{Z}_{k0}dy-d_i \chi_{\{i\geq l+1\}}\int_{\Omega_\varepsilon}\chi_i Z_{i1}Z_{i0}dy, \quad i=1,2,\cdots, m
\end{equation}
and
\begin{equation}\label{80}
\sum_{k=0}^m d_k \int_{\Omega_\varepsilon}\chi_i Z_{i0}\tilde{Z}_{k0}dy=\int_{\Omega_\varepsilon}\phi \chi_i Z_{i0}dy-e_{i1}\chi_{\{i\geq l+1\}}\int_{\Omega_\varepsilon} \chi_i^2 Z_{i1}Z_{i0}dy, \quad i=0,1,2,\cdots, m.
\end{equation}
Here $\chi_{\{i\geq l+1\}}$ denote a term which equals to 1 in the case of $i\geq l+1$ and equals to $0$ otherwise.
From the definition of $F_i^\varepsilon(y)$, we have
\begin{equation}\label{79}
F_i^\varepsilon(y)=A_i(y-\xi_i')\left(1+O(\varepsilon|y-\xi_i'|)\right), \qquad \nabla F_i^\varepsilon(y)=A_i+O(\varepsilon|y-\xi_i'|).
\end{equation}
For $i\geq l+1$, we have
\[\int_{\Omega_\varepsilon}\chi_i^2 Z_{i1}Z_{i0}dy =\int_{\R_+^2}\chi^2(|z|)Z_1(z)Z_0(z)\left(1+O(\varepsilon \mu_i|z|)\right)dz=O(\varepsilon \mu_i).\]
From the same method, we get
\[\int_{\Omega_\varepsilon} \chi_i Z_{i0}Z_{i1}dy=O(\varepsilon \mu_i),\]
\[\int_{\Omega_\varepsilon} \chi_i^2 Z_{ij}^2=\left\{\begin{array}{ll}
\int_{\R^2} \chi^2(|z|)Z_1^2dz, &i=1,2,\cdots, l; \\
\int_{\R_+^2}\chi^2(|z|)Z_1^2(z)dz+O(\varepsilon\mu_i), &i=l+1, \cdots,m
\end{array}\right.\]
and
\begin{equation*}
\int_{\Omega_\varepsilon} \chi_i Z_{i0}\tilde{Z}_{i0}dy=\int_{\Omega_\varepsilon} \chi_i Z^2_{i0}dy=\left\{\begin{array}{ll}
\int_{\R_+^2} \chi^2 \tilde{Z}^2dy+O(\varepsilon \mu_0)& i=0; \\
\int_{\R^2} \chi^2 Z_0^2dy & i=1,2,\cdots, l; \\
\int_{\R_+^2} \chi^2 Z_0^2dy+O(\varepsilon \mu_i)& i=l+1, \cdots, m.
\end{array}\right.
\end{equation*}
For $i=1,2,\cdots, m; j=0,1,J_i$ and $k\not=i$, we have
\[\int_{\Omega_\varepsilon} \chi_i Z_{ij}\tilde{Z}_{k0}dy=\int_{\Omega_\varepsilon} \chi_iZ_{ij}\eta_2^q \hat{Z}_{k0}dy.\]
If $\chi_i\not=0$, we get $|z_i'|\leq (R_0+1)\mu_i$. It is obvious that for $\varepsilon>0$ small enough, it hold that
\[|y-q'|\leq |y-\xi_i'|+|q'-\xi_i'|\leq 2|z_i|+|q'-\xi_i'|\leq 2(R_0+1)\mu_i+\frac{d}\varepsilon\leq \frac{2d}\varepsilon,\]
and
\[|z_k|\geq\frac12|y-\xi_k'|\geq\frac12\left(|\xi_i'-\xi_k'|-|y-\xi_i'|\right)\geq \frac12\left(|\xi_i'-\xi_k'|-2|z_i|\right) \geq \frac1{4\varepsilon |\log\varepsilon|^\kappa}.\]
In this case, we get
\[\hat{Z}_{k0}=Z_{k0}-\frac1{\mu_k}+a_{k0}^\xi G(\varepsilon y, \xi_k)+(\mathring{Z}_{k0}-Z_{k0})\chi_{\{1\leq k\leq l\}} =O\left(\frac{\log|\log\varepsilon|}{\mu_k|\log\varepsilon|}\right).\]
So
\[\int_{\Omega_\varepsilon} \chi_i Z_{ij}\tilde{Z}_{k0}=O\left(\frac{\mu_i\log|\log\varepsilon|}{\mu_k|\log\varepsilon|}\right).\]
From \eqref{21} and the estimate above, we get
\begin{equation}\label{47}
|e_{ij}|\leq C\sum_{k=0}^m |d_k| \frac{\mu_i\log|\log\varepsilon|}{\mu_k|\log\varepsilon|}.
\end{equation}

We also get $d_k$'s are uniquely determined according to \eqref{80}. From its definition, we get $\tilde{\phi}$ solves the following problem
\begin{equation}\label{20}
\left\{\begin{array}{ll}
L(\tilde{\phi})=h-\sum_{k=0}^m d_k L(\tilde{Z}_{k0})-\sum_{k=1}^m \sum_{s=1}^{J_k} e_{ks} L(\chi_k Z_{ks}), &\mbox{in $\Omega_\varepsilon$,} \\
\frac{\partial \tilde{\phi}}{\partial n}=0, &\mbox{on $\partial \Omega_\varepsilon$,}  \\
\int_{\Omega_\varepsilon}\chi_i(y)Z_{ij}(y)\tilde{\phi}(y)dy=0, &\mbox{$i=0,1,\cdots m;$ $j=0,1,J_i$,}
\end{array}\right.
\end{equation}
From Lemma \ref{lm2}, we get
\begin{equation}\label{46}
\|\tilde{\phi}\|_{L^{\infty}(\Omega_\varepsilon)}\leq C\left[\|h\|_*+\sum_{k=0}^m |d_k| \|L(\tilde{Z}_{k0})\|_*+\sum_{k=1}^m \sum_{s=1}^{J_k} |e_{ks}|\|L(\chi_k Z_{ks})\|_*\right].
\end{equation}
Now we have to derive some estimate of the constants $d_k$'s. Multiplying the both sides of \eqref{20} by $a(\varepsilon y) \tilde{Z}_{j0}$ and integrating, we get
\begin{align*}
  \sum_{k=0}^m d_k \int_{\Omega_\varepsilon} a(\varepsilon y)L(\tilde{Z}_{k0})\tilde{Z}_{j0}dy =& \int_{\Omega_\varepsilon} a(\varepsilon y)\tilde{Z}_{j0} h dy-\int_{\Omega_\varepsilon} a(\varepsilon y)L(\tilde{Z}_{j0})\tilde{\phi}dy \\
   & -\sum_{k=1}^m \sum_{l=1}^{J_k} e_{kl}\int_{\Omega_\varepsilon} a(\varepsilon y) L(\tilde{Z}_{j0})\chi_k Z_{kl}dy.
\end{align*}
It is obvious that
\[\left|\int_{\Omega_\varepsilon} a(\varepsilon y)\tilde{Z}_{j0}hdy\right|\leq \frac{C}{\mu_j}\|h\|_*.\]
From Lemma \ref{lm3}, \eqref{47} and \eqref{46},
\begin{align*}
  \left|\int_{\Omega_\varepsilon} a(\varepsilon y)L(\tilde{Z}_{j0})\tilde{\phi}\right|\leq & C\|L(\tilde{Z}_{j0})\|_*\|\tilde{\phi}\|_{L^\infty(\Omega_\varepsilon)} \leq C\frac{\log|\log\varepsilon|}{\mu_j|\log\varepsilon|} \|\tilde{\phi}\|_{L^\infty(\Omega_\varepsilon)}\\
  \leq & C\frac{\log|\log\varepsilon|}{\mu_j|\log\varepsilon|}\left[\|h\|_*+\sum_{k=0}^m |d_k| \frac{\log|\log\varepsilon|}{\mu_k|\log\varepsilon|}\right]  \\
  \leq & C\frac{\log|\log\varepsilon|}{\mu_j|\log\varepsilon|}\|h\|_*+C\sum_{k=0}^m |d_k| \frac{\log^2|\log\varepsilon|}{\mu_j\mu_k|\log\varepsilon|^2}.
\end{align*}
Using \eqref{47} and Lemma \ref{lm3}, we get
\begin{align*}
  \left|\sum_{k=1}^m \sum_{l=1}^{J_k} e_{kl}\int_{\Omega_\varepsilon} a(\varepsilon y) L(\tilde{Z}_{j0})\chi_k Z_{kl}dy\right|\leq & C\sum_{k=1}^m\sum_{l=1}^{J_k}\frac{|e_{kl}|}{\mu_k}\|L(\tilde{Z}_{j0})\|_* \\
  \leq & C\sum_{k=1}^m \sum_{l=1}^{J_l}\frac{\log|\log\varepsilon|}{\mu_j\mu_k|\log\varepsilon|}|e_{kl}| \\
  \leq & C\sum_{k=0}^m|d_k|\frac{\log^2|\log\varepsilon|}{\mu_j\mu_k|\log\varepsilon|^2}.
\end{align*}
According to Lemma \ref{lm5} and \eqref{47},
\begin{equation}\label{48}
|d_k|\leq C\mu_k|\log\varepsilon|\|h\|_*, \qquad |e_{ks}|\leq C(\mu_k \log|\log\varepsilon|) \|h\|_*.
\end{equation}
From \eqref{46}, \eqref{48} and Lemma \ref{lm3}, we arrive at
\[\|\phi\|_{L^\infty(\Omega_\varepsilon)}\leq C\left[\|\tilde{\phi}\|_{L^\infty(\Omega)}+|\log\varepsilon|\|h\|_*\right]\leq C|\log\varepsilon|\|h\|_*.\]
Hence this lemma follows.

\end{proof}

\subsection{Solutions to the linear problem}\label{subs3}
Now we turn around to consider problem \eqref{18}. We have the following proposition.

\begin{proposition}\label{prop2}
There exist positive constants $\varepsilon$ and $C$ such that for any $(\xi_1, \cdots, \xi_m)\in \mathcal{O}_\varepsilon$, problem \eqref{18} has a unique solution $\phi:=T(h)$. Moreover, the solution $\phi$ satisfies the following estimates
\begin{equation}\label{39}
\|\phi\|_{L^\infty(\Omega_\varepsilon)}\leq C|\log\varepsilon|\|h\|_*
\end{equation}
and
\begin{equation*}
\|\partial_{\xi_{kl}'}\phi\|_{L^\infty(\Omega_\varepsilon)}\leq C\frac{|\log\varepsilon|^2}{\mu_k}\|h\|_*.
\end{equation*}
\end{proposition}

\begin{proof}
We first prove the priori estimate \eqref{39}. According to Lemma \ref{lm4}, we get
\begin{equation}\label{44}
  \|\phi\|_{L^\infty(\Omega_\varepsilon)}\leq C|\log\varepsilon|\left[\|h\|_*+\sum_{i=1}^m\sum_{j=1}^{J_i}|c_{ij}|\|\chi_i Z_{ij}\|_*\right]\leq C|\log\varepsilon|\left[\|h\|_*+\sum_{i=1}^m\sum_{j=1}^{J_i}\mu_i|c_{ij}|\right].
\end{equation}
Now we estimate the constants $c_{ij}$'s. For $k=l+1, \cdots,m$, we multiply the both sides of the first equation in \eqref{18} by $a(\varepsilon y)\eta_2^q Z_{ks}$ and integrate by part. Then we have
\begin{equation}\label{43}
  \int_{\Omega_\varepsilon} a(\varepsilon y)L(\eta_2^q Z_{ks})\phi dy=\int_{\Omega_\varepsilon} a(\varepsilon y)\eta_2^q Z_{ks}hdy+\sum_{i=1}^m \sum_{j=1}^{J_i}c_{ij}\int_{\Omega_\varepsilon} \chi_iZ_{ij}\eta_2^q Z_{ks}dy.
\end{equation}
It is apparent that
\begin{eqnarray*}
  L(\eta_2^q Z_{ks})&=&\eta_2^q\left[\left(\frac{8\mu_k^2}{[\mu_k^2+|z_k|^2]^2}-W\right)Z_{ks}-\nabla (\log a(\varepsilon y))\cdot \nabla Z_{ks}+\varepsilon^2 Z_{ks}\right] -\Delta \eta_2^q Z_{ks} -2\nabla \eta_2^q\cdot \nabla Z_{ks} \\
  && -\nabla (\log a(\varepsilon y))\cdot \nabla \eta_2^q Z_{ks}+O\left(\frac{\varepsilon |\eta_2^q|}{(\mu_k+|z_k|)^2}\right).
\end{eqnarray*}
In the equality above, we notice
\begin{equation}\label{81}
-\Delta Z_{ks}=\frac{8\mu_k^2}{[\mu_k^2+|z_k|^2]^2}Z_{ks}+O\left(\frac{ \varepsilon}{(\mu_k+|z_k|)^2}\right),
\end{equation}
which follows from relationship \eqref{117}.

If $\frac{4d}{\varepsilon}\leq |z_0|\leq \frac{8d}\varepsilon$, we get $\frac{d}{2\varepsilon}\leq |z_j|\leq \frac{34d}\varepsilon$ for $\varepsilon$ small enough. There holds that
\[\Delta \eta_2^q Z_{ks}+2\nabla \eta_2^q \cdot \nabla Z_{ks}+\nabla (\log a(\varepsilon y))\cdot \nabla \eta_2^q Z_{ks}=O(\varepsilon^3).\]
Then
\[\int_{\Omega_\varepsilon} a(\varepsilon y)\left[\Delta \eta_2^q Z_{ks}+2\nabla \eta_2^q \cdot \nabla Z_{ks}+\nabla (\log a(\varepsilon y))\cdot \nabla \eta_2^q Z_{ks}\right]\phi dy=O\left(\varepsilon \|\phi\|_{L^\infty(\Omega_\varepsilon )}\right).\]
It is easy to get
\begin{eqnarray*}
  &&\int_{\Omega_\varepsilon} a(\varepsilon y)\eta_2^q\left[-\nabla(\log a(\varepsilon y))\cdot \nabla Z_{ks}+\varepsilon^2 Z_{ks}\right]\phi \\
  & \leq& C\|\phi\|_{L^\infty(\Omega_\varepsilon)} \left(\int_{\eta_2^q\not=0}\frac{\varepsilon}{(\mu_k+|y-\xi_k'|)^2}dy+\int_{\eta_2^q\not=0} \frac{\varepsilon^2}{(\mu_k+|y-\xi_k'|)}dy\right) \\
   &\leq& C\varepsilon |\log\varepsilon| \|\phi\|_{L^\infty(\Omega)}.
\end{eqnarray*}
From the same argument, we also get
\[\left|\int_{\Omega_\varepsilon}a(\varepsilon y)\frac{\varepsilon \eta_2^q}{(\mu_k+|z_k|)^2}\phi dy\right|\leq C\varepsilon |\log\varepsilon|\|\phi\|_{L^\infty(\Omega_\varepsilon)}.\]
Denote
\[A_k=\left\{y\in \Omega_\varepsilon\;:\; |y-\xi_k'|\leq \frac1{\varepsilon|\log\varepsilon|^{2\kappa}}\right\}, \qquad \mathrm{for} \qquad k=1,2,\cdots,m,\]
\[A_0=\left\{y\in \Omega_\varepsilon\;:\; |y-q'|\leq \frac1{\varepsilon|\log\varepsilon|^{2\kappa}}\right\},\]
and $A_{m+1}=\Omega_\varepsilon\char92 \cup_{j=0}^m A_j$. Then there holds
\begin{align*}
   & \int_{\Omega_\varepsilon} \eta_2^q\left(\frac{8\mu_k^2}{[\mu_k+|z_k|^2]^2}-W\right)Z_{ks}\phi dy \\
  \leq & C\|\phi\|_{L^\infty(\Omega_\varepsilon)}\int_{\Omega_\varepsilon} \eta_2^q \left|\frac{8\mu_k^2}{[\mu_k^2+|z_k|^2]^2}-W\right||Z_{ks}|dy \\
  \leq & C\|\phi\|_{L^\infty(\Omega_\varepsilon)}\left[\int_{A_k} \eta_2^q \left|\frac{8\mu_k^2}{[\mu_k^2+|z_k|^2]^2}-W\right||Z_{ks}|dy+ \sum_{i=0,i\not=k}^m  \int_{A_i}\eta_2^q \left|\frac{8\mu_k^2}{[\mu_k^2+|z_k|^2]^2}-W\right||Z_{ks}|dy.\right. \\
   & \left.+\int_{A_{m+1}} \eta_2^q \left|\frac{8\mu_k^2}{[\mu_k^2+|z_k|^2]^2}-W\right||Z_{ks}|dy\right]=:C\|\phi\|_{L^\infty (\Omega_\varepsilon)}(I_1+I_2+I_3).
\end{align*}
For $i\not=k$, it is apparent that
\begin{eqnarray*}
  &&\int_{A_i}\eta_2^q \left|\frac{8\mu_k^2}{[\mu_k^2+|z_k|^2]^2}-W\right||Z_{ks}|dy \\
   &\leq& C \int_{A_i} \eta_2^q\left[\frac{\mu_k^2}{(\mu_k+|y-\xi_k'|)^4}+\frac{\mu_i^2}{(\mu_i+|y-\xi_i'|)^4} +\varepsilon^4|\log\varepsilon| \right]\frac1{\mu_k+|z_k|}dy \\
   &\leq& C \varepsilon|\log\varepsilon|^\kappa.
\end{eqnarray*}
Then we get $|I_2|\leq C\varepsilon |\log\varepsilon|^\kappa$. For $I_3$, we have
\begin{equation*}
  |I_3| \leq \int_{A_{m+1}}\eta_2^q\left[\frac{\mu_k^2}{(\mu_k+|y-\xi_k'|)^4}+\varepsilon^4| \log\varepsilon|^C\right]\frac1{\mu_k+|z_k|}dy \leq C\varepsilon.
\end{equation*}
For the first term $I_1$, we have
\begin{align*}
  |I_1| \leq & \int_{A_k} \eta_2^q\left|\frac{8\mu_k^2}{[\mu_k^2+|z_k|^2]^2}-W\right|\frac1{\mu_k+|z_k|}dy \\
  \leq & C\int_{A_k} \eta_2^q\left|\frac{8\mu_k^2}{[\mu_k^2+|z_k|^2]^2}-\frac{8\mu_k^2}{[\mu_k^2+|y-\xi_k'|^2]^2} \right|\frac1{\mu_k+|y-\xi_k'|}dy \\
  & +C \int_{A_k} \eta_2^q \frac{8\mu_k^2}{[\mu_k^2+|y-\xi_k'|^2]^2}\left|\frac{|\varepsilon y-q|^{2\alpha}}{|\xi_k-q|^{2\alpha}}-1\right|\frac1{\mu_k+|y-\xi_k'|}dy \\
  & +C \int_{A_k} \eta_2^q \frac{8\mu_k^2}{[\mu_k^2+|y-\xi_k'|^2]^2} \left[\varepsilon^\beta|y-\xi_k'|^\beta+ \sum_{k=0}^m (\varepsilon \mu_k)^{\frac2p-1}\right] \frac1{\mu_k+|y-\xi_k'|}dy \\
  & +C\varepsilon^3\int_{A_k}\eta_2^q \frac1{\mu_k+|y-\xi_k'|}dy  \\
  \leq & \frac{C}{\mu_k}\left[(\varepsilon\mu_k)^\beta+\sum_{i=0}^m (\varepsilon \mu_i)^{\frac2p-1}\right].
\end{align*}
As a summary,
\begin{equation}\label{41}
\int_{\Omega_\varepsilon} a(\varepsilon y)L(\eta_2^q Z_{ks})\phi dy=O\left(\frac1{\mu_k}\left((\varepsilon\mu_k)^\beta+\sum_{i=0}^m (\varepsilon \mu_i)^{\frac2p-1}\right)\|\phi\|_{L^\infty(\Omega_\varepsilon)}\right).
\end{equation}
It is easy to get
\begin{equation}\label{42}
\left|\int_{\Omega_\varepsilon}a(\varepsilon y)\eta_2^q Z_{ks}h dy\right|\leq \frac{C}{\mu_k}\|h\|_*
\end{equation}
For $i\not=k$, we also get
\[\int_{\Omega_\varepsilon} \chi_i Z_{ij}\eta_2^q Z_{ks}dy=O\left(\varepsilon \mu_i|\log\varepsilon|^\kappa\right).\]
It is easy to derive that for $i=1,2,\cdots, l$,
\[\int_{\Omega_\varepsilon} \chi_i Z_{ij}\eta_2^q Z_{is}dy=\delta_{js}\int_{\R^2} \chi Z_1^2dz,\]
and for $i=l+1, \cdots,m$,
\[\int_{\Omega_\varepsilon} \chi_i Z_{ij}\eta_2^q Z_{is}dy =\delta_{js}\int_{\R_+^2} \chi Z_1^2dy +O(\varepsilon \mu_i).\]
According to \eqref{43}, \eqref{41} and \eqref{42}, we get
\begin{equation}\label{59}
|c_{ks}|\leq C\left[\frac1{\mu_k}\|h\|_*+\frac1{\mu_k}\left((\varepsilon\mu_k)^\beta+\sum_{i=0}^m (\varepsilon \mu_i)^{\frac2p-1}\right)\|\phi\|_{L^\infty(\Omega_\varepsilon)}\right].
\end{equation}
However, in the case of $k=1,2,\cdots,l$,  \eqref{59} still holds. To prove this fact, we multiply the both sides of \eqref{18} by $a(\varepsilon y)\eta_2^q \mathring{Z}_{ks}$, where $\mathring{Z}_{ks}$ is the unique solution of
\begin{equation*}
\left\{\begin{array}{ll}
-\Delta \mathring{Z}_{ks}-\nabla(\log a(\varepsilon y))\cdot \nabla \mathring{Z}_{ks}+\varepsilon^2 \mathring{Z}_{ks}= -\Delta Z_{ks}-\nabla(\log a(\varepsilon y))\cdot \nabla Z_{ks}+\varepsilon^2 Z_{ks} , &\mbox{in $\Omega_\varepsilon$} \\
\frac{\partial \mathring{Z}_{ks}}{\partial n}=0, &\mbox{on $\partial\Omega_\varepsilon$}
\end{array}\right.
\end{equation*}
Then the estimate \eqref{59} follows from Lemma \ref{lm9} and a similar procedure as above.

Plugging \eqref{59} into \eqref{44}, we get
\[\|\phi\|_{L^\infty(\Omega_\varepsilon)}\leq C|\log\varepsilon|\left[\|h\|_*+\left((\varepsilon\mu_k)^\beta+\sum_{i=0}^m (\varepsilon \mu_i)^{\frac2p-1}\right)\|\phi\|_{L^\infty(\Omega)}\right].\]
For $\varepsilon>0$ small enough, we get the prior estimate \eqref{39}.

Now we prove the existence and uniqueness of the solution to \eqref{prop2}. Define the Hilbert space
\[H=\left\{\phi\in H^1(\Omega_\varepsilon)\;:\; \int_{\Omega_\varepsilon}\chi_i Z_{ij}\phi dy=0, \; i=1,2,\cdots,m;\; j=1, J_i\right\}\]
with the norm
\[\|\phi\|=\left[\int_{\Omega_\varepsilon }|\nabla \phi|^2+\varepsilon^2 \phi^2dy\right]^{\frac12}.\]
Then  \eqref{18} is equivalent to
\[\int_{\Omega_\varepsilon}a(\varepsilon y)\left[\nabla \phi\cdot \nabla \varphi+\varepsilon^2 \phi\varphi -W\phi\varphi\right]dy=\int_{\Omega_\varepsilon} a(\varepsilon y)h\varphi dy.\]
According to the prior estimate \eqref{39} and Fredholm alternative, we find a unique solution to \eqref{18}.

Now we consider the dependence of $T$ on $\xi_{kl}'$. The function $\psi:=\partial_{\xi_{kl}'}\phi$ solves the following problem
\begin{equation*}\left\{
\begin{array}{ll}
L(\psi)=\left(\partial_{\xi_{kl}'}W\right)\phi+\frac1{a(\varepsilon y)}\sum_{i=1}^m\sum_{j=1}^{J_i} \left(\partial_{\xi_{kl}'} c_{ij}\right)\chi_iZ_{ij} +\frac1{a(\varepsilon y)}\sum_{i=1}^m\sum_{j=1}^{J_i} c_{ij}\partial_{\xi_{kl}'}\left(\chi_iZ_{ij}\right) &\mbox{in $\Omega_\varepsilon$}, \\
\frac{\partial \psi}{\partial n}=0 &\mbox{on $\partial \Omega_\varepsilon$}, \\
\int_{\Omega_\varepsilon} \chi_i Z_{ij}\psi dy=-\int_{\Omega_\varepsilon} \partial_{\xi_{kl}'}\left(\chi_i Z_{ij}\right)\phi dy, \qquad i=1,\cdots, m; j=1,J_i.
\end{array}\right.
\end{equation*}

We define $\tilde{\psi}$ by
\[\tilde{\psi}=\psi+\sum_{t=1}^m \sum_{s=1}^{J_t} b_{ts}\eta_{t1}^\xi Z_{ts},\]
where the constants $b_{ts}$'s satisfy the following identity
\begin{equation}\label{56}
\sum_{s=1}^{J_t} b_{is}\int_{\Omega_\varepsilon} \chi_i Z_{ij} Z_{is}dy=\int_{\Omega_\varepsilon} \partial_{\xi_{kl}'}\left(\chi_i Z_{ij}\right)\phi dy.
\end{equation}
Then the function $\tilde{\psi}$ is a solution of the following problem
\begin{equation*}\left\{
\begin{array}{ll}
L(\tilde{\psi})=\left(\partial_{\xi_{kl}'}W\right)\phi+\frac1{a(\varepsilon y)}\sum_{i=1}^m\sum_{j=1}^{J_i} \left(\partial_{\xi_{kl}'} c_{ij}\right)\chi_iZ_{ij} \\
 \qquad\qquad +\sum_{t=1}^m\sum_{s=1}^{J_t} b_{ts} L(\eta_{t1}^\xi Z_{ts})+\frac1{a(\varepsilon y)}\sum_{i=1}^m\sum_{j=1}^{J_i} c_{ij}\partial_{\xi_{kl}'}\left(\chi_iZ_{ij}\right) &\mbox{in $\Omega_\varepsilon$}, \\
\frac{\partial \tilde\psi}{\partial n}=0, &\mbox{on $\partial \Omega_\varepsilon$}, \\
\int_{\Omega_\varepsilon} \chi_i Z_{ij}\tilde\psi dy=0, \qquad i=1,\cdots, m; j=1,J_i.
\end{array}\right.
\end{equation*}
We first give some estimates of the constants $b_{ts}$'s.
For $t\not=i$, we have
\begin{equation*}
  \left|\int_{\Omega_\varepsilon} \chi_i Z_{ij}Z_{ts}\right|\leq C\int_{\Omega_\varepsilon} \frac{\chi_i}{\mu_i+|y-\xi_i'|}\frac1{\mu_t+|y-\xi_t'|}dy\leq C(\varepsilon \mu_i)|\log\varepsilon|^{\kappa}.
\end{equation*}
Whereas, for $t=i$, we have
\begin{align*}
  \int_{\Omega_\varepsilon}\chi_i Z_{ij}Z_{is}dy = & \int_{|z_i|\leq R+1} \chi(|z|)Z_j(z)Z_s(z)\left(1+O(\varepsilon \mu_i |z|)\right)dz  \\
   =& \delta_{js} \int_{\R} \chi(|z|)Z_1^2(z)dz +O(\varepsilon \mu_i).
\end{align*}
Now we estimate the right hand side of \eqref{56}. Since
\begin{align}\label{60}
  \nonumber\left|\partial_{\xi_{kl}'}(\chi_iZ_{ij})\right|\leq   &  C\chi\left(\frac{|z_i|}{\mu_i}\right)\left[\frac1{\mu_i^2}\left|\partial_{\xi_{kl}'} \mu_i\right| \left|Z_j\left(\frac{z_i}{\mu_i}\right)\right|+\frac{C}{\mu_i^3}|z_i| \left|\nabla Z_j\left(\frac{z_i}{\mu_i}\right)\right|\left|\partial_{\xi_{kl}'}\mu_i\right|\right] \\
  \nonumber& +\frac{C}{\mu_i^3}\left|\chi'\left(\frac{|z_i|}{\mu_i}\right)\right|\left|\partial_{\xi_{kl}'} \mu_i\right||z_i|\left|Z_j\left(\frac{z_i}{\mu_i}\right)\right| + \frac{C}{\mu_i^2}\left|\chi'\left(\frac{|z_i|}{\mu_i}\right)\right|\left|Z_j\left(\frac{z_i}{ \mu_i}\right)\right| \left|\partial_{\xi_{kl}'} z_i\right| \\
  & + \frac{C}{\mu_i^2}\chi\left(\frac{|z_i|}{\mu_i}\right) \left|\nabla Z_j\left(\frac{z_i}{\mu_i}\right)\right| \left|\partial_{\xi_{kl}'} z_i\right|,
\end{align}
we get
\begin{equation*}
  \left|\int_{\Omega_\varepsilon}\partial_{\xi_{kl}'}\left(\chi_i Z_{ij}\right)\phi dy\right|\leq C\left|\partial_{\xi_{kl}'}\mu_i\right|\|\phi\|_{L^\infty(\Omega_\varepsilon)}\leq C\varepsilon \mu_i|\log\varepsilon|^{\kappa} \|\phi\|_{L^\infty(\Omega_\varepsilon)}, \qquad \mathrm{for} \qquad i\not=k
\end{equation*}
and
\begin{equation*}
  \left|\int_{\Omega_\varepsilon}\partial_{\xi_{kl}'}\left(\chi_k Z_{kj}\right)\phi dy\right|\leq C\|\phi\|_{L^\infty(\Omega_\varepsilon)}.
\end{equation*}
Hence
\begin{equation*}
  |b_{ks}|\leq C\|\phi\|_{L^\infty(\Omega_\varepsilon)}, \qquad \mathrm{and} \qquad |b_{is}|\leq C\varepsilon \mu_i|\log\varepsilon|^\kappa \|\phi\|_{L^\infty(\Omega_\varepsilon)} \quad \mathrm{for} \quad i\not=k.
\end{equation*}
It is easy to see
\begin{equation*}
  \partial_{\xi_{kl}'}W=\varepsilon^{4+2\alpha}|y-q'|^{2\alpha}\left(e^{U_\xi(\varepsilon y)}-e^{-U_\xi(\varepsilon y)}\right) \partial_{\xi_{kl}'} U_\xi(\varepsilon y).
\end{equation*}
From \eqref{82},  we get
\begin{equation*}
  \|(\partial_{\xi_{kl}'} W)\phi\|_*\leq C\|W\|_*\|\partial_{\xi_{kl}'}U_{\xi}(\varepsilon y)\|_{L^\infty(\Omega_\varepsilon)}\|\phi\|_{L^\infty(\Omega_\varepsilon)}\leq \frac{C}{\mu_k}\|\phi\|_{L^\infty(\Omega_\varepsilon)}.
\end{equation*}
It is easy to derive
\begin{equation*}
  L(\eta_{t1}^\xi Z_{ts})=\eta_{t1}^\xi L(Z_{ts})-\Delta \eta_{t1}^\xi Z_{ts} -2\nabla \eta_{t1}^\xi \cdot \nabla Z_{ts}-\nabla(\log a(\varepsilon y))\cdot \nabla \eta_{t1}^\xi Z_{ts}.
\end{equation*}
However
\[\left|\Delta \eta_{t1}^\xi Z_{ts} +2\nabla \eta_{t1}^\xi \cdot \nabla Z_{ts} +\nabla(\log a(\varepsilon y))\cdot \nabla \eta_{t1}^\xi Z_{ts}\right|\leq \frac{C}{\mu_t}\frac{\mu_t^\sigma}{(\mu_t+|y-\xi_t'|)^{2+\sigma}}.\]
In the case of $|z_t|\leq (R+1)\mu_t$, we get the following estimate from \eqref{81}:
\[L(Z_{ts})=\left(\frac{8\mu_t^2}{[\mu_t^2+|z_t|^2]^2}-W\right)Z_{ts}+O\left(\frac{\varepsilon}{( \mu_t+|z_t|)^2}\right)+O\left(\frac{\varepsilon^2}{ \mu_t+|z_t|}\right)\]
For the first term above, we get
\begin{align*}
  \left|\left(\frac{8\mu_t^2}{[\mu_t^2+|z_t|^2]^2}-W\right)Z_{ts}\right| \leq & \frac{C}{\mu_t+|z_t|}\left( \frac{8\mu_t^2}{[\mu_t^2+|z_t|^2]^2}-\frac{8\mu_t^2}{ [\mu_t^2+|y-\xi_t'|^2]^2}\frac{|\varepsilon y-q|^{2\alpha}}{|\xi_t-q|^{2\alpha}}\right) \\
   & +\frac{C\mu_t^2}{(\mu_t+|y-\xi_t'|)^5}\left(|\varepsilon y-\xi_t|^\beta+\sum_{k=0}^m (\varepsilon \mu_k)^{\frac2p-1}\right)+\frac{C\varepsilon^4}{\mu_t+|y-\xi_t'|}.
\end{align*}
Hence there holds
\[\|L(\eta_{t1}^\xi Z_{ts})\|_*\leq \frac{C}{\mu_t}.\]
Using this fact, we get
\begin{align*}
  \left\|\sum_{t=1}^m\sum_{s=1}^{J_t} b_{ts} L(\eta_{t1}^\xi Z_{ts})\right\|_*\leq & \sum_{s=1}^{J_k} |b_{ks}|\|L(\eta_{k1}^\xi Z_{ks})\|_* +\sum_{t\not=k}\sum_{s=1}^{J_t} |b_{ts}|\|L(\eta_{t1}^\xi Z_{ts})\|_* \\
  \leq & \frac{C}{\mu_k}\|\phi\|_{L^\infty(\Omega_\varepsilon)}+\varepsilon|\log\varepsilon|^\kappa \|\phi\|_{L^\infty(\Omega_\varepsilon)} \\
  \leq &  \frac{C}{\mu_k}\|\phi\|_{L^\infty(\Omega_\varepsilon)}.
\end{align*}
From \eqref{39} and \eqref{59}, we have $|c_{ij}|\leq \frac{C}{\mu_i}\|h\|_*$. Using this fact and  \eqref{60}, we also get
\[\left\|\sum_{i\not=k}\sum_{j=1}^{J_i}c_{ij}\partial_{\xi_{kl}'}(\chi_i Z_{ij})\right\|_*\leq C\varepsilon |\log\varepsilon|^\kappa \|h\|_*\]
and
\[\left\|\sum_{j=1}^{J_i} c_{kj}\partial_{\xi_{kl}'}(\chi_k Z_{kj})\right\|_* \leq \frac{C}{\mu_k}\|h\|_*.\]
Hence
\[\|\tilde{\psi}\|_{L^\infty(\Omega_\varepsilon)}\leq \frac{C|\log\varepsilon|^2}{\mu_k}\|h\|_*.\]
From the definition of $\tilde{\psi}$, we get
\[\|\psi\|_{L^\infty(\Omega_\varepsilon)}\leq \frac{C|\log\varepsilon|^2}{\mu_k}\|h\|_*.\]
This proposition follows.

\end{proof}

\section{Nonlinear theory}\label{sect3}
In this section we solve the following problem
\begin{equation}\label{58}
\left\{\begin{array}{ll}
L(\phi)=R +N(\phi) +\frac1{a(\varepsilon y)}\sum_{k=1}^m\sum_{l=1}^{J_k} c_{kl} \chi_k Z_{kl}& \mbox{in $\Omega_\varepsilon$,} \\
\frac{\partial \phi}{\partial n}=0, & \mbox{on $\partial \Omega$}, \\
\int_{\Omega_\varepsilon} \chi_k Z_{kl} \phi dx=0, \qquad k=1,2,\cdots, m; l=1, J_k
\end{array}\right.
\end{equation}
where $R$ and $N(\phi)$ were defined in \eqref{49} and \eqref{50}.

\begin{proposition}\label{prop3}
For $\xi=(\xi_1, \xi_2, \cdots, \xi_m)\in \mathcal{O}_\varepsilon$, there exist constants $C>0$ and $\varepsilon_0>0$ such that for $\varepsilon \in (0, \varepsilon_0)$,  problem \eqref{58} has a unique solution $\phi$. Moreover, $\phi$ satisfies the following estimates
\begin{equation}\label{67}
  \|\phi\|_{L^\infty(\Omega_\varepsilon)}\leq C|\log\varepsilon|\left[(\varepsilon \mu_0)^{\min\{\beta, 2(\alpha-\hat{\alpha}), \frac2p-1\}}+\sum_{i=1}^m (\varepsilon \mu_i)^{\min\{\beta, \frac{2}{p}-1\}}\right].
\end{equation}
and
\begin{equation}\label{65}
  \|\partial_{\xi_{kl}'}\phi\|_{L^\infty(\Omega_\varepsilon)}\leq \frac{C|\log\varepsilon|^2}{\mu_k}\left[(\varepsilon \mu_0)^{\min\{\beta, 2(\alpha-\hat{\alpha}), \frac2p-1\}}+\sum_{i=1}^m (\varepsilon \mu_i)^{\min\{\beta, \frac{2}{p}-1\}}\right].
\end{equation}
\end{proposition}
\begin{proof}
Define the operator $A(\phi):=T(R+N(\phi))$ and the set
\[\mathcal{F}=\left\{\phi\in C(\bar\Omega_\varepsilon)\;;\;\|\phi\|_{L^\infty(\Omega_\varepsilon)}\leq \theta \left[(\varepsilon \mu_0)^{\min\{\beta, 2(\alpha-\hat{\alpha}), \frac2p-1\}}+\sum_{i=1}^m (\varepsilon \mu_i)^{\min\{\beta, \frac{2}{p}-1\}}\right]|\log\varepsilon|\right\},\]
where $\theta$ is a constant large enough. We consider the problem $\phi=A(\phi)$ on $\mathcal{F}$.

Since
\[|N(\phi)|\leq C\varepsilon^4|\varepsilon y-q|^{2\alpha}\left(e^{U_\xi(\varepsilon y)}+e^{-U_\xi(\varepsilon y)}\right)|\phi|^2=CW|\phi|^2,\]
we get
\begin{equation}\label{66}
\|N(\phi)\|_*\leq C\|W\|_*\|\phi\|^2_{L^\infty(\Omega_\varepsilon)}\leq C\|\phi\|_{L^\infty(\Omega_\varepsilon)}^2.
\end{equation}
Hence for $\phi\in \mathcal{F}$
\begin{equation*}
  \|A\phi\|_{L^\infty(\Omega_\varepsilon)}\leq C|\log\varepsilon|\left(\|R\|_*+\|N(\phi)\|_*\right)\leq C\left[(\varepsilon \mu_0)^{\min\{\beta, 2(\alpha-\hat{\alpha}), \frac2p-1\}}+\sum_{i=1}^m (\varepsilon \mu_i)^{\min\{\beta, \frac{2}{p}-1\}}\right]|\log\varepsilon|.
\end{equation*}
It is apparent that for $\phi_1, \phi_2\in \mathcal{F}$, we get
\begin{align*}
  &\left|N(\phi_1)-N(\phi_2)\right| \\
  \leq & C\varepsilon^4|\varepsilon y-q|^{2\alpha}\left\{e^{U_\xi(\varepsilon y)}\left|e^{\phi_1}-e^{\phi_2}-(\phi_1-\phi_2)\right|+ e^{-U_\xi(\varepsilon y)}\left|e^{-\phi_1}-e^{-\phi_2}+(\phi_1-\phi_2)\right|\right\} \\
  \leq & CW\left[|\phi_2||\phi_1-\phi_2|+|\phi_1-\phi_2|^2\right].
\end{align*}
Hence
\begin{align*}
  \|N(\phi_1)-N(\phi_2)\|_*\leq & C\left[\|\phi_2\|_{L^\infty(\Omega_\varepsilon)}\|\phi_1-\phi_2\|_{L^\infty( \Omega_\varepsilon)}+\|\phi_1-\phi_2\|_{L^\infty(\Omega_\varepsilon)}^2\right] \\
  \leq & C\left[(\varepsilon \mu_0)^{\min\{\beta, 2(\alpha-\hat{\alpha}), \frac2p-1\}}+\sum_{i=1}^m (\varepsilon \mu_i)^{\min\{\beta, \frac{2}{p}-1\}}\right]|\log\varepsilon|\|\phi_1-\phi_2\|_{L^\infty(\Omega_\varepsilon)}.
\end{align*}
From Proposition \ref{prop2}, we get
\begin{eqnarray*}
  &&\|A(\phi_1)-A(\phi_2)\|_{L^\infty(\Omega_\varepsilon)} \\
  &\leq& C|\log\varepsilon|\|N(\phi_1)-N(\phi_2)\|_* \\
  &\leq& C\left[(\varepsilon \mu_0)^{\min\{\beta, 2(\alpha-\hat{\alpha}), \frac2p-1\}}+\sum_{i=1}^m (\varepsilon \mu_i)^{\min\{\beta, \frac{2}{p}-1\}}\right]|\log\varepsilon|^2\|\phi_1-\phi_2\|_{L^\infty(\Omega_\varepsilon)}.
\end{eqnarray*}
According to Banach fixed point theorem, the problem $\phi=A(\phi)$ has a unique solution $\phi\in \mathcal{F}$ for $\varepsilon>0$ small enough. Moreover, the solution $\phi$ satisfies
\[\|\phi\|_{L^\infty(\Omega_\varepsilon)}\leq C\left[\|R\|_*+\|N(\phi)\|_*\right]\leq C\left[(\varepsilon \mu_0)^{\min\{\beta, 2(\alpha-\hat{\alpha}), \frac2p-1\}}+\sum_{i=1}^m (\varepsilon \mu_i)^{\min\{\beta, \frac{2}{p}-1\}}\right]|\log\varepsilon|.\]
Since $R$ is continuously dependent on $\xi_i$, $i=1, 2,\cdots,m$,  $\phi$ is $C^1$-continuous on $\xi_i$, $i=1, 2,\cdots,m$. Then $\partial_{\xi_{kl}'} \phi$ is a solution of
\begin{equation}\label{63}
\partial_{\xi_{kl}'}\phi=(\partial_{\xi_{kl}'}T)[R+N(\phi)]+T[\partial_{\xi_{kl}'} R+\partial_{\xi_{kl}'} N(\phi)].
\end{equation}
It is apparent
\begin{align*}
  \partial_{\xi_{kl}'} N(\phi)= & \varepsilon^{4+2\alpha}|y-q'|^{2\alpha}\left[\left(e^{U_\xi(\varepsilon y)+\phi}-e^{U_\xi(\varepsilon y)}-e^{U_\xi(\varepsilon y)}\phi\right)\right. \\
  &+\left. \left(e^{-U_\xi(\varepsilon y)-\phi}-e^{-U_\xi(\varepsilon y)}+e^{-U_\xi(\varepsilon y)}\phi\right)\right]\partial_{\xi_{kl}'} U_\xi(\varepsilon y) \\
   & +\varepsilon^{4+2\alpha}|y-q'|^{2\alpha}\left[\left(e^\phi-1\right)e^{U_\xi(\varepsilon y)} +\left(e^{-\phi}-1\right)e^{-U_\xi(\varepsilon y)}\right] \partial_{\xi_{kl}'}\phi.
\end{align*}
Then we get
\begin{equation*}
  |\partial_{\xi_{kl}'}N(\phi)|\leq C|W| \left[|\phi|^2 |\partial_{\xi_{kl}'}U_\xi(\varepsilon y)|+|\phi||\partial_{\xi_{kl}'}\phi|\right].
\end{equation*}
So
\begin{equation}\label{64}
\|\partial_{\xi_{kl}'}N(\phi)\|_*\leq C\|W\|_*\left[\frac1{\mu_k}\|\phi\|_{L^\infty(\Omega_\varepsilon)}^2+\|\phi\|_{L^\infty( \Omega_\varepsilon)}\|\partial_{\xi_{kl}'}\phi\|_{L^\infty( \Omega_\varepsilon)}\right].
\end{equation}
From \eqref{63} and Proposition \ref{prop2}, we get
\[\|\partial_{\xi_{kl}'}\phi\|_{L^\infty(\Omega_\varepsilon)}\leq \frac{C}{\mu_k}|\log\varepsilon|^2 \left[\|R\|_*+\|N(\phi)\|_*\right]+C|\log\varepsilon|\left[\|\partial_{\xi_{kl}'}R\|_* +\|\partial_{\xi_{kl}'} N(\phi)\|_*\right].\]
From \eqref{54}, \eqref{62}, \eqref{67}, \eqref{66} and \eqref{64}, we get \eqref{65}.

\end{proof}

\section{Variational reduction}\label{sect5}
According to Proposition \ref{prop3}, we recognize that if  all the constants $c_{kl}$'s in \eqref{58} equal to $0$, we will solve \eqref{124}.

Define the functional $F(\xi)$ on $\mathcal{O}_\varepsilon$ by
\[F(\xi)=J\left[V\left(\frac{x}{\varepsilon}\right)+ \phi\left(\frac{x}{\varepsilon}\right)-4\log\varepsilon\right].\]
Recall that the functional $J(u)$ is defined in \eqref{6}. We also define the functional
\[I_\varepsilon(v)=\frac12\int_{\Omega_\varepsilon}a(\varepsilon y)\left[|\nabla v|^2+\varepsilon^2v^2\right]dy-\varepsilon^4\int_{\Omega_\varepsilon}a(\varepsilon y)|\varepsilon y-q|^{2\alpha}\left(e^v+e^{-v}\right)dy.\]
It is apparent that
\[J\left[V\left(\frac{x}{\varepsilon}\right)+ \phi\left(\frac{x}{\varepsilon}\right)-4\log\varepsilon\right]=I_\varepsilon(V(y)+\phi(y) -4\log\varepsilon)\]
and
\[J\left[V\left(\frac{x}{\varepsilon}\right)+ \phi\left(\frac{x}{\varepsilon}\right)-4\log\varepsilon\right]-J\left[V\left(\frac{x}{ \varepsilon}\right)-4\log\varepsilon\right]=I_\varepsilon(V(y)+\phi(y) -4\log\varepsilon)-I_\varepsilon(V(y)-4\log\varepsilon).\]
\begin{proposition}\label{prop4}
$F: \mathcal{O}_\varepsilon\to \R$ is a $C^1$ functional. For $\varepsilon>0$ small enough, $DF(\xi)=0$ is equivalent to $c_{ij}(\xi')=0$ for $i=1,2,\cdots,m;$ $j=1, J_i$.
\end{proposition}
\begin{proof}
If $D_\xi F(\xi)=0$, we get $D_\xi I_\varepsilon(U_\xi(\varepsilon y)+\phi(y))=0$. That is
\[I_\varepsilon'[U_\xi(\varepsilon y)+\phi(y)][\partial_{\xi_{kl}'}V(y)+ \partial_{\xi_{kl}'}\phi]=0.\]
From Proposition \ref{prop3},  we notice that the equality above is equivalent to
\begin{equation*}
  \sum_{i=1}^m\sum_{j=1}^{J_i} c_{ij}\left[\int_{\Omega_\varepsilon} \chi_i Z_{ij}\partial_{\xi_{kl}'}V+\int_{\Omega_\varepsilon} \chi_i Z_{ij}\partial_{\xi_{kl}'}\phi\right]=0.
\end{equation*}
Now we calculate each term in the bracket. It is easy to see
\begin{eqnarray*}
  \left|\int_{\Omega_\varepsilon} \chi_i Z_{ij} \partial_{\xi_{kl}'} \phi dy \right| &\leq& \|\partial_{\xi_{kl}'}\phi\|_{L^\infty(\Omega_\varepsilon)}\int_{\Omega_\varepsilon}\chi_i|Z_{ij}|dy \leq C\mu_i\|\partial_{\xi_{kl}'}\phi\|_{L^\infty(\Omega_\varepsilon)} \\
  &\leq& C\frac{\mu_i}{\mu_k} \left[(\varepsilon \mu_0)^{\min\{\beta, 2(\alpha-\hat{\alpha}), \frac2p-1\}}+\sum_{i=1}^m (\varepsilon \mu_i)^{\min\{\beta, \frac{2}{p}-1\}}\right].
\end{eqnarray*}
From \eqref{82}, we get
\begin{equation*}
  \int_{\Omega_\varepsilon} \chi_i Z_{ij}\partial_{\xi_{kl}'} Vdy=\int_{\Omega_\varepsilon} \chi_i Z_{ij}\left[\frac{4b_k}{\mu_k}Z_l\left(\frac{y-\xi_k'}{\mu_k}\right)+O\left( \frac1{\mu_k}(\varepsilon \mu_k)^{\frac2p-1} \right)\right]
\end{equation*}
In the case of $k\not=i$, we have
\begin{equation*}
  \int_{\Omega_\varepsilon} \chi_i Z_{ij}\frac{4b_k}{\mu_k} Z_l\left(\frac{y-\xi_k'}{\mu_k}\right)= O\left( \varepsilon\mu_i|\log\varepsilon|^\kappa \right).
\end{equation*}
Whereas, in the case of $k=i$
\begin{align*}
  \int_{\Omega_\varepsilon} \chi_i Z_{ij}\frac{4b_k}{\mu_k} Z_l\left(\frac{y-\xi_k'}{\mu_k} \right)dy= & \frac{4b_i}{\mu_i^2}\int_{\Omega_\varepsilon} \chi\left(\frac{z_i}{\mu_i}\right) Z_j\left(\frac{z_i}{\mu_i}\right)\left[Z_l\left(\frac{z_i}{\mu_i}\right)+O(\varepsilon \mu_i)\right]dy \\
  = & \frac{4b_i}{\mu_i^2}\int_{\Omega_\varepsilon} \chi\left(\frac{z_i}{\mu_i}\right)Z_j\left(\frac{z_i}{\mu_i}\right)Z_l\left( \frac{z_i}{\mu_i}\right)dy +O(\varepsilon\mu_i) \\
  =& 4b_i\delta_{jl}\int_{\R_{\xi_i}^2} \chi(t) Z_1^2(t)dt+O(\varepsilon \mu_i).
\end{align*}
Then we get
\[\int_{\Omega_\varepsilon} \chi_i Z_{ij}\partial_{\xi_{kl}'}V dy=4b_i\delta_{jl}\delta_{ik}\int_{\R_{\xi_{i}}}\chi(t) Z_1^2(t)dt+o(1).\]
Hence we get $c_{ij}(\xi)=0$ for $\varepsilon>0$ small enough.

\end{proof}
According to Proposition \ref{prop4},  we notice that if we  find a critical point of $F$ in $\mathcal{O}_\varepsilon$, we will solve \eqref{124}.

\begin{lemma}\label{lm8}
For $\varepsilon>0$ small enough, we get
\begin{eqnarray}\label{68}
\nonumber&& I_{\varepsilon} (V-4\log\varepsilon+\phi) \\
&=&I_\varepsilon (V-4\log\varepsilon) +O\left(\left[(\varepsilon \mu_0)^{\min\{\beta, 2(\alpha-\hat{\alpha}), \frac2p-1\}}+\sum_{i=1}^m (\varepsilon \mu_i)^{\min\{\beta, \frac{2}{p}-1\}}\right]^2|\log\varepsilon|\right)
\end{eqnarray}
and
\begin{eqnarray}\label{70}
\nonumber&&\partial_{\xi_{kl}'} I_\varepsilon(V-4\log\varepsilon+\phi) \\
&=& \partial_{\xi_{kl}'} I_\varepsilon(V-4\log\varepsilon) +O\left(\frac1{\mu_k}\left[(\varepsilon \mu_0)^{\min\{\beta, 2(\alpha-\hat{\alpha}), \frac2p-1\}}+\sum_{i=1}^m (\varepsilon \mu_i)^{\min\{\beta, \frac{2}{p}-1\}}\right]^2|\log\varepsilon|^2\right).
\end{eqnarray}
\end{lemma}

\begin{proof}
From Proposition \ref{prop3}, we get
\[I'_\varepsilon(V-4\log\varepsilon+\phi)\phi=\int_{\Omega_\varepsilon}a(\varepsilon y)\left[L(\phi)-R-N(\phi) \right]\phi dy=\sum_{k=1}^m\sum_{l=1}^{J_k}c_{kl}\int_{\Omega_\varepsilon} \chi_k Z_{kl}\phi dy=0.\]
Then
\begin{align}\label{69}
  \nonumber &I_\varepsilon(V-4\log\varepsilon +\phi)-I_\varepsilon(V-4\log\varepsilon) \\
  \nonumber = -& \int_0^1tI_\varepsilon''(V-4\log\varepsilon+t\phi)[\phi, \phi]dt \\
  \nonumber = -& \int_0^1tdt\int_{\Omega_\varepsilon} a(\varepsilon y)\left[-\Delta \phi-\nabla(\log a(\varepsilon y))\cdot \nabla \phi +\varepsilon^2 \phi-|\varepsilon y-q|^{2\alpha}\left(e^{V+t\phi}+\varepsilon^8 e^{-V-t\phi}\right)\phi\right]\phi dy \\
  \nonumber=& -\int_0^1tdt\int_{\Omega_\varepsilon} a(\varepsilon y)\left\{R+N(\phi)-|\varepsilon y-q|^{2\alpha}\left[e^{V+t\phi}-e^V+\varepsilon^8\left(e^{-V-t\phi}-e^{-V}\right)\right]\phi \right\}\phi dy \\
  \nonumber=& -\int_0^1 tdt\int_{\Omega_\varepsilon} a(\varepsilon y)\left[R+N(\phi)\right]\phi dy \\
  &+ \int_0^1 tdt \int_{\Omega_\varepsilon} a(\varepsilon y)|\varepsilon y-q|^{2\alpha}\left[(e^{V+t\phi}-e^V)+\varepsilon^8 (e^{-V-t\phi}-e^{-V})\right]\phi^2dy.
\end{align}
It is apparent that
\begin{align*}
   \left|\int_0^1 tdt\int_{\Omega_\varepsilon} a(\varepsilon y)\left[R+N(\phi)\right]\phi dy\right| \leq & C\|\phi\|_{L^\infty(\Omega_\varepsilon)}\left[\|R\|_*+\|N(\phi)\|_*\right] \\
   \leq & C\|\phi\|_{L^\infty(\Omega_\varepsilon)}\left[\|R\|_*+\|\phi\|_{L^\infty( \Omega_\varepsilon)}^2\right] \\
   \leq & C\left[(\varepsilon \mu_0)^{\min\{\beta, 2(\alpha-\hat{\alpha}), \frac2p-1\}}+\sum_{i=1}^m (\varepsilon \mu_i)^{\min\{\beta, \frac{2}{p}-1\}}\right]^2|\log\varepsilon|
\end{align*}
and
\begin{align*}
   & \int_0^1tdt \int_{\Omega_\varepsilon} a(\varepsilon y) |\varepsilon y-q|^{2\alpha}\left[(e^{V+t\phi}-e^V)+\varepsilon^8(e^{-V-t\phi}-e^{-V})\right]\phi^2dy \\
  \leq & C\int_{\Omega_\varepsilon} W |\phi|^3dy \leq C\|\phi\|_{L^\infty(\Omega_\varepsilon)}^3\leq C\left[\left((\varepsilon \mu_0)^{\min\{\beta, 2(\alpha-\hat{\alpha}), \frac2p-1\}}+\sum_{i=1}^m (\varepsilon \mu_i)^{\min\{\beta, \frac{2}{p}-1\}}\right)|\log\varepsilon|\right]^3.
\end{align*}
Hence we get \eqref{68}.

Taking derivatives on $\xi_{kl}'$ on the both sides of \eqref{69}, we get
\begin{align*}
   & \partial_{\xi_{kl}'}\left[I_\varepsilon(V-4\log\varepsilon +\phi)-I_\varepsilon(V-4\log\varepsilon)\right] \\
  = & -\frac12\int_{\Omega_\varepsilon} a(\varepsilon y)\left[\partial_{\xi_{kl}'} R+\partial_{\xi_{kl}'} N(\phi)\right]\phi dy -\frac12 \int_{\Omega_\varepsilon} a(\varepsilon y) \left[R+N(\phi)\right]\partial_{\xi_{kl}'}\phi dy \\
  & +2\int_0^1tdt \int_{\Omega_\varepsilon} a(\varepsilon y)|\varepsilon y-q|^{2\alpha}\left[(e^{V+t\phi}-e^V)+\varepsilon^8(e^{-V-t\phi}-e^{-V})\right]\phi \partial_{\xi_{kl}'}\phi dy \\
  & +\int_0^1tdt \int_{\Omega_\varepsilon} a(\varepsilon y)|\varepsilon y-q|^{2\alpha}\left[(e^{V+t\phi}-e^V)-\varepsilon^8(e^{-V-t\phi}-e^{-V})\right]\phi^2 \partial_{\xi_{kl}'}Vdy \\
  & +\int_0^1t^2dt\int_{\Omega_\varepsilon} a(\varepsilon y)|\varepsilon y-q|^{2\alpha}\left[e^{V+t\phi}-\varepsilon^8 e^{-V-t\phi}\right]\phi^2\partial_{\xi_{kl}'}\phi dy.
\end{align*}
It is easy to get
\begin{eqnarray*}
  &&\left|\int_{\Omega_\varepsilon} a(\varepsilon y)\left[\partial_{\xi_{kl}'}R+ \partial_{\xi_{kl}'}N(\phi)\right]\phi dy\right|\leq C\|\phi\|_{L^\infty(\Omega_\varepsilon)} \left[\|\partial_{\xi_{kl}'}R\|_*+\|\partial_{\xi_{kl}'}N(\phi)\|_*\right] \\
  &\leq & C\|\phi\|_{L^\infty(\Omega_\varepsilon)} \left[\|\partial_{\xi_{kl}'}R\|_*+\frac1{\mu_k}\|\phi\|_{L^\infty(\Omega_\varepsilon)}^2 +\|\phi\|_{L^\infty(\Omega_\varepsilon)}\|\partial_{\xi_{kl}'}\phi\|_{L^\infty(\Omega_\varepsilon)}\right] \\
  &\leq& \frac{C}{\mu_k}\left((\varepsilon \mu_0)^{\min\{\beta, 2(\alpha-\hat{\alpha}), \frac2p-1\}}+\sum_{i=1}^m (\varepsilon \mu_i)^{\min\{\beta, \frac{2}{p}-1\}}\right)^2|\log\varepsilon|
\end{eqnarray*}
and
\begin{eqnarray*}
  &&\left|\int_{\Omega_\varepsilon} a(\varepsilon y)\left[R+N(\phi)\right]\partial_{\xi_{kl}'}\phi dy\right|\leq C\|\partial_{\xi_{kl}'}\phi\|_{L^\infty(\Omega_\varepsilon)}\left[\|R\|_*+\|N(\phi)\|_*\right] \\
  &\leq& \frac{C}{\mu_k}\left((\varepsilon \mu_0)^{\min\{\beta, 2(\alpha-\hat{\alpha}), \frac2p-1\}}+\sum_{i=1}^m (\varepsilon \mu_i)^{\min\{\beta, \frac{2}{p}-1\}}\right)^2|\log\varepsilon|^2.
\end{eqnarray*}
As the same way, we have
\begin{align*}
   &  \left|\int_0^1tdt \int_{\Omega_\varepsilon} a(\varepsilon y)|\varepsilon y-q|^{2\alpha}\left[(e^{V+t\phi}-e^V)+\varepsilon^8(e^{-V-t\phi}-e^{-V})\right]\phi \partial_{\xi_{kl}'}\phi dy\right|\\
  \leq & C\|W\|_*\|\phi\|_{L^\infty(\Omega_\varepsilon)}^2\|\partial_{\xi_{kl}'}\phi\|_{L^\infty( \Omega_\varepsilon)}\leq \frac{C}{\mu_k}\left[(\varepsilon \mu_0)^{\min\{\beta, 2(\alpha-\hat{\alpha}), \frac2p-1\}}+\sum_{i=1}^m (\varepsilon \mu_i)^{\min\{\beta, \frac{2}{p}-1\}}\right]^3|\log\varepsilon|^4,
\end{align*}
\begin{align*}
   & \left|\int_0^1tdt \int_{\Omega_\varepsilon} a(\varepsilon y)|\varepsilon y-q|^{2\alpha}\left[(e^{V+t\phi}-e^V)-\varepsilon^8(e^{-V-t\phi}-e^{-V})\right]\phi^2 \partial_{\xi_{kl}'}Vdy\right| \\
  \leq & C\|W\|_*\|\phi\|_{L^\infty(\Omega_\varepsilon)}^3\|\partial_{\xi_{kl}'}V\|_{L^\infty( \Omega_\varepsilon)}\leq \frac{C}{\mu_k}\left[(\varepsilon \mu_0)^{\min\{\beta, 2(\alpha-\hat{\alpha}), \frac2p-1\}}+\sum_{i=1}^m (\varepsilon \mu_i)^{\min\{\beta, \frac{2}{p}-1\}}\right]^3|\log\varepsilon|^3
\end{align*}
and
\begin{align*}
   & \left|\int_0^1t^2dt\int_{\Omega_\varepsilon} a(\varepsilon y)|\varepsilon y-q|^{2\alpha}\left[e^{V+t\phi}-\varepsilon^8 e^{-V-t\phi}\right]\phi^2\partial_{\xi_{kl}'}\phi dy\right| \\
  \leq & C\|W\|_*\|\phi\|_{L^\infty(\Omega_\varepsilon)}^2\|\partial_{\xi_{kl}'}\phi\|_{L^\infty( \Omega_\varepsilon)}\leq \frac{C}{\mu_k}\left[(\varepsilon \mu_0)^{\min\{\beta, 2(\alpha-\hat{\alpha}), \frac2p-1\}}+\sum_{i=1}^m (\varepsilon \mu_i)^{\min\{\beta, \frac{2}{p}-1\}}\right]^3|\log\varepsilon|^4.
\end{align*}
Hence \eqref{70} follows.

\end{proof}

\section{Proof of our main theorem}\label{sect6}
From the Proposition \ref{prop1} and Lemma \ref{lm8}, we get
\begin{align}\label{94}
\nonumber  F(\xi)= &-\frac12 c_0 a(q)\left[4\log \varepsilon+4-2\log8(1+\alpha)^2+c_0H(q,q)+\sum_{i=1}^m c_i G(q, \xi_i)\right] \\
\nonumber& -\frac12\sum_{i=1}^m c_ia(\xi_i)\left[4\log\varepsilon +4-2\log 8+c_i H(\xi_i, \xi_i)+c_0 G(\xi_i, q)+\sum_{k=1, k\not=i}^m c_k G(\xi_i, \xi_k)\right.\\
&\left.+4\alpha \log |\xi_i-q|\right]+O\left((\varepsilon \mu_0)^{\min\{\beta, 2(\alpha-\hat{\alpha}), \frac2p-1\}}\right)+O\left(\sum_{i=1}^m (\varepsilon \mu_i)^{\min\{\beta, \frac{2}{p}-1\}}\right).
\end{align}

\noindent\textit{Proof of Theorem \ref{th1}:}
We consider the following maximization problem
\[\max_{\xi\in \bar{\mathcal{O}}_\varepsilon} F(\xi).\]
Let $\xi^\varepsilon=(\xi^\varepsilon_1, \xi^\varepsilon_2, \cdots, \xi^\varepsilon_m)$ be the maximizer of the problem above. We will prove the maximizer belongs to $\mathcal{O}_\varepsilon$. Then we will get a critical point of $F(\xi)$ in $\mathcal{O}_\varepsilon$.

Let
\[\xi_i^0=q+\frac1{|\log\varepsilon|}\hat{\xi}_i, \qquad i=1,2,\cdots,m,\]
where the vectors $\hat{\xi}_i$'s form a $m$-regular polygon in $\R^2$. Since $\kappa>1$, we get $\xi^0:=(\xi_1^0, \xi_2^0, \cdots, \xi_m^0)\in \mathcal{O}_\varepsilon$. Then we get
\begin{equation}\label{91}
\max_{\xi\in \bar{\mathcal{O}}_\varepsilon} F(\xi)\geq F(\xi^0)=16\pi a(q)\left[(1+\alpha+m)|\log\varepsilon|-m(m+1+\alpha)\log|\log\varepsilon|\right]+O(1).
\end{equation}
To prove $\xi^\varepsilon\in\mathcal{O}_\varepsilon$, we argue indirectly. Suppose $\xi^\varepsilon\in\partial\mathcal{O}_\varepsilon$, the following cases occur:
\begin{description}
  \item[(C1)] There exists $i_0\in \{1,2,\cdots,m\}$ such that $\xi_{i_0}\in\partial B_d(q)$;
  \item[(C2)] There exist $i_0, j_0\in \{1,2,\cdots,m\}$ such that $|\xi_{i_0}-\xi_{j_0}|=\frac1{|\log\varepsilon|^\kappa}$.
  \item[(C3)] There exists $k_0\in \{1,2,\cdots,m\}$ such that $|\xi_{k_0}-q|=\frac1{|\log\varepsilon|^\kappa}$
\end{description}

In the case (C1), there exists constant $d_0>0$, such that $a(\xi_{i_0})<a(q)-d_0$. Then we have
\begin{equation}\label{92}
\max_{\xi\in \bar{\mathcal{O}}_\varepsilon} F(\xi)=F(\xi^\varepsilon)<16\pi a(q)(1+m+\alpha-d_0)|\log\varepsilon|+O(\log|\log\varepsilon|).
\end{equation}
We will get a contradiction for $\varepsilon>0$ small enough according to \eqref{91}.

In the case of (C2), we get
\begin{equation}\label{93}
\max_{\xi\in \bar{\mathcal{O}}_\varepsilon} F(\xi)=F(\xi^\varepsilon)<16\pi a(q)(1+\alpha+m)|\log\varepsilon|-32\pi \kappa a(q)\log|\log\varepsilon|+O(1).
\end{equation}
From the definition of $\kappa$, we see \eqref{93} contradict to \eqref{91}.

In the case of (C3), we get
\begin{equation}\label{95}
\max_{\xi\in \bar{\mathcal{O}}_\varepsilon} F(\xi)=F(\xi^\varepsilon)<16\pi a(q)(1+\alpha+m)|\log\varepsilon|-16\pi \kappa a(q)\log|\log\varepsilon|+O(1).
\end{equation}
It also contradict to \eqref{91}. Hence $\xi^\varepsilon \in \mathcal{O}_\varepsilon$. Then $DF(\xi^\varepsilon)=0$. According to Proposition \ref{prop4}, $c_{kl}(\xi^\varepsilon)=0$.
Then Theorem \ref{th1} follows.

$\hfill\Box$

\noindent\textit{Proof of Theorem \ref{th2}:}
From Green formula, we get $a(x)G(x,y)=a(y)G(y,x)$ for $x,y\in \bar{\Omega}$. Hence  $a(q)G(q, \xi_i)=a(\xi_i)G(\xi_i, q)$, $a(\xi_j)G(\xi_j, \xi_i)=a(\xi_i)G(\xi_i, \xi_j)$. From these fact and Lemmata in \cite[Appendix A]{Zhang2022}, we get
\begin{eqnarray}\label{96}
\nonumber F(\xi)&=&8\pi(1+\alpha)a(q)\left[-\log\varepsilon -4\pi \sum_{i=1}^l G(q, \xi_i)-2\pi \sum_{i=l+1}^m G(q, \xi_i)\right] \\
\nonumber&&+ 16\pi \sum_{i=1}^l a(\xi_i) \left[-\log\varepsilon -2\pi H(\xi_i, \xi_i)-2\pi \sum_{k=1, k\not=i}^l G(\xi_i, \xi_k)-\alpha \log|\xi_i -q|\right] \\
\nonumber&&+8\pi \sum_{i=l+1}^m a(\xi_i) \left[-\log\varepsilon -\pi \sum_{k=l+1, k\not=i}^m G(\xi_i, \xi_k)-\alpha \log|\xi_i-q|\right] \\
&& -32\pi^2 \sum_{i=1}^l\sum_{k=l+1}^m a(\xi_k)G(\xi_k, \xi_i)+O(1).
\end{eqnarray}
In order to prove this theorem, we consider the following maximum problem and let $\xi^\varepsilon=(\xi_1^\varepsilon, \xi_2^\varepsilon, \cdots, \xi_m^\varepsilon)$ be its maximizer:
\[\max_{(\xi_1, \xi_2, \cdots, \xi_m)\in \bar{\mathcal{O}}_\varepsilon} F(\xi).\]
Assume $H_q: B_d(q)\to \mathcal{M}$ be a diffeomorphism, where $\mathcal{M}$ is some neighborhood of $0$ in $\R^2$. It satisfies $H_q(B_d(q)\cap \Omega)=\mathcal{M}\cap \R^2_+$, and $H_q(B_d(q)\cap \partial\Omega)=\mathcal{M}\cap \partial\R^2_+$. Let $n(q)$ be the outward unit normal vector of $\Omega$ at $q$. We define
\[\xi_i^0=q-\frac{t_i}{\sqrt{|\log\varepsilon|}}n(q), \qquad \mathrm{where}\qquad t_i>0 \quad \mathrm{and} \quad t_{i+1}-t_i=\tilde{\sigma}, \qquad i=1,2,\cdots,l.\]
and
\[\xi_i^0=H_q^{-1}\left(\frac1{\sqrt{|\log\varepsilon|}}\hat{\xi}_i^0\right),\qquad \mathrm{where}\qquad \hat{\xi}_i\in \mathcal{M}\cap \R^2_+ \quad \mathrm{and} \quad |\hat{\xi}_i^0-\hat{\xi}_{i+1}^0|=\tilde{\sigma},  \qquad i=l+1,\cdots,m,\]
where $\tilde{\sigma}>0$ is a fixed constant.
From the definition of $H_q$, we have
\[\xi_i^0=q+\frac1{\sqrt{|\log\varepsilon|}}\hat{\xi}_i^0 +O\left(\frac1{|\log\varepsilon|}|\hat{\xi}_i^0|\right), \qquad \mathrm{for}\qquad i=l+1,\cdots,m.\]
It is obvious that $\xi^0\in \mathcal{O}_\varepsilon$. Since $q\in \partial \Omega$ is a strictly local maximum point of $a(x)$ satisfying $\partial_n a(q)=0$. Then we get
\[a(q)> a(\xi_i^0)>a(q)-\frac{C}{|\log\varepsilon|}, \qquad \mathrm{for}\qquad i=1,\cdots,m.\]
From Lemmata in \cite[Appendix A]{Zhang2022}, we have
\[H(\xi_i^0,\xi_i^0)=\frac1{4\pi}\log|\log\varepsilon|+O(1), \qquad \mathrm{for}\qquad i=1,\cdots,l;\]
\[G(q, \xi_i^0)=\frac1{2\pi}\log|\log\varepsilon|+O(1),  \qquad \mathrm{for}\qquad i=1,\cdots,l;\]
and
\[G(\xi_j^0, \xi_i^0)=\frac1{2\pi}\log|\log\varepsilon|+O(1),  \qquad \mathrm{for}\qquad i=1,\cdots,l; \quad j=1,2,\cdots,m.\]
From \eqref{96} and the facts above, we get
\begin{eqnarray}\label{97}
\nonumber\max_{(\xi_1, \xi_2, \cdots, \xi_m)\in \bar{\mathcal{O}}_\varepsilon} F(\xi)&\geq& F(\xi^0)\geq8\pi a(q)\left\{(m+l+\alpha+1)|\log\varepsilon| -\left[l(2l+\alpha+1)\right.\right. \\
&&\left.\left.+\frac{m-l}2(m+3l+1+\alpha)\right]\log|\log\varepsilon|\right\}.
\end{eqnarray}
Next, we will prove the maximizer $\xi^\varepsilon\in \Omega_\varepsilon$. Suppose $\xi^\varepsilon\in \partial\Omega_\varepsilon$, the following cases occur:
\begin{description}
  \item[(C1)] There exists $i_0\in \{1,2, \cdots,l\}$ such that $\xi_{i_0}\in \partial B_d(q)$;
  \item[(C2)] There exists $j_0\in \{l+1, \cdots,m\}$ such that $\xi_{j_0}\in \partial B_d(q)$;
  \item[(C3)] There exists $i_0\in \{1,2, \cdots,l\}$ such that $\mathrm{dist}(\xi_{i_0}, \partial \Omega)=\frac1{|\log\varepsilon|^\kappa}$.
  \item[(C4)] There exists $i_0\in \{1,2, \cdots m\}$ such that $|\xi_{i_0}-q|=\frac1{|\log\varepsilon|^\kappa}$
  \item[(C5)] There exist $i_0, k_0\in \{1,2,\cdots, m\}$ such that $i_0\not=k_0$ and $|\xi_{i_0}-\xi_{k_0}|=\frac1{|\log\varepsilon|^\kappa}$.
\end{description}

In the case of (C1) or (C2), there exists $d_0>0$ such that $a(\xi_{i_0})<a(q)-d_0$, since $q$ is a local maximum point of $a(x)$ in $\bar{\Omega}$ and $\partial_n a(q)=0$. From \eqref{96}, we get
\[F(\xi^\varepsilon)\leq 8\pi a(q)(m+l+\alpha+1)|\log\varepsilon|-8\pi d_0 |\log\varepsilon|+O(\log|\log\varepsilon|).\]
We will get a contradiction to \eqref{97} for $\varepsilon>0$ small enough.

In the case of (C3),  for $i=1,2,\cdots,l$ and  $k=1,2,\cdots,m$, we get
\[G(q, \xi_i)>0, \quad G(\xi_k, \xi_i)>0, \quad H(\xi_k,\xi_i)>0\]
and
\[H(\xi_i, \xi_i)=-\frac1{2\pi}\log(\mathrm{dist}(\xi_i, \partial \Omega))+O(1).\]
Then we have
\begin{eqnarray*}
F(\xi^\varepsilon)&<&8\pi(1+\alpha)a(q)\left[-\log\varepsilon -4\pi\sum_{i=1}^l G(q, \xi_i) -2\pi \sum_{i=l+1}^m G(q, \xi_i) \right] \\
&&+16\pi a(q)\sum_{i=1}^l\left[-\log\varepsilon -\alpha \log|\xi_i-q|\right] -32\pi^2 a(q)H(\xi_{i_0}, \xi_{i_0})  \\
&& +8\pi a(q) \sum_{i={l+1}}^m\left[-\log\varepsilon -\alpha |\xi_i-q|\right]+O(1) \\
&=&8\pi (1+\alpha+m+l)|\log\varepsilon|-16 \pi \kappa a(q)\log|\log\varepsilon| \\
&& +8\pi a(q)\sum_{i=1}^l\left[-2\alpha \log|\xi_i-q|-4\pi (1+\alpha)G(q, \xi)\right] \\
&& +8\pi a(q)\sum_{i=l+1}^m \left[-\alpha \log|\xi_i-q|-2\pi (1+\alpha)G(q, \xi_i)\right]+O(1).
\end{eqnarray*}
However, for $i=1,2,\cdots, l$, it holds that
\begin{eqnarray}\label{98}
\nonumber  -2\alpha\log|\xi_i-q|-4\pi(1+\alpha)G(q, \xi_i) &=& 2\log|\xi_i-q|-4\pi(1+\alpha)H(q,\xi_i) \\
   &=& 2\log|\xi_i-q|+2(1+\alpha)\log|q-\xi_i^*|+O(1)\leq C.
\end{eqnarray}
Whereas, for $i=l+1, \cdots, m$,
\begin{equation}\label{99}
  -\alpha\log|\xi_i-q|-2\pi(1+\alpha)G(q, \xi_i) = (2+\alpha)\log|\xi_i-q|-2\pi(1+\alpha)H(q,\xi_i) \leq C.
\end{equation}
Hence we get
\[F(\xi^\varepsilon)<8\pi (1+\alpha +m+l)|\log\varepsilon|-16\pi \kappa a(q)\log|\log\varepsilon|+O(1).\]

In the case of (C4), we have
\begin{eqnarray*}
  F(\xi^\varepsilon) &<& 8\pi (1+\alpha) a(q)\left[-\log\varepsilon -4\pi \sum_{i=1}^l G(q, \xi_i)-2\pi \sum_{i=l+1}^m G(q, \xi_i)\right] \\
   &&+16\pi a(q)\sum_{i=1}^l \left[-\log\varepsilon -\alpha \log|\xi_i-q|\right] \\
   && +8\pi a(q)\sum_{i=l+1}^m \left[-\log\varepsilon-\alpha \log|\xi_i-q|\right]+O(1)  \\
   &=&8\pi a(q)(m+l+\alpha+1)|\log\varepsilon|+8\pi a(q)\sum_{i=1}^l\left[-2\alpha \log|\xi_i-q|-4\pi (1+\alpha)G(q, \xi_i)\right] \\
   &&+8\pi a(q)\sum_{i=l+1}^m \left[-\alpha \log|\xi_i-q|-2\pi (\alpha+1)G(q, \xi_i)\right]+O(1).
\end{eqnarray*}
If $i_0\in \{1,2,\cdots,l\}$, from \eqref{98} and \eqref{99}, we get
\begin{eqnarray*}
F(\xi^\varepsilon)&<&8\pi a(q)(m+l+\alpha+1)|\log\varepsilon|+8\pi a(q)\left[2\log|\xi_i-q|+2(1+\alpha)\log|q-\xi_i^*|\right]+O(1) \\
&=& 8\pi a(q)\left[(m+l+\alpha+1)|\log\varepsilon|-(4+2\alpha)\kappa\log|\log\varepsilon|\right]+O(1).
\end{eqnarray*}
If $i_0\in \{l+1,\cdots,m\}$, we also get the following estimate via the same method:
\begin{eqnarray*}
F(\xi^\varepsilon)&<&8\pi a(q)(m+l+\alpha+1)|\log\varepsilon|+8\pi a(q)(2+\alpha)\log|\xi_i-q|+O(1) \\
&=& 8\pi a(q)\left[(m+l+\alpha+1)|\log\varepsilon|-(2+\alpha)\kappa\log|\log\varepsilon|\right]+O(1).
\end{eqnarray*}

In the case of (C5), we see if $i_0\in \{1,2,\cdots,l\}$ and $k_0\in \{l+1, \cdots, m\}$, we get
\[\frac1{|\log\varepsilon|^\kappa}\leq \mathrm{dist}(\xi_{i_0}, \partial \Omega)\leq |\xi_{i_0}-\xi_{k_0}|=\frac1{|\log\varepsilon|^\kappa}.\]
Hence $\mathrm{dist}(\xi_{i_0}, \partial \Omega)=\frac1{|\log\varepsilon|^\kappa}$. The case (C3) occurs. Therefore we only need to deal with the cases that $i_0,k_0\in \{1,2,\cdots, l\}$ and the case $i_0,k_0\in \{l+1,\cdots, m\}$, respectively.

If $i_0,k_0\in \{1,2,\cdots, l\}$, we get
\begin{eqnarray*}
  F(\xi^\varepsilon) &<& 8\pi(1+\alpha)a(q)\left[-\log\varepsilon -4\pi \sum_{i=1}^l G(q, \xi_i)-2\pi \sum_{i=l+1}^m G(q, \xi_i)\right] \\
  &&+16\pi a(q)\sum_{i=1}^l \left[-\log\varepsilon -\alpha \log|\xi_i-q|\right] +8\pi a(q)\sum_{i=l+1}^m \left[-\log\varepsilon -\alpha\log|\xi_i-q|\right] \\
  && -32\pi^2 a(q)\left[G(\xi_{i_0}, \xi_{k_0})+G(\xi_{k_0},\xi_{i_0} )\right] +O(1) \\
  &\leq& 8\pi a(q)(m+l+\alpha+1)|\log\varepsilon|+8\pi a(q)\sum_{i=1}^l\left[-2\alpha \log|\xi_i-q|-4\pi (1+\alpha)G(q, \xi_i)\right] \\
  &&+8\pi a(q)\sum_{i=l+1}^m \left[-\alpha \log|\xi_i-q|-2\pi (\alpha+1)G(q, \xi_i)\right] \\
  && -32\pi^2 a(q)\left[G(\xi_{i_0}, \xi_{k_0})+G(\xi_{k_0},\xi_{i_0} )\right] +O(1) \\
  &\leq& 8\pi a(q)\left[(m+l+\alpha+1)|\log\varepsilon|-8 \kappa \log|\log\varepsilon|\right]+O(1).
\end{eqnarray*}
If $i_0,k_0\in \{l+1,\cdots, m\}$, we also get
\begin{eqnarray*}
  F(\xi^\varepsilon) &<& 8\pi(1+\alpha)a(q)\left[-\log\varepsilon -4\pi \sum_{i=1}^l G(q, \xi_i)-2\pi \sum_{i=l+1}^m G(q, \xi_i)\right] \\
  &&+16\pi a(q)\sum_{i=1}^l \left[-\log\varepsilon -\alpha \log|\xi_i-q|\right] +8\pi a(q)\sum_{i=l+1}^m \left[-\log\varepsilon -\alpha\log|\xi_i-q|\right] \\
  && -8\pi^2 a(q)\left[G(\xi_{i_0}, \xi_{k_0})+G(\xi_{k_0},\xi_{i_0} )\right] +O(1) \\
  &\leq& 8\pi a(q)\left[(m+l+\alpha+1)|\log\varepsilon|-2 \kappa \log|\log\varepsilon|\right]+O(1).
\end{eqnarray*}
Since
\[\kappa \geq \max\{\frac12, \frac1{2+\alpha}\}\left[l(2l+\alpha+1)+\frac{m-l}2(m+3l+1+\alpha)\right],\]
we get contradiction in the case (C3)-(C5). Hence $\xi^\varepsilon \in \mathcal{O}_\varepsilon$. Theorem \ref{th2} follows.

$\hfill\Box$

\appendix

\section{Useful estimate and Detailed computation}\label{sect4}
In this section, we will prove some  lemmata which are technical and also used in this paper.

Recall the functions $Z_{ks}$'s were defined in Section \ref{sect2}. Let $\mathring{Z}_{ks}$ be the unique solution of the following problem
\begin{equation*}
\left\{\begin{array}{ll}
-\Delta \mathring{Z}_{ks}-\nabla(\log a(\varepsilon y))\cdot \nabla \mathring{Z}_{ks}+\varepsilon^2 \mathring{Z}_{ks}= -\Delta Z_{ks}-\nabla(\log a(\varepsilon y))\cdot \nabla Z_{ks}+\varepsilon^2 Z_{ks} , &\mbox{in $\Omega_\varepsilon$} \\
\frac{\partial \mathring{Z}_{ks}}{\partial n}=0, &\mbox{on $\partial\Omega_\varepsilon$}
\end{array}\right.
\end{equation*}
where $k=1,2,\cdots,l$ and $s=0,1,2$.
\begin{lemma}\label{lm9}
For $k=1,2,\cdots, l$, we get the following estimates
\begin{equation}\label{110}
\mathring{Z}_{k0}=Z_{k0}+O\left(\varepsilon^2|\log\varepsilon|^C\right),\qquad \nabla\mathring{Z}_{k0}=\nabla Z_{k0}+O\left(\varepsilon^3|\log\varepsilon|^C\right),
\end{equation}
and
\begin{equation}\label{111}
\mathring{Z}_{kl}=Z_{kl}+O\left(\varepsilon|\log\varepsilon|^C\right), \qquad \nabla\mathring{Z}_{kl}=\nabla Z_{kl}+O\left(\varepsilon^2|\log\varepsilon|^C\right),
\end{equation}
where $l=1, 2$.
\end{lemma}
\begin{proof}
Let
\[\tilde{Y}_{kl}(x)=\mathring{Z}_{kl}\left(\frac x\varepsilon\right)-Z_{kl}\left(\frac x\varepsilon\right).\]
Then it solves the following problem
\begin{equation*}
\left\{\begin{array}{ll}
-\Delta \tilde{Y}_{kl}-\nabla(\log a(\varepsilon y))\cdot \nabla \tilde{Y}_{kl}+ \tilde{Y}_{kl}=0 , &\mbox{in $\Omega$,} \\
\frac{\partial \tilde{Y}_{kl}}{\partial n}=-\frac1\varepsilon\frac{\partial Z_{kl}}{\partial n}\left(\frac x\varepsilon\right), &\mbox{on $\partial\Omega$.}
\end{array}\right.
\end{equation*}
On $\partial \Omega$, it is apparent that
\[\frac1\varepsilon\frac{\partial Z_{k0}}{\partial n}\left(\frac x\varepsilon\right)=O(\varepsilon^2|\log\varepsilon|^C), \qquad
\frac1\varepsilon\frac{\partial Z_{kl}}{\partial n}\left(\frac x\varepsilon\right)=O(\varepsilon|\log\varepsilon|^C).\]
From elliptic estimate and Sobolev embedding theorem, we get
\[\|\tilde{Y}_{k0}\|_{C^1(\Omega)}\leq C\varepsilon^2|\log\varepsilon|^C, \quad \|\tilde{Y}_{kl}\|_{C^1(\Omega)}\leq C\varepsilon|\log\varepsilon|^C.\]
Hence \eqref{110} and \eqref{111} hold.
\end{proof}

We first derive an expansion of $J(U_\xi)$. Recall $U_\xi$ is defined in \eqref{5}.
\begin{proposition}\label{prop1}
For $\xi\in \mathcal{O}_\varepsilon$, we have
\begin{align*}
J(U_\xi)=&-\frac12 c_0 a(q)\left[4\log \varepsilon+4-2\log8(1+\alpha)^2+c_0H(q,q)+\sum_{i=1}^m c_i G(q, \xi)\right] \\
& -\frac12\sum_{i=1}^m c_ia(\xi_i)\left[4\log\varepsilon +4-2\log 8+c_i H(\xi_i, \xi_i)+c_0 G(\xi_i, q) \right. \\
&\left.+\sum_{k=1, k\not=i}^m c_k G(\xi_i, \xi_k)+4\alpha \log |\xi_i-q|\right] +O\left(\sum_{k=0}^m (\varepsilon \mu_k)^{\min\{\frac2p-1, \beta\}}\right).
\end{align*}
\end{proposition}

\begin{proof}
It is apparent that
\begin{equation*}
J(U_\xi)=\frac12\int_\Omega a(x)\left[|\nabla U_\xi|^2+U_\xi^2\right]dx-\varepsilon^2 \int_\Omega a(x)|x-q|^{2\alpha}(e^{U_\xi}+e^{-U_\xi})dx=: I_A-I_B,
\end{equation*}
where
\[I_A=\frac12\int_\Omega a(x)\left[|\nabla U_\xi|^2+U_\xi^2\right]dx\]
and
\[I_B=\varepsilon^2 \int_\Omega a(x)|x-q|^{2\alpha}(e^{U_\xi}+e^{-U_\xi})dx.\]
For notation simplicity, we define
\[\R_{z}=\left\{\begin{array}{ll}
\R^2,  &\mbox{if $z\in \Omega$,} \\
\R_+^2, &\mbox{if $z\in \partial\Omega$.}
\end{array}\right.\]

We calculate $I_A$ first. From \eqref{84}, \eqref{85} and \eqref{86}, we get
\begin{align}\label{7}
\nonumber I_A&= \frac12\int_\Omega a(x)\left[\varepsilon^2|x-q|^{2\alpha}e^{U_0}+\sum_{i=1}^m\varepsilon^2|\xi_i-q|^{2\alpha} e^{U_i}\right]U_\xi dx \\
&=\frac12\int_\Omega a(x)\left[\frac{8(1+\alpha)^2b_0(\varepsilon\mu_0)^{2(1+\alpha)}|x-q|^{2\alpha}}{[(\varepsilon \mu_0)^{2(1+\alpha)}+|x-q|^{2(1+\alpha)}]^2}+\sum_{i=1}^m \frac{8b_i(\varepsilon \mu_i)^2}{\left[(\varepsilon \mu_i)^2+|x-\xi_i|^2\right]^2}\right]U_\xi(x)dx.
\end{align}
For the first term in \eqref{7}, we have
\begin{align}\label{9}
\nonumber&\int_\Omega a(x)\frac{8(1+\alpha)^2b_0(\varepsilon\mu_0)^{2(1+\alpha)}|x-q|^{2\alpha}}{\left[(\varepsilon \mu_0)^{2(1+\alpha)}+|x-q|^{2(1+\alpha)}\right]^2}U_\xi(x)dx \\
\nonumber=&\int_{\Omega\cap B_{\frac1{|\log\varepsilon|^{2\kappa}}}(q)} a(x)\frac{8(1+\alpha)^2b_0(\varepsilon\mu_0)^{2(1+\alpha)}|x-q|^{2\alpha}}{\left[(\varepsilon \mu_0)^{2(1+\alpha)}+|x-q|^{2(1+\alpha)}\right]^2}U_\xi(x)dx \\
&+\sum_{i=1}^m\int_{\Omega\cap B_{\frac1{|\log\varepsilon|^{2\kappa}}}(\xi_i)} a(x)\frac{8(1+\alpha)^2b_0(\varepsilon\mu_0)^{2(1+\alpha)}|x-q|^{2\alpha}}{\left[(\varepsilon \mu_0)^{2(1+\alpha)}+|x-q|^{2(1+\alpha)}\right]^2}U_\xi(x)dx \\
\nonumber&+\int_{\Omega\char92 \left\{ B_{\frac1{|\log\varepsilon|^{2\kappa}}}(q)\cup \left[\cup_{i=1}^m B_{\frac1{|\log\varepsilon|^{2\kappa}}}(\xi_i)\right]\right\}} a(x)\frac{8(1+\alpha)^2b_0(\varepsilon\mu_0)^{2(1+\alpha)}|x-q|^{2\alpha}}{\left[(\varepsilon \mu_0)^{2(1+\alpha)}+|x-q|^{2(1+\alpha)}\right]^2}U_\xi(x)=:I_4+I_5+I_6.
\end{align}
From expression \eqref{8} and \eqref{51}, we have
\begin{align*}
I_4=& \int_{\Omega\cap B_{\frac1{|\log\varepsilon|^\kappa}}(q)} a(x)\frac{8(1+\alpha)^2(\varepsilon\mu_0)^{2(1+\alpha)}|x-q|^{2\alpha}}{\left[(\varepsilon \mu_0)^{2(1+\alpha)}+|x-q|^{2(1+\alpha)}\right]^2}U_0(x)dx \\
&+O\left(\sum_{k=0}^m(\varepsilon \mu_k)^{\frac2p-1}\right)+O((\varepsilon \mu_0)^{\min\{\beta,2+2\alpha\}}) \\
=& a(q)\int_{\R_q}\frac{8(1+\alpha)^2|t|^{2\alpha}}{[1+|t|^{2(1+\alpha)}]^2}\log \frac{8(1+\alpha)^2}{[1+|t|^{2(1+\alpha)}]^2}dt  \\
 &-\left(4\log\varepsilon +\log(\mu_0^{2(1+\alpha)}\varepsilon^{2\alpha})\right)a(q)\int_{\R_q}\frac{8(1+\alpha)^2 |t|^{2\alpha}}{[1+|t|^{2(1+\alpha)}]^2}dt \\
& +O\left(\sum_{k=0}^m(\varepsilon \mu_k)^{\frac2p-1}\right)+O((\varepsilon \mu_0)^\beta).
\end{align*}
For each term in $I_5$, we get the following estimate from \eqref{75}:
\begin{align*}
& \left|\int_{\Omega\cap B_{\frac1{|\log\varepsilon|^{2\kappa}}}(\xi_i)} \frac{b_0(\varepsilon\mu_0)^{2(1+\alpha)}|x-q|^{2\alpha}}{\left[(\varepsilon\mu_0)^{2(1+\alpha)} +|x-q|^{2(1+\alpha)}\right]^2}U_\xi(x)dx\right| \\
\leq &C(\varepsilon \mu_0)^{2(1+\alpha)}|\log\varepsilon|^{4(2+\alpha)\kappa}\int_{\Omega\cap B_{\frac1{|\log\varepsilon|^{2\kappa}}}(\xi_i)}\left|U_i(x)+O\left(\sum_{k=0}^m(\varepsilon \mu_k)^{\frac2p-1}\right)+O((\varepsilon \mu_0)^\beta)\right| \\
\leq & C(\varepsilon \mu_0)^{2(1+\alpha)}|\log \varepsilon|^{4(1+\alpha) \kappa}(\log|\log\varepsilon|+\log(\varepsilon\mu_i)),
\end{align*}
Using \eqref{53}, we also get
\begin{align*}
I_6=& \int_{\Omega\char92 \left\{ B_{\frac1{|\log\varepsilon|^{2\kappa}}}(q)\cup \left[ \cup_{i=1}^m B_{\frac1{|\log\varepsilon|^{2\kappa}}}(\xi_i)\right]\right\}} a(x)\frac{8(1+\alpha)^2b_0(\varepsilon\mu_0)^{2(1+\alpha)}|x-q|^{2\alpha}}{\left[(\varepsilon \mu_0)^{2(1+\alpha)}+|x-q|^{2(1+\alpha)}\right]^2} \\
&\times\left(c_0 G(x, q)+\sum_{i=1}^m c_i G(x, \xi_i)+O\left(\sum_{k=0}^m(\varepsilon \mu_k)^{\frac2p-1}\right)\right)dx \\
\leq &C\log|\log\varepsilon|\int_{\Omega\char92 B_{\frac1{|\log\varepsilon|^{2\kappa}}}(q)} \frac{(\varepsilon\mu_0)^{2(1+\alpha)}|x-q|^{2\alpha}}{\left[(\varepsilon\mu_0)^{2(1+\alpha)} +|x-q|^{2(1+\alpha)}\right]^2}dx \\
\leq & C(\varepsilon \mu_0)^{2+2\alpha}|\log\varepsilon|^{4\kappa(1+\alpha)}\log|\log\varepsilon|.
\end{align*}
Hence there holds
\begin{align}\label{11}
\nonumber&\int_\Omega a(x)\frac{8(1+\alpha)^2b_0(\varepsilon\mu_0)^{2(1+\alpha)}|x-q|^{2\alpha}}{\left[(\varepsilon \mu_0)^{2(1+\alpha)}+|x-q|^{2(1+\alpha)}\right]^2}U_\xi(x)dx \\
= & a(q)\left[ -\left(4\log\varepsilon \nonumber+\log(\mu_0^{2(1+\alpha)}\varepsilon^{2\alpha})\right)\int_{\R_q}\frac{8(1+\alpha)^2 |t|^{2\alpha}}{[1+|t|^{2(1+\alpha)}]^2}dt+\int_{\R_q}\frac{8(1+\alpha)^2|t|^{2\alpha}}{[1+|t|^{ 2(1+\alpha)}]^2}\log \frac{8(1+\alpha)^2}{[1+|t|^2]}dt \right] \\
& +O\left(\sum_{k=0}^m(\varepsilon \mu_k)^{\frac2p-1}\right)+O((\varepsilon \mu_0)^\beta).
\end{align}
Now we estimate the second term in \eqref{7}. As in \eqref{9}, we have
\begin{align*}
&\int_\Omega a(x)\frac{8b_i(\varepsilon \mu_i)^2}{\left[(\varepsilon \mu_i)^2+|x-\xi_i|^2\right]^2}U_\xi(x)dx \\
\nonumber=&\int_{\Omega\cap B_{\frac1{|\log\varepsilon|^{2\kappa}}}(\xi_i)} a(x)\frac{8b_i(\varepsilon \mu_i)^2}{\left[(\varepsilon \mu_i)^2+|x-\xi_i|^2\right]^2}U_\xi(x)dx \\
& +\sum_{j=1, j\not=i}^m\int_{\Omega\cap B_{\frac1{|\log\varepsilon|^{2\kappa}}}(\xi_j)} a(x)\frac{8b_i(\varepsilon \mu_i)^2}{\left[(\varepsilon \mu_i)^2+|x-\xi_i|^2\right]^2}U_\xi(x)dx \\
& +\int_{\Omega\cap B_{\frac1{|\log\varepsilon|^{2\kappa}}}(q)} a(x)\frac{8b_i(\varepsilon \mu_i)^2}{\left[(\varepsilon \mu_i)^2+|x-\xi_i|^2\right]^2}U_\xi(x)dx\\
\nonumber&+\int_{\Omega\char92 \left[\cup_{i=1}^m B_{\frac1{|\log\varepsilon|^{2\kappa}}}(\xi_i)\cup B_{\frac1{|\log\varepsilon|^{2\kappa}}}(q)\right]} a(x)\frac{8b_i(\varepsilon \mu_i)^2}{\left[(\varepsilon \mu_i)^2+|x-\xi_i|^2\right]^2}U_\xi(x)dx=:I_7+I_8+I_9+I_{10}.
\end{align*}
For the term $I_7$, we get
\begin{align*}\
I_7=& \int_{\Omega\cap B_{\frac1{|\log\varepsilon|^{2\kappa}}}(\xi_i)} a(x)\frac{8b_i(\varepsilon \mu_i)^2}{\left[(\varepsilon \mu_i)^2+|x-\xi_i|^2\right]^2}\left(b_iU_i(x)+O(|x-\xi_i|^\beta)+O\left(\sum_{k=0}^m(\varepsilon \mu_k)^{\frac2p-1}\right)\right)dx \\
=& a(\xi_i)\left[-\left(4\log\varepsilon+\log (\mu_i^2|\xi_i-q|^{2\alpha})\right)\int_{\R_{\xi_i}}\frac{8}{[1+|t|^2]^2}dt +\int_{\R_{\xi_i}}\frac{8}{[1+|t|^2]^2}\log \frac{8}{[1+|t|^2]^2}dt \right] \\
& +O\left((\varepsilon\mu_i)^\beta\right)+O\left(\sum_{k=0}^m(\varepsilon \mu_k)^{\frac2p-1}\right).
\end{align*}
For $I_8$, we get
\begin{align*}
|I_8|\leq & C\sum_{j=1}^m (\varepsilon \mu_i)^2|\log|^{8\kappa}\int_{\Omega\cap B_{\frac1{|\log\varepsilon|^{2\kappa}}}(\xi_j)}\left|b_jU_j(x)+O(|x-\xi_j|^\beta)+O\left( \sum_{k=0}^m(\varepsilon \mu_k)^{\frac2p-1}\right)\right|dx \\
\leq & C\sum_{j=1}^m (\varepsilon\mu_i)^2|\log|^{4\kappa}\left[\log|\log\varepsilon|+\log(\varepsilon\mu_j)\right].
\end{align*}
As the same way, we get
\begin{align*}
|I_9|\leq & (\varepsilon \mu_i)^2|\log|^{4\kappa}\int_{\Omega\cap B_{\frac1{|\log\varepsilon|^{2\kappa}}}(q)}\left|b_0U_0(x)+O(|x-\xi_j|^\beta)+O\left(\sum_{k=0}^m( \varepsilon \mu_k)^{\frac2p-1}\right)\right| \\
\leq & C(\varepsilon\mu_i)^2|\log|^{4\kappa}\left[\log|\log\varepsilon|+\log(\varepsilon\mu_0)\right].
\end{align*}
For the last term, we have
\begin{align*}
|I_{10}|\leq & \left|\int_{\Omega\char92 \left[\cup_{i=1}^m B_{\frac1{|\log\varepsilon|^{2\kappa}}}(\xi_i)\cup B_{\frac1{|\log\varepsilon|^{2\kappa}}}(q)\right]} \frac{8b_i(\varepsilon \mu_i)^2}{\left[(\varepsilon \mu_i)^2+|x-\xi_i|^2\right]^2} \right. \\
&\left.\times\left(b_0c_0 G(x, q)+\sum_{i=1}^m b_ic_i G(x, \xi_i)+O\left(\sum_{k=0}^m(\varepsilon \mu_k)^{\frac2p-1}\right)\right)dx\right| \\
\leq & C(\varepsilon \mu_i)^2|\log\varepsilon|^{4\kappa}\log|\log\varepsilon|.
\end{align*}
Hence
\begin{align}\label{10}
&\nonumber\int_\Omega a(x)\frac{8b_i(\varepsilon \mu_i)^2}{\left[(\varepsilon \mu_i)^2+|x-\xi_i|^2\right]^2}U_\xi(x)dx \\
=&  a(\xi_i)\left[ -\left(4\log\varepsilon +\log(\mu_i^2|\xi_i-q|^{2\alpha})\right)\int_{\R_{\xi_i}}\frac8{[1+|t|^2]^2}dt\right. \\
\nonumber&\left.+\int_{\R_{\xi_i}}\frac8{[1+|t|^2]^2}\log \frac8{[1+|t|^2]}dt \right]+O\left(\sum_{k=0}^m(\varepsilon \mu_k)^{\frac2p-1}\right)+O((\varepsilon \mu_i)^\beta).
\end{align}
From \eqref{11} and \eqref{10}, we get
\begin{align}\label{14}
\nonumber I_A=& \frac12a(q) \left[-\left(4\log\varepsilon+\log (\mu_0^{2(1+\alpha)}\varepsilon^{2\alpha})\right)\int_{\R_q}\frac{8(1+\alpha)^2|t|^{2\alpha}}{ [1+|t|^{2(1+\alpha)}]^2}dt \right.\\
&\left.+\int_{\R_q}\frac{8(1+\alpha)^2|t|^{2\alpha}}{[1+|t|^{2(1+\alpha)}]^2}\log \frac{8(1+\alpha)^2}{[1+|t|^2]^2}dt\right] +\frac12\sum_{i=1}^m a(\xi_i)\left[\int_{\R_{\xi_i}}\frac{8}{[1+|t|^2]^2}\log \frac{8}{[1+|t|^2]^2}dt\right.\\
\nonumber & \left. -\left(4\log\varepsilon+\log (\mu_i^2|\xi_i-q|^{2\alpha})\right)\int_{\R_{\xi_i}}\frac{8}{[1+|t|^2]^2}dt\right] + O\left(\sum_{k=0}^m(\varepsilon \mu_k)^{\min\{\frac2p-1, \beta\}}\right).
\end{align}
Now we estimate the term $I_B$. It is apparent that
\begin{align*}
I_B=& \left(\int_{\Omega\cap B_{\frac1{|\log\varepsilon|^{2\kappa}}}(q)}+\sum_{i=1}^m\int_{\Omega\cap B_{\frac1{|\log\varepsilon|^{2\kappa}}}(\xi_i)}+\int_{\Omega \char92 \left[\left(\cup_{i=1}^m B_{\frac1{|\log\varepsilon|^{2\kappa}}}(\xi_i)\right)\cup B_{\frac1{|\log\varepsilon|^{2\kappa}}}(q)\right]}\right)\varepsilon^2 a(x)|x-q|^{2\alpha} \\
& \times\left(e^{U_\xi(x)}+e^{-U_\xi(x)}\right)dx=: I_{11}+I_{12}+I_{13}.
\end{align*}
For $x\in \Omega\cap B_{\frac1{|\log\varepsilon|^{2\kappa}}}(q)$, we have
\[e^{U_\xi(x)}=e^{b_0U_0(x)}\left(1+O(|x-q|^\beta)+O\left(\sum_{k=0}^m(\varepsilon \mu_k)^{\frac2p-1}\right)\right)\]
and
\begin{equation*}
e^{U_\xi(x)}+e^{-U_\xi(x)}=\left(e^{U_0(x)}+e^{-U_0(x)}\right)\left(1+O(|x-q|^\beta)+O\left(\sum_{k=0}^m( \varepsilon \mu_k)^{\frac2p-1}\right)\right)
\end{equation*}
Hence we get
\begin{align*}
  I_{11} =& \int_{\Omega\cap B_{\frac1{|\log\varepsilon|^{2\kappa}}}(q)} a(x)\frac{8(1+\alpha)^2\mu_0^{2+2\alpha}\varepsilon^{2+2\alpha}|x-q|^{2\alpha}}{[(\varepsilon \mu_0)^{2(1+\alpha)}+|x-q|^{2(1+\alpha)}]^2}\left(1+O(|x-q|^\beta)+O\left(\sum_{k=0}^m( \varepsilon \mu_k)^{\frac2p-1}\right)\right) \\
  &+ \varepsilon^{2-2\alpha}\int_{\Omega\cap B_{\frac1{|\log\varepsilon|^{2\kappa}}}(q)}  a(x)|x-q|^{2\alpha}\frac{[(\varepsilon \mu_0)^{2(1+\alpha)}+|x-q|^{2(1+\alpha)}]^2}{8(1+\alpha)^2 \mu_0^{2+2\alpha}} \\
  &\quad \times\left(1+O(|x-q|^\beta)+O\left(\sum_{k=0}^m(\varepsilon \mu_k)^{\frac2p-1}\right)\right)dx \\
  =& \int_{\Omega\cap B_{\frac1{|\log\varepsilon|^{2\kappa}}}(q)} a(x) \frac{8(1+\alpha)^2\mu_0^{2+2\alpha}\varepsilon^{2(1+\alpha)}|x-q|^{2\alpha}}{[(\varepsilon \mu_0)^{2(1+\alpha)}+|x-q|^{2(1+\alpha)}]^2}dx +O((\varepsilon\mu_0)^\beta)+O\left(\sum_{k=0}^m(\varepsilon \mu_k)^{\frac2p-1}\right) \\
  =& a(q) \int_{\R_q} \frac{8(1+\alpha)^2|t|^{2\alpha}}{[1+|t|^{2(1+\alpha)}]^2}dt +O((\varepsilon\mu_0)^\beta)+O\left(\sum_{k=0}^m(\varepsilon \mu_k)^{\frac2p-1}\right).
\end{align*}
However, in $\Omega\cap B_{\frac1{|\log\varepsilon|^{2\kappa}}(\xi_i)}$, we have
\[e^{U_\xi(x)}=e^{b_iU_i(x)}\left[1+O(|x-\xi_i|^\beta)+O\left(\sum_{k=0}^m(\varepsilon \mu_k)^{\frac2p-1}\right) \right] \]
and
\[e^{U_\xi(x)}+E^{-U_\xi(x)}=\left[e^{U_i(x)}+e^{U_i(x)}\right]\left[1+O(|x-\xi_i|^\beta)+O\left(\sum_{k=0}^m(\varepsilon \mu_k)^{\frac2p-1}\right) \right] \]
Thus
\begin{align*}
  I_{12}= & \sum_{i=1}^m \int_{\Omega\cap B_{\frac1{|\log\varepsilon|^{2\kappa}}}(\xi_i)}\frac{8(\varepsilon \mu_i)^2 a(x)|x-q|^{2\alpha}}{[(\varepsilon\mu_i)^2+|x-\xi_i|^2]^2|\xi_i-q|^{2\alpha}}\left[ 1+O(|x-\xi_i|^\beta)+O\left(\sum_{k=0}^m(\varepsilon \mu_k)^{\frac2p-1}\right) \right] \\
  & +\varepsilon^2\sum_{i=1}^m\frac{|\xi_i-q|^{2\alpha}}{8\mu_i^2}\int_{\Omega\cap B_{\frac1{|\log\varepsilon|^{2\kappa}}}(\xi_i)}a(x)|x-q|^{2\alpha}[(\varepsilon \mu_i)^2+|x-\xi_i|^2]^2 \\
  &\qquad\times\left[1+O(|x-\xi_i|^\beta)+O\left(\sum_{k=0}^m(\varepsilon \mu_k)^{\frac2p-1}\right) \right] \\
  =& \sum_{i=1}^m \int_{\Omega\cap B_{\frac1{|\log\varepsilon|^{2\kappa}}}(\xi_i)}a(x)\frac{8(\varepsilon \mu_i)^2|x-q|^{2\alpha}}{[(\varepsilon\mu_i)^2+|x-\xi_i|^2]^2|\xi_i-q|^{2\alpha}}dx + O\left(\sum_{k=0}^m(\varepsilon \mu_i)^{\min\{\frac2p-1, \beta\}}\right) \\
  =&\sum_{i=1}^m a(\xi_i)\int_{\R_{\xi_i}}\frac{8}{[1+|t|^2]^2}dt+ O\left(\sum_{k=0}^m(\varepsilon \mu_i)^{\min\{\frac2p-1, \beta\}}\right).
\end{align*}
For $x\in \Omega \char92 \left[\left(\cup_{i=1}^m B_{\frac1{|\log\varepsilon|^{2\kappa}}}(\xi_i)\right)\cup B_{\frac1{|\log\varepsilon|^{2\kappa}}}(q)\right]$, we get the following estimate from definition of $\mathcal{O}_\varepsilon$ in \eqref{12}:
\[0<G(x,q)<C\log|\log\varepsilon|  \quad \mathrm{and} \quad 0<G(x, \xi_i)<C\log|\log\varepsilon|\]
Hence
\[
  |I_{13}| = \varepsilon^2 \int_{\Omega \char92 \left[\left(\cup_{i=1}^m B_{\frac1{|\log\varepsilon|^{2\kappa}}}(\xi_i)\right)\cup B_{\frac1{|\log\varepsilon|^{2\kappa}}}(q)\right]} a(x)|x-q|^{2\alpha} \left(e^{U_\xi(x)}+e^{-U_\xi(x)}\right)dx \leq C\varepsilon^2|\log\varepsilon|^C.
\]
Therefore
\begin{equation}\label{13}
  I_B= a(q)\int_{\R_q}\frac{8(1+\alpha)^2|t|^{2\alpha}}{[1+|t|^{2(1+\alpha)}]^2}dt+\sum_{i=1}^m a(\xi_i)\int_{\R_{\xi_i}}\frac{8}{[1+|t|^2]^2}dt + O\left(\sum_{k=0}^m(\varepsilon \mu_i)^{\min\{\frac2p-1, \beta\}}\right).
\end{equation}
Then this lemma follows from \eqref{14}, \eqref{13}, \eqref{4}, the definition of $c_i$'s and the following identities:
\[\int_{\R^2} \frac8{[1+|t|^2]^2}dt=8\pi, \qquad \int_{\R^2}\frac{8(1+\alpha)^2|t|^{2\alpha}}{[1+|t|^{2(1+\alpha)}]^2}dt=8\pi(1+\alpha),\]
\[\int_{\R^2}\frac{8(1+\alpha)^2|t|^{2\alpha}}{[1+|t|^{2(1+\alpha)}]^2}\log \frac{8(1+\alpha)^2}{[1+|t|^2]^2}dt=8\pi(1+\alpha)\left[\log(8(1+\alpha)^2)-2\right],\]
and
\[\int_{\R^2}\frac{8}{[1+|t|^2]^2}\log \frac{8}{[1+|t|^2]^2}dt=8\pi\left[\log8-2\right],\]
which is easy to derive.

\end{proof}

Now we will prove some estimates which are needed in Lemma \ref{lm4}. Recall that $\tilde{Z}_{k0}$'s are defined in \eqref{22} and \eqref{23} for $k=0,1,\cdots, m,$ and the norm $\|\cdot\|_*$ is defined in \eqref{24}. The operator $L$ is defined in \eqref{87}.
\begin{lemma}\label{lm3}
For $\varepsilon>0$ small enough and $\mu_j$'s satisfying \eqref{4}, we have
\begin{equation}\label{25}
\|L(\tilde{Z}_{k0})\|_*\leq C\frac{\log|\log \varepsilon|}{\mu_k|\log\varepsilon|},\qquad \mathrm{for} \qquad k=0,1,\cdots,m.
\end{equation}
and
\begin{equation}\label{26}
\|L(\chi_k Z_{ks})\|\leq \frac C{\mu_k}, \qquad  \mathrm{for} \qquad k=1,2,\cdots,m; s=1, J_k.
\end{equation}
\end{lemma}

\begin{proof}
We first prove \eqref{25}. For $j=1,2, \cdots, m$, we have
\begin{align}\label{27}
\nonumber  L(\tilde{Z}_{j0}) =& \eta_{j1}^\xi L(Z_{j0}) +\eta_2^q(1-\eta_{j1}^\xi) L(\hat{Z}_{j0}) -\Delta \eta_{j1}^\xi(Z_{j0}-\hat{Z}_{j0}) -2\nabla \eta_{j1}^\xi \cdot \nabla(Z_{j0}-\hat{Z}_{j0}) \\
\nonumber    & -\varepsilon (\nabla\log a)(\varepsilon y )\cdot \nabla \eta_{j1}^\xi (Z_{j0}-\hat{Z}_{j0})-\Delta \eta_2^q \hat{Z}_{j0}-2\nabla \eta_2^q\cdot \nabla \hat{Z}_{j0} \\ &-\varepsilon (\nabla\log a)(\varepsilon y)\cdot \nabla \eta_2^q \hat{Z}_{j0}.
\end{align}
From the definition of $Z_{j0}$, we get
\[\Delta Z_{j0}=-\frac{8\mu_j^2}{[\mu_j^2+|z_j|^2]^2}\frac1{\mu_j}Z_0\left(\frac{z_j}{\mu_j}\right)+O\left( \frac{\varepsilon \mu_j}{(\mu_j+|z_j|)^3}\right).\]

For $|z_j|\leq (R+1)\mu_j$, we have
\begin{align*}
  L(Z_{j0})= &\left(\frac{8\mu_j^2}{[\mu_j^2+|z_j|^2]^2}- \frac{8\mu_j^2}{[\mu_j^2+|y-\xi_j'|^2]^2}\frac{|\varepsilon y-q|^{2\alpha}}{|\xi_j-q|^{2\alpha}}\right)\frac1{\mu_j} Z_0\left(\frac {z_j}{\mu_j}\right) +O\left(\frac{\varepsilon^2}{\mu_j}\right)  \\
  & +\frac{8\mu_j^2}{[\mu_j^2+|z_j|^2]^2}\left(O(|\varepsilon y-\xi_j|^\beta)+O\left(\sum_{k=0}^m (\varepsilon \mu_k)^{\frac2p-1}\right)\right)\frac1{\mu_j}Z_0\left(\frac{z_j}{\mu_j}\right).
\end{align*}
Hence
\begin{equation}\label{31}
\left|\eta_{j1}L(Z_{j0})\right|\leq \frac{C}{\mu_j}\left((\varepsilon \mu_j)^\beta+\sum_{k=0}^m(\varepsilon \mu_k)^{\frac2p-1}\right)\left[\varepsilon^2+\frac{\mu_j^\sigma}{(\mu_j+|y-\xi_j'|)^{2+\sigma}}\right].
\end{equation}

For $R\mu_j\leq |z_j|\leq (R+1)\mu_j$, we get
\begin{equation*}
Z_{j0}-\hat{Z}_{j0}=-\frac{\log\left|\frac{y-\xi_j'}{\mu_j R}\right|}{\mu_j |\log(\varepsilon\mu_j R)|}\left(1+O\left(\frac{\log|\log\varepsilon|}{|\log\varepsilon|}\right)\right).
\end{equation*}
In this case
\begin{equation}\label{89}
\left|Z_{j0}-\hat{Z}_{j0}\right|=O\left(\frac1{\mu_j|\log\varepsilon|}\right), \qquad \left|\nabla Z_{j0}-\nabla \hat{Z}_{j0}\right|=O\left(\frac1{\mu_j^2|\log\varepsilon|}\right).
\end{equation}
Hence we get
\begin{equation}\label{30}
\left|\Delta \eta_{j1}^\xi (Z_{j0}-\hat{Z}_{j0})+2\nabla \eta_{j1}^\xi\cdot \nabla(Z_{j0}-\hat{Z}_{j0})+\varepsilon (\log a)(\varepsilon y)\cdot \nabla \eta_{j1}^\xi(Z_{j0}-\hat{Z}_{j0})\right|\leq \frac{C}{\mu_j^3|\log\varepsilon|}.
\end{equation}

In the case of $\frac{4d}\varepsilon\leq |z_0|\leq \frac{8d}\varepsilon$, we have $\frac{d}{\varepsilon}\leq |y-\xi_j'|\leq \frac{17d}\varepsilon$. Then we get
\begin{equation}\label{90}
\hat{Z}_{j0}=Z_{j0}-\frac1{\mu_j}+a_{j0}^\xi G(\varepsilon y, \xi_j)+(\mathring{Z}_{j0}-Z_{j0})\chi_{\{1\leq j\leq l\}}=O\left(\frac1{\mu_j|\log\varepsilon|}\right)
\end{equation}
and
\begin{equation}\label{121}
\nabla \hat{Z}_{j0}=\nabla Z_{j0}+\varepsilon a_{j0}^\xi\nabla G(\varepsilon y, \xi_j)+(\nabla\mathring{Z}_{j0}-\nabla Z_{j0})\chi_{\{1\leq j\leq l\}}=O\left(\frac{\varepsilon}{\mu_j |\log\varepsilon|}\right).
\end{equation}
Hence
\begin{equation}\label{29}
\Delta \eta_2^q \hat{Z}_{j0}+2\nabla \eta_2^q\cdot \nabla \hat{Z}_{j0}+\varepsilon(\nabla\log a)(\varepsilon y)\cdot \nabla \eta_2^q \hat{Z}_{j0}=O\left(\frac{\varepsilon^2}{\mu_j |\log\varepsilon|}\right).
\end{equation}

At last, we consider the case that $R\mu_j\leq |z_j|$ and $|z_0|\leq \frac{8d}\varepsilon$. We consider different cases.

1) If $|\varepsilon y-\xi_j|\leq \frac1{|\log\varepsilon|^{2\kappa}}$, we have
\begin{align*}
  L(\hat{Z}_{j0})= &-\Delta Z_{j0} -WZ_{j0}-\varepsilon (\nabla \log a)(\varepsilon y)\cdot \nabla Z_{j0}+\varepsilon^2\left(Z_{j0}-\frac1{\mu_j}\right) +W\left(\frac1{\mu_j}-a_{j0}^\xi G(\varepsilon y, \xi_j)\right) \\
  &-W(\mathring{Z}_{j0}-Z_{j0})\chi_{\{1\leq j\leq l\}}\\
  =& \left[\frac{8\mu_j^2}{[\mu_j^2+|z_j|^2]^2}-\frac{8\mu_j^2}{[\mu_j^2+|y-\xi_j'|^2]^2}\frac{ |\varepsilon y-q|^{2\alpha}}{|\xi_j-q|^{2\alpha}}\right]\frac1{\mu_j}Z_0\left(\frac{z_j}{\mu_j}\right) +O\left(\frac{\varepsilon^4|\log\varepsilon|^C}{\mu_j}\right) \\
  & +O\left(\frac{\varepsilon \mu_j}{(\mu_j+|z_j|)^3}\right)+\frac{8\mu_j^2}{[\mu_j^2+|y-\xi_j'|^2]^2}\left(O\left(|\varepsilon y-\xi_j|^\beta\right)+O\left(\sum_{k=0}^m (\varepsilon \mu_k)^{\frac2p-1}\right)\right)\frac1{\mu_j}Z_0\left( \frac{z_j}{\mu_j}\right) \\
  &+O\left(\frac{\mu_j^2}{(\mu_j+|y-\xi_j'|)^4}\frac{1+\log\left|\frac{y-\xi_j'}{\mu_j} \right|}{ \mu_j|\log\varepsilon|}\right).
\end{align*}

2) If $|\varepsilon y-\xi_k|\leq \frac1{|\log\varepsilon|^{2\kappa}}$, where $k\not=j$. In this case we get $|\varepsilon y-\xi_j|\geq \frac1{2|\log\varepsilon|^{\kappa}}$ for $\varepsilon>0$ small enough. Hence
\begin{align*}
  L(\hat{Z}_{j0})= & -\Delta Z_{j0}+W\left(\frac1{\mu_j}-Z_{j0}\right) -\varepsilon (\nabla \log a)(\varepsilon y)\cdot \nabla Z_{j0}+\varepsilon^2 \left(Z_{j0}-\frac1{\mu_j}\right) -a_{j0}^\xi W G(\varepsilon y, \xi_j) \\
  &-W(\mathring{Z}_{j0}-Z_{j0})\chi_{\{1\leq j\leq l\}}\\
   =& \frac{8\mu_j^2}{[\mu_j^2+|z_j|^2]^2}\frac1{\mu_j}Z_0\left(\frac{z_j}{\mu_j}\right) +O\left(\frac{\mu_k^2}{(\mu_k+|y-\xi_k'|)^4}\frac{\mu_j}{(\mu_j+|z_j|)^2}\right) +O\left(\frac{\varepsilon \mu_j}{(\mu_j+|z_j|)^3}\right) \\
   & +O\left(\frac{\mu_j^2}{(\mu_j+|y-\xi_j'|)^4}\frac{\log|\log\varepsilon|}{\mu_j| \log\varepsilon|}\right)+O\left(\varepsilon^2\frac{\log|\log\varepsilon|}{|\log\varepsilon|}\right).
\end{align*}

3) If $|\varepsilon y-q|\leq \frac1{|\log\varepsilon|^{2\kappa}}$, we have
\begin{align*}
  L(\hat{Z}_{j0})= & -\Delta Z_{j0}+W\left(\frac1{\mu_j}-Z_{j0}\right) -\varepsilon (\nabla \log a)(\varepsilon y)\cdot \nabla Z_{j0}+\varepsilon^2 \left(Z_{j0}-\frac1{\mu_j}\right) -a_{j0}^\xi W G(\varepsilon y, \xi_j) \\
  & -W(\mathring{Z}_{j0}-Z_{j0})\chi_{\{1\leq j\leq l\}}\\
   =& \frac{8\mu_j^2}{[\mu_j^2+|z_j|^2]^2}\frac1{\mu_j}Z_0\left(\frac{z_j}{\mu_j}\right) +O\left(\frac{\mu_0^{2+2\alpha}|y-q'|^{2\alpha}}{(\mu_0+|y-q'| )^{4+4\alpha}}\frac{\mu_j}{(\mu_j+|z_j|)^2}\right) +O\left(\frac{\varepsilon \mu_j}{(\mu_j+|z_j|)^3}\right) \\
   & +O\left(\frac{\mu_0^{2+2\alpha}|y-q'|^{2\alpha}}{(\mu_0+|y-q'| )^{4+4\alpha}}\frac{\log| \log\varepsilon|}{\mu_j|\log\varepsilon|}\right)+O\left(\varepsilon^2\frac{\log|\log\varepsilon|}{|\log\varepsilon|}\right).
\end{align*}

4) If $y$ satisfies $|\varepsilon y-q|\geq \frac1{|\log\varepsilon |^{2\kappa}}$ and $|\varepsilon y-\xi_k|\geq \frac1{|\log\varepsilon |^{2\kappa}}$ for $k=1,2,\cdots,m$, we also have
\begin{align*}
  L(\hat{Z}_{j0})= & -\Delta Z_{j0}+W\left(\frac1{\mu_j}-Z_{j0}\right) -\varepsilon (\nabla \log a)(\varepsilon y)\cdot \nabla Z_{j0}+\varepsilon^2 \left(Z_{j0}-\frac1{\mu_j}\right) -a_{j0}^\xi W G(\varepsilon y, \xi_j) \\
  & -W(\mathring{Z}_{j0}-Z_{j0})\chi_{\{1\leq j\leq l\}} \\
   =& \frac{8\mu_j^2}{[\mu_j^2+|z_j|^2]^2}\frac1{\mu_j}Z_0\left(\frac{z_j}{\mu_j}\right) +O\left(\frac{\varepsilon^4|\log\varepsilon|^C\mu_j}{(\mu_j+|z_j|)^2}\right) +O\left(\frac{\varepsilon \mu_j}{(\mu_j+|z_j|)^3}\right) \\
   & +O\left(\varepsilon^4|\log\varepsilon|^C\frac{\log|\log\varepsilon|}{\mu_j| \log\varepsilon|}\right).
\end{align*}
Combining these cases above, we get
\begin{equation}\label{28}
\left|\eta_2^q(1-\eta_{j1}^\xi)L(\hat{Z}_{j0})\right|\leq C\frac{\log|\log\varepsilon|}{\mu_j|\log\varepsilon|}\left(\varepsilon^2+\sum_{j=1}^m \frac{\mu_j^\sigma}{(\mu_j+|y-\xi_j'|)^{2+\sigma}}+\frac{\mu_0^{2+2\hat{\alpha}}|y-q'|^{2\alpha}}{(\mu_0+|y-q'| )^{4+2\hat{\alpha}+2\alpha}}\right).
\end{equation}
From \eqref{27}, \eqref{31}, \eqref{30}, \eqref{29} and \eqref{28}, we get \eqref{25} for $j=1,2,\cdots,m$. From the same argument, we see that \eqref{25} still holds in the case of $j=0$.

Next, we prove \eqref{26}. It is easy to see
\[L(\chi_k Z_{ks})=\chi_k L(Z_{ks})-\Delta \chi_k Z_{ks}-2\nabla \chi_k\cdot \nabla Z_{ks}-\nabla(\log a(\varepsilon y))Z_{ks}, \quad \mathrm{where} \quad k=1,2,\cdots, m; s=1, J_k.\]
In the case of $R_0\mu_k\leq |z_k|\leq (R_0+1)\mu_k$, we have
\begin{align}\label{32}
 \nonumber \left|\Delta \chi_k Z_{ks}+2\nabla \chi_k\cdot \nabla Z_{ks}+\nabla(\log a(\varepsilon y))\cdot \nabla \chi_k Z_{ks}\right|\leq  & C\left(\frac1{\mu_k^2}+\frac{\varepsilon}{\mu_k}\right)\frac1{\mu_k+|z_k|} +\frac{C}{\mu_k}\frac1{(\mu_k+|z_k|)^2} \\
  \leq  & \frac{C}{\mu_k}\frac{\mu_k^\sigma}{(\mu_k+|y-\xi_k'|)^{2+\sigma}}.
\end{align}
In the case of $|z_k|\leq (R_0+1)\mu_k$, we get
\begin{align*}
  L(Z_{ks}) =& \left(\frac{8\mu_k^2}{[\mu_k^2+|z_k|^2]^2}-\frac{8\mu_k^2}{[\mu_k^2+|y-\xi_k'|^2]^2} \frac{|\varepsilon y-q|^{2\alpha}}{|\xi_k-q|^{2\alpha}}\right) \frac1{\mu_k}Z_s\left( \frac{z_k}{\mu_k}\right) +O\left(\frac{\varepsilon}{(\mu_k+|z_k|)^2}\right)\\
   & +O\left(\frac{\varepsilon^2}{\mu_k}\right) +\frac{8\mu_k^2}{[\mu_k^2+|y-\xi_k'|^2]^2}\left(O\left(|\varepsilon y-\xi_k|^\beta\right)+O\left(\sum_{l=0}^m (\varepsilon \mu_l)^{\frac2p-1}\right)\right)\frac1{\mu_k}Z_s\left( \frac{z_k}{\mu_k}\right).
\end{align*}
Hence
\begin{equation}\label{33}
\left|\chi_k L(Z_{ks})\right|\leq \frac1{\mu_k}\left(\varepsilon^2+\frac{\mu_k^\sigma}{(\mu_k+|y-\xi_k'|)^{2+\sigma}}\right).
\end{equation}
The estimate \eqref{26} follows from \eqref{32} and \eqref{33}.

\end{proof}

\begin{lemma}\label{lm5}
Under the same condition of Lemma \ref{lm3}, we have
\begin{equation}\label{34}
\int_{\Omega_\varepsilon} a(\varepsilon y)L(\tilde{Z}_{00})\tilde{Z}_{00}dy=\frac{c_0a(q)}{4(1+\alpha)\mu_0^2|\log(\varepsilon\mu_0 R)|}+O\left(\frac1{\mu_0^2|\log\varepsilon|R}\right),
\end{equation}
\begin{equation}\label{35}
\int_{\Omega_\varepsilon} a(\varepsilon y)L(\tilde{Z}_{j0})\tilde{Z}_{j0}dy=\frac{c_j a(\xi_j)}{4\mu_j^2|\log(\varepsilon \mu_j R)|}+O\left(\frac1{\mu_j^2|\log\varepsilon|R}\right), \quad \mathrm{where}\quad j=1,2,\cdots,m
\end{equation}
and
\begin{equation}\label{36}
\int_{\Omega_\varepsilon} a(\varepsilon y)L(\tilde{Z}_{k0})\tilde{Z}_{j0}dy=O\left(\frac{\log^2|\log\varepsilon|}{\mu_j\mu_k |\log\varepsilon|^2}\right)\quad \mathrm{for}\quad j\not=k.
\end{equation}
\end{lemma}
\begin{proof}
From \eqref{27}, we get
\begin{align*}
  &\int_{\Omega_\varepsilon}a(\varepsilon y)L(\tilde{Z}_{j0})\tilde{Z}_{j0}dy \\
  =& \int_{\Omega_\varepsilon}a(\varepsilon y)\left[\eta_{j1}^\xi L(Z_{j0})+\eta_2^q(1-\eta_{j1}^\xi)L(\hat{Z}_{j0})\right]\tilde{Z}_{j0}+\int_{\Omega_\varepsilon} a(\varepsilon y) \left[-\Delta \eta_{j1}^\xi (Z_{j0}-\hat{Z}_{j0})\right.\\
  &-2\nabla \eta_{j1}^\xi\cdot \nabla(Z_{j0}-\hat{Z}_{j0})-\varepsilon (\nabla\log a)(\varepsilon y)\cdot \nabla \eta_{j1}^\xi (Z_{j0}-\hat{Z}_{j0})-\Delta \eta_2^q \hat{Z}_{j0}  \\
  &\left.-2\nabla \eta_2^q\cdot \nabla \hat{Z}_{j0}-\varepsilon (\nabla \log a)(\varepsilon y)\cdot \nabla \eta_2^q \hat{Z}_{j0}\right]\tilde{Z}_{j0}dy \\
  =:&K_5+K_6.
\end{align*}
Notice if $|z_j|\leq (R+1)\mu_j$, we get
\[|z_0|\leq 2|y-q'|\leq 2\left[|y-\xi_j'|+|\xi_j'-q'|\right]\leq 2\left[2|z_j|+\frac{d}\varepsilon\right]\leq \frac{4d}\varepsilon.\]
Thus
\[\tilde{Z}_{j0}=\eta_{j1}^\xi(Z_{j0}-\hat{Z}_{j0})+\eta_2^q \hat{Z}_{j0}.\]
If $\nabla\eta_2^q\not=0$, we get $|y-\xi_j'|\geq \frac{d}{\varepsilon}$. Integrating by part, we get
\begin{align}\label{37}
\nonumber  K_6= & \int_{\Omega_\varepsilon} a(\varepsilon y)|\nabla \eta_{j1}^\xi|^2(Z_{j0}-\hat{Z}_{j0})^2dy+\int_{\Omega_\varepsilon} a(\varepsilon y)\nabla \eta_{j1}^\xi\cdot \nabla \hat{Z}_{j0}(Z_{j0}-\hat{Z}_{j0})dy \\
   &  -\int_{\Omega_\varepsilon} a(\varepsilon y)\nabla\eta_{j1}^\xi\cdot \nabla(Z_{j0}-\hat{Z}_{j0})\hat{Z}_{j0}dy+\int_{\Omega_\varepsilon} a(\varepsilon y)|\nabla \eta_2^q|^2 \hat{Z}_{j0}^2dy
\end{align}
From \eqref{89} and \eqref{90}, We have
\begin{equation*}
\int_{\Omega_\varepsilon} a(\varepsilon y)|\nabla \eta_{j1}^\xi|^2(Z_{j0}-\hat{Z}_{j0})^2dy=O\left(\frac1{\mu_j^2|\log\varepsilon|^2}\right),
\end{equation*}
\[\int_{\Omega_\varepsilon} a(\varepsilon y)|\nabla \eta_2^q|^2\hat{Z}_{j0}^2dy=O\left(\frac1{\mu_j^2|\log\varepsilon|^2}\right)\]
and
\[\int_{\Omega_\varepsilon} a(\varepsilon y)\nabla \eta_{j1}^\xi\cdot \nabla \hat{Z}_{j0}(Z_{j0}-\hat{Z}_{j0})dy=O\left(\frac1{\mu_j^2|\log\varepsilon |R}\right).\]
For the last term in \eqref{37}, we get
\begin{align*}
  &-\int_{\Omega_\varepsilon} a(\varepsilon y) \nabla \eta_{j1}^\xi \cdot \nabla(Z_{j0}-\hat{Z}_{j0})\hat{Z}_{j0} \\
  = & \varepsilon a_{j0}^\xi \int_{\Omega_\varepsilon} a(\varepsilon y) \nabla \eta_{j1}^\xi\cdot \nabla G(\varepsilon y, \xi_j)Z_{j0}+O\left(\frac1{\mu_j^2|\log\varepsilon|^2}\right) \\
  = & -\frac{4\varepsilon a_{j0}^\xi}{c_j}\int_{\Omega_\varepsilon} a(\varepsilon y)\nabla \eta_{j1}^\xi\cdot \frac{\varepsilon y-\xi_j}{|\varepsilon y-\xi_j|^2}Z_{j0} +O\left(\frac1{\mu_j^2|\log\varepsilon|^2}\right) \\
  = & -\frac{4a_{j0}^\xi}{c_j}\int_{\Omega_\varepsilon} a(\varepsilon y)\left[\frac1{\mu_j}\eta'\left(\frac{|z_j|}{\mu_j}\right)\frac{A_j^T z_j}{|z_j|}+O\left(\frac{\varepsilon }{\mu_j}|y-\xi_j'| |\eta'|\right)\right]\frac{y-\xi_j'}{|y-\xi_j'|^2}Z_{j0}dy +O\left(\frac1{\mu_j^2|\log\varepsilon|^2}\right) \\
  = & -\frac{4a_{j0}^\xi}{c_j\mu_j^2}\int_{\Omega_\varepsilon} a(\varepsilon y) \eta'\left(\frac{|z_j|}{\mu_j}\right)\frac{A_j^T z_j}{|z_j|}\cdot \frac{y-\xi_j'}{|y-\xi_j'|^2}\left(1+ O\left(\frac1{R^2}\right)\right)dy+ O\left(\frac1{\mu_j^2|\log\varepsilon|^2}\right)\\
  = & -\frac{4a_{j0}^\xi}{c_j\mu_j^2}\int_{\R_{\xi_j}} \left[a(\xi_j)+O(\varepsilon |z_j|)\right]\eta'\left(\frac{|z_j|}{\mu_j}\right)\frac{A_j^{-1}z_j+O(\varepsilon |z_j|^2)}{|A_j^{-1}z_j+O(\varepsilon |z_j|^2)|^2}\cdot \frac{A_j^T z_j}{|z_j|}(1+O(\varepsilon |z_j|))dz_j \\
  &+ O\left(\frac1{\mu_j^2|\log\varepsilon|R}\right) \\
  =&-\frac{4a_{j0}^\xi}{c_j\mu_j^2}a(\xi_j)\int_{\R_{\xi_j}} \eta'\left(\frac{|z_j|}{\mu_j}\right)\frac1{|z_j|}dz_j+O\left(\frac1{\mu_j^2| \log\varepsilon|R}\right) \\
  =&\frac{a_{j0}^\xi}{\mu_j}a(\xi_j)+O\left(\frac1{\mu_j^2|\log\varepsilon|R}\right)=\frac{c_j a(\xi_j)}{4\mu_j^2|\log(\varepsilon \mu_j R)|}+O\left(\frac1{\mu_j^2|\log\varepsilon|R}\right).
\end{align*}

For the term $K_5$, it it apparent that
\begin{align*}
  K_5= & \int_{\Omega_\varepsilon} (\eta_2^q)^2 a(\varepsilon y)\left[-\Delta Z_{j0}-\nabla(\log a(\varepsilon y))\cdot \nabla Z_{j0}+\varepsilon^2(Z_{j0}-\frac1{\mu_j})+\frac{\varepsilon^2}{\mu_j}\eta_{j1}^\xi-W Z_{j0}\right. \\
  &\left.+(1-\eta_{j1}^\xi)W\left(Z_{j0}-\hat{Z}_{j0}\right)\right]\cdot \left[Z_{j0}+(1-\eta_{j1}^\xi)(\hat{Z}_{j0}-Z_{j0})\right]dy.
\end{align*}
From direct computation, we get
\[\int_{\Omega_\varepsilon} (\eta_2^q)^2a(\varepsilon y)\left[\varepsilon^2\left(Z_{j0}-\frac1{\mu_j}\right)+\frac{\varepsilon^2}{\mu_j} \eta_{j1}^\xi\right] \left[Z_{j0}+(1-\eta_{j1}^\xi)\left(\hat{Z}_{j0}-Z_{j0}\right)\right]dy=O\left(\frac1{\mu_j^2|\log\varepsilon|^2}\right)\]
and
\[\int_{\Omega_\varepsilon} (\eta_2^q)^2a(\varepsilon y)\nabla(\log a(\varepsilon y))\cdot \nabla Z_{j0}\left[Z_{j0}+(1-\eta_{j1}^\xi)\left(\hat{Z}_{j0}-Z_{j0}\right)\right]dy=O\left(\frac{ \varepsilon}{\mu_j}\right).\]
Recall that
\[A_k=\left\{y\in \Omega_\varepsilon\;:\; |y-\xi_k'|\leq \frac1{\varepsilon|\log\varepsilon|^{2\kappa}}\right\}, \qquad \mathrm{for} \qquad k=1,2,\cdots,m,\]
\[A_0=\left\{y\in \Omega_\varepsilon\;:\; |y-q'|\leq \frac1{\varepsilon|\log\varepsilon|^{2\kappa}}\right\},\]
and $A_{m+1}=\Omega_\varepsilon\char92 \cup_{j=0}^m A_j$.

It is obvious that
\begin{align*}
   & \int_{\Omega_\varepsilon} (\eta_2^q)^2 a(\varepsilon y)\left[-\Delta Z_{j0}-WZ_{j0}+(1-\eta_{j1}^\xi) W\left(Z_{j0}-\hat{Z}_{j0}\right)\right]\left[Z_{j0}+(1-\eta_{j1}^\xi)\left(\hat{Z}_{j0}-Z_{j0}\right)\right]dy \\
  = & \int_{A_j}(\eta_2^q)^2 a(\varepsilon y)\left[-\Delta Z_{j0}-WZ_{j0}+(1-\eta_{j1}^\xi) W\left(Z_{j0}-\hat{Z}_{j0}\right)\right]\left[Z_{j0}+(1-\eta_{j1}^\xi)\left(\hat{Z}_{j0}-Z_{j0}\right)\right]dy  \\
  & +\sum_{k=0,k\not=j}^m\int_{A_k}(\eta_2^q)^2 a(\varepsilon y)\left[-\Delta Z_{j0}+W\left(\frac1{\mu_j}-Z_{j0}\right)-a_{j0}^\xi WG(\varepsilon y, \xi_j)-W(\mathring{Z}_{j0}-Z_{j0})\chi_{\{1\leq j\leq l\}})\right] \\
  &\qquad\times\left[Z_{j0}-\frac1{\mu_j}+a_{j0}^\xi G(\varepsilon y, \xi_j)+(\mathring{Z}_{j0}-Z_{j0})\chi_{\{1\leq j\leq l\}})\right]dy \\
  & \int_{A_{m+1}}(\eta_2^q)^2 a(\varepsilon y)\left[-\Delta Z_{j0}+W\left(\frac1{\mu_j}-Z_{j0}\right)-a_{j0}^\xi WG(\varepsilon y, \xi_j)-W(\mathring{Z}_{j0}-Z_{j0})\chi_{\{1\leq j\leq l\}})\right] \\
  &\qquad \times\left[Z_{j0}-\frac1{\mu_j}+a_{j0}^\xi G(\varepsilon y, \xi_j)+(\mathring{Z}_{j0}-Z_{j0})\chi_{\{1\leq j\leq l\}})\right]dy \\
  =&I_{14}+I_{15}+I_{16}.
\end{align*}

On $A_j$, we get
\begin{align}\label{45}
\nonumber   & \left|-\Delta Z_{j0}-WZ_{j0}+(1-\eta_{j1}^\xi)W\left(Z_{j0}-\hat{Z}_{j0}\right)\right| \\
\nonumber  \leq & C\left[\frac1{\mu_j}\left(\frac{8\mu_j^2}{[\mu_j^2+|z_j|^2]^2}-\frac{8\mu_j^2}{[\mu_j^2 +|y-\xi_j'|^2]^2}\frac{|\varepsilon y-q|^{2\alpha}}{|\xi_j-q|^{2\alpha}}\right) +\frac{\varepsilon \mu_j}{(\mu_j+|z_j|)^3}+\frac{\varepsilon^3}{\mu_j}\right. \\
\nonumber  &\left.\frac{8\mu_j}{[\mu_j^2+|y-\xi_j'|^2]^2}\left(|\varepsilon y-\xi_j|^\beta+\sum_{l=0}^m (\varepsilon \mu_l)^{\frac2p-1}\right)+(1-\eta_{j1}^\xi)\frac{\mu_j^2}{(\mu_j+|y-\xi_j'|)^4} \frac{1+\log|\frac{y-\xi_j'}{\mu_j}|}{\mu_j|\log\varepsilon|}\right] \\
  \leq & \frac{C}{(\mu_j+|y-\xi_j'|)^3}\left[(\varepsilon \mu_j)^\beta+\sum_{k=0}^m (\varepsilon \mu_k)^{\frac2p-1}\right]+\frac{C\varepsilon^3}{\mu_j}+(1-\eta_{j1}^\xi)\frac{ C\mu_j^2}{(\mu_j+|y-\xi_j'|)^4}\frac{1+\log|\frac{y-\xi_j'}{\mu_j}|}{\mu_j|\log\varepsilon|}.
\end{align}
However
\[\left|Z_{j0}+(1-\eta_{j1}^\xi)\left(\hat{Z}_{j0}-Z_{j0}\right)\right|\leq \frac1{\mu_j}+(1-\eta_{j1}^\xi)\frac{1+\log|\frac{y-\xi_j'}{\mu_j}|}{\mu_j|\log\varepsilon|}.\]
Hence we get
\[|I_{14}|\leq \frac{C}{\mu_j^2|\log\varepsilon|R}.\]
For $y\in A_k$ ($k\not=j,0$), we have $|y-\xi_j'|\geq \frac1{2\varepsilon|\log\varepsilon|^\kappa}$. Hence there hold
\begin{align*}
   & \left|-\Delta Z_{j0}+W\left(\frac1{\mu_j}-Z_{j0}\right)-a_{j0}^\xi W G(\varepsilon y, \xi_j)-W(\mathring{Z}_{j0}-Z_{j0})\chi_{\{1\leq j\leq l\}}\right| \\
  \leq & C\left[\frac{\mu_j}{(\mu_j+|z_j|)^4}+\frac{\varepsilon \mu_j}{(\mu_j+|z_j|)^3}+\frac{\mu_k^2}{(\mu_k+|y-\xi_k'|)^4}\frac{\mu_j}{ [\mu_j+|z_j|]^2}+\frac{\log|\log\varepsilon|}{\mu_j|\log\varepsilon|}\frac{\mu_k^2}{ (\mu_k+|y-\xi_k'|)^4}\right]
\end{align*}
and
\[\left|Z_{j0}-\frac1{\mu_j}+a_{j0}^\xi G(\varepsilon y, \xi_j)+(\mathring{Z}_{j0}-Z_{j0})\chi_{\{1\leq j\leq l\}}\right|\leq C\frac{\log|\log\varepsilon|}{\mu_j|\log\varepsilon|}.\]
Then we get
\begin{eqnarray*}
&&\int_{A_k}(\eta_2^q)^2 a(\varepsilon y)\left[-\Delta Z_{j0}+W\left(\frac1{\mu_j}-Z_{j0}\right)-a_{j0}^\xi WG(\varepsilon y, \xi_j)-W(\mathring{Z}_{j0}-Z_{j0})\chi_{\{1\leq j\leq l\}})\right] \\
&&\qquad\times\left[Z_{j0}-\frac1{\mu_j}+a_{j0}^\xi G(\varepsilon y, \xi_j)+(\mathring{Z}_{j0}-Z_{j0})\chi_{\{1\leq j\leq l\}})\right]dy =O\left(\frac{\log^2|\log \varepsilon|}{\mu_j^2 |\log\varepsilon|^2}\right).
\end{eqnarray*}
From the same argument, we also get the estimate above also hold in the case of $k=0$. Hence we have $|I_{15}|\leq C\frac{\log^2|\log\varepsilon|}{\mu_j^2|\log\varepsilon|^2}$. From the same method along with \eqref{38}, we also get $|I_{16}|\leq C\frac{\log^2|\log\varepsilon|}{\mu_j^2|\log\varepsilon|^2}$.

As a summary, we have $|K_5|\leq C\frac{1}{\mu_j^2|\log\varepsilon|R}$. Hence \eqref{35} hold.

From the same procedure, we also get \eqref{34}.

Now we prove \eqref{36}. We only prove \eqref{36} for $j,k=1,2, \cdots,m$. From \eqref{27}, we get
\begin{align*}
  &\int_{\Omega_\varepsilon}a(\varepsilon y)L(\tilde{Z}_{k0})\tilde{Z}_{j0}dy \\
  =& \int_{\Omega_\varepsilon}a(\varepsilon y)\left[\eta_{k1}^\xi L(Z_{k0})+\eta_2^q(1-\eta_{k1}^\xi)L(\hat{Z}_{k0})\right]\tilde{Z}_{j0}dy+\int_{ \Omega_\varepsilon} a(\varepsilon y) \left[-\Delta \eta_{k1}^\xi (Z_{k0}-\hat{Z}_{k0}) \right. \\
  & -2\nabla \eta_{k1}^\xi\cdot \nabla(Z_{k0}-\hat{Z}_{k0}) -\varepsilon (\nabla\log a)(\varepsilon y)\cdot \nabla \eta_{k1}^\xi (Z_{k0}-\hat{Z}_{k0})-\Delta \eta_2^q \hat{Z}_{k0}-2\nabla \eta_2^q\cdot \nabla \hat{Z}_{k0} \\
  &\left.-\varepsilon (\nabla \log a)(\varepsilon y)\cdot \nabla \eta_2^q \hat{Z}_{k0}\right]\tilde{Z}_{j0}dy \\
  =:&K_7+K_8.
\end{align*}
Integrating by part, we have
\begin{align*}
  K_8=& \int_{\Omega_\varepsilon} a(\varepsilon y)\nabla \eta_{k1}^\xi\cdot \nabla \tilde{Z}_{j0}(Z_{k0}-\hat{Z}_{k0})dy-\int_{\Omega_\varepsilon} a(\varepsilon y)\nabla \eta_{k1}^\xi\cdot \nabla (Z_{k0}-\hat{Z}_{k0})\tilde{Z}_{j0}dy \\
   & +\int_{\Omega_\varepsilon} a(\varepsilon y)\nabla \eta_2^q\cdot \nabla \tilde{Z}_{j0}\hat{Z}_{k0}dy-\int_{\Omega_\varepsilon}a(\varepsilon y) \nabla \eta_2^q\cdot \nabla \hat{Z}_{k0}\tilde{Z}_{j0}dy.
\end{align*}
From \eqref{89}, \eqref{90} and \eqref{121}, we have
\[\int_{\Omega_\varepsilon} a(\varepsilon y)\nabla \eta_{k1}^\xi\cdot \nabla \tilde{Z}_{j0}(Z_{k0}-\hat{Z}_{k0})dy=O\left(\frac1{\mu_j\mu_k|\log\varepsilon|^2}\right),\]
\[\int_{\Omega_\varepsilon} a(\varepsilon y)\nabla \eta_{k1}^\xi\cdot \nabla (Z_{k0}-\hat{Z}_{k0})\tilde{Z}_{j0}dy=O\left(\frac{\log|\log\varepsilon|}{\mu_j\mu_k| \log\varepsilon|^2}\right),\]
\[\int_{\Omega_\varepsilon} a(\varepsilon y)\nabla \eta_2^q\cdot \nabla \tilde{Z}_{j0}\hat{Z}_{k0}dy=O\left(\frac1{\mu_j\mu_k|\log\varepsilon|^2}\right)\]
and
\[\int_{\Omega_\varepsilon}a(\varepsilon y)\nabla \eta_2^q\cdot \nabla \hat{Z}_{k0}\tilde{Z}_{j0}dy=O\left(\frac1{\mu_j\mu_k|\log\varepsilon|^2}\right).\]
For the term $K_7$, we have
\begin{align*}
  K_7= & \int_{\Omega_\varepsilon} a(\varepsilon y)(\eta_2^q)^2\left[-\Delta Z_{k0}-(\nabla \log a(\varepsilon y))\cdot \nabla Z_{k0}+\varepsilon^2\left(Z_{k0}-\frac1{\mu_k}\right)- W Z_{k0}+\frac{\varepsilon^2}{\mu_k}\eta_{k1}^\xi\right.\\
   & +\left.(1-\eta_{k1}^\xi)W\left(\frac1{\mu_k}-a_{k0}^\xi G(\varepsilon y, \xi_k)-(\mathring{Z}_{k0}-Z_{k0})\chi_{\{1\leq k\leq l\}}\right)\right]\left[Z_{j0}+(1-\eta_{j1}^\xi)(\hat{Z}_{j0}-Z_{j0})\right]dy.
\end{align*}
From direct computation, we have
\[\int_{\Omega_\varepsilon} a(\varepsilon y)(\eta_2^q)^2\left[\varepsilon^2\left(Z_{k0}-\frac1{\mu_k}\right) +\frac{\varepsilon^2}{\mu_k}\eta_{k1}^\xi \right] \left[Z_{j0}+(1-\eta_{j1}^\xi)(\hat{Z}_{j0}-Z_{j0})\right]dy=O\left(\frac1{\mu_j\mu_k| \log\varepsilon|^2}\right)\]
and
\[\int_{\Omega_\varepsilon}a(\varepsilon y)(\eta_2^q)^2(\nabla\log a(\varepsilon y), \nabla Z_{k0})\left[Z_{j0}+(1-\eta_{j1}^\xi)(\hat{Z}_{j0}-Z_{j0})\right]=O\left(\frac1{\mu_j\mu_k| \log\varepsilon|^2}\right).\]
For the rest part, we get
\begin{align*}
   & \int_{\Omega_\varepsilon} a(\varepsilon y)(\eta_2^q)^2\left[-\Delta Z_{k0}- W Z_{k0}+(1-\eta_{k1}^\xi)W\left(\frac1{\mu_k}-a_{k0}^\xi G(\varepsilon y, \xi_k)-(\mathring{Z}_{k0}-Z_{k0})\chi_{\{1\leq k\leq l\}}\right)\right] \\
   &\qquad \left[Z_{j0}+(1-\eta_{j1}^\xi)(\hat{Z}_{j0}-Z_{j0})\right]dy \\
   =&\int_{A_k} a(\varepsilon y)(\eta_2^q)^2\left[-\Delta Z_{k0}-WZ_{k0}+(1-\eta_{k1}^\xi)W\left(\frac1{\mu_k}-a_{k0}^\xi G(\varepsilon y, \xi_k)-(\mathring{Z}_{k0}-Z_{k0})\chi_{\{1\leq k\leq l\}}\right)\right]\hat{Z}_{j0}dy \\
   &+\int_{A_j} a(\varepsilon y)(\eta_2^q)^2\left[-\Delta Z_{k0}+W\left(\frac1{\mu_k}-Z_{k0}\right)-a_{k0}^\xi WG(\varepsilon y, \xi_k)-W(\mathring{Z}_{k0}-Z_{k0})\chi_{\{1\leq k\leq l\}}\right] \\
   &\qquad\left[Z_{j0}+(1-\eta_{j1}^\xi)\left(a_{j0}^\xi G(\varepsilon y, \xi_j)-\frac1{\mu_j}+(\mathring{Z}_{j0}-Z_{j0})\chi_{\{1\leq j\leq l\}}\right)\right]dy \\
   &+\sum_{i=0,\atop i\not=k,j}^m \int_{A_i} a(\varepsilon y)(\eta_2^q)^2\left[-\Delta Z_{k0}+W\left(\frac1{\mu_k}-Z_{k0}\right)-a_{k0}^\xi WG(\varepsilon y, \xi_k)-W(\mathring{Z}_{k0}-Z_{k0})\chi_{\{1\leq k\leq l\}}\right]\hat{Z}_{j0}dy \\
   &+ \int_{A_{m+1}} a(\varepsilon y)(\eta_2^q)^2\left[-\Delta Z_{k0}+W\left(\frac1{\mu_k}-Z_{k0}\right)-a_{k0}^\xi WG(\varepsilon y, \xi_k)-W(\mathring{Z}_{k0}-Z_{k0})\chi_{\{1\leq k\leq l\}}\right]\hat{Z}_{j0}dy.
\end{align*}
Now we estimate each term above. On $A_k$, we have
\[|\hat{Z}_{j0}|=\left|Z_{j0}-\frac1{\mu_j}+a_{j0}^\xi G(\varepsilon y, \xi_j)+(\mathring{Z}_{k0}-Z_{k0})\chi_{\{1\leq k\leq l\}}\right|\leq C\frac{\log|\log\varepsilon|}{\mu_j|\log\varepsilon|}.\]
Using the same method as in \eqref{45}, we get
\begin{align*}
& \left|-\Delta Z_{k0}-WZ_{k0}+(1-\eta_{k1}^\xi)W\left(\frac1{\mu_k}-a_{k0}^\xi G(\varepsilon y, \xi_k)- (\mathring{Z}_{k0}-Z_{k0})\chi_{\{1\leq k\leq l\}}\right)\right| \\
\leq & \frac{C}{(\mu_k+|y-\xi_k'|)^3}\left[(\varepsilon \mu_k)^\beta+\sum_{l=0}^m (\varepsilon \mu_l)^{\frac2p-1}\right]+\frac{C\varepsilon^3}{\mu_k}+(1-\eta_{k1}^\xi)\frac{C\mu_k^2}{( \mu_k+|y-\xi_k'|)^4}\frac{1+\log|\frac{y-\xi_k'}{\mu_k}|}{\mu_k|\log\varepsilon|}.
\end{align*}
Hence
\[\int_{A_k}a(\varepsilon y)(\eta_2^q)^2\left[-\Delta Z_{k0}-WZ_{k0}+(1-\eta_{k1}^\xi)W\left(\hat{Z}_{j0}-Z_{j0}\right)\right]\hat{Z}_{j0}=O\left(\frac{\log^2|\log\varepsilon|}{\mu_j\mu_k |\log\varepsilon|^2}\right).\]
However, on $A_j$, we have
\[|z_k|\geq \frac12|y-\xi_k'|\geq \frac12\left[|\xi_j'-\xi_i'|-|y-\xi_j'|\right]\geq \frac1{4\varepsilon |\log\varepsilon|^\kappa}\]
for $\varepsilon$ small enough. Then we get
\begin{eqnarray*}
&&\int_{A_j} a(\varepsilon y)(\eta_2^q)^2\left[-\Delta Z_{k0}+W\left(\frac1{\mu_k}-Z_{k0}\right)-W(\mathring{Z}_{k0}-Z_{k0})\chi_{\{1\leq k\leq l\}}\right] \\
&&\qquad\left[Z_{j0}+(1-\eta_{j1}^\xi)(\hat{Z}_{j0} -Z_{j0})\right]dy=O\left(\frac1{\mu_j\mu_k|\log\varepsilon|^2}\right)
\end{eqnarray*}
and
\[\int_{A_j} a(\varepsilon y)(\eta_2^q)^2 a_{k0}^\xi WG(\varepsilon y, \xi_k)(1-\eta_{j1}^\xi)(\hat{Z}_{j0}-Z_{j0})dy=O\left(\frac{\log|\log\varepsilon|}{ \mu_j\mu_k|\log\varepsilon|^2}\right).\]
For the last term, we get
\begin{align*}
  & \int_{A_j} a(\varepsilon y)(\eta_2^q)^2 a_{k0}WG(\varepsilon y, \xi_k)Z_{j0}dy \\
= & \int_{A_j} a(\varepsilon y)(\eta_2^q)^2 a_{k0}\frac{8\mu_j^2}{[\mu_j^2+|z_j|^2]^2}G(\varepsilon y, \xi_k)Z_{j0}dy \\
   & +\int_{A_j} a(\varepsilon y)(\eta_2^q)^2 a_{k0}\left(W-\frac{8\mu_j^2}{[\mu_j^2+|z_j|^2]^2}\right)G(\varepsilon y, \xi_k)Z_{j0}dy.
\end{align*}
Using the estimate of $W$, we get
\[\int_{A_j} a(\varepsilon y)(\eta_2^q)^2 a_{k0}\left(W-\frac{8\mu_j^2}{[\mu_j^2+|z_j|^2]^2}\right)G(\varepsilon y, \xi_k)Z_{j0}dy =O\left(\frac1{\mu_j\mu_k|\log\varepsilon|^2}\right).\]
On $A_j$, we get $|y-\xi_j'|\leq \frac1{\varepsilon|\log\varepsilon|^{2\kappa}}$. For $\varepsilon$ small enough, we have
\[|z_0|\leq 2|y-q'|\leq 2\left[|y-\xi_j'|+|\xi_j'-q'|\right]\leq 2\left[\frac1{\varepsilon |\log\varepsilon|^{2\kappa}}+\frac d{\varepsilon}\right]\leq \frac {4d}\varepsilon.\]
Hence
\begin{align*}
   & \int_{A_j} a(\varepsilon y)(\eta_2^q)^2 a_{k0}\frac{8\mu_j^2}{[\mu_j^2+|z_j|^2]^2}G(\varepsilon y, \xi_k)Z_{j0}dy \\
  = & a(\xi_j)a_{k0}^\xi G(\xi_j,\xi_k)\int_{A_j} \frac{8\mu_j^2}{[\mu_j^2+|z_j|^2]^2}Z_{j0} +O\left(\frac1{\mu_j\mu_k|\log\varepsilon|^2}\right) \\
  = & a(\xi_j)a_{k0}^\xi G(\xi_j,\xi_k)\int_{A_j}-\Delta_{z_j} Z_{j0}dy+O\left(\frac1{\mu_j\mu_k|\log\varepsilon|^2}\right) \\
  = & a(\xi_j)a_{k0}^\xi G(\xi_j,\xi_k)\int_{A_j}-\Delta Z_{j0}dy+O\left(\frac1{\mu_j\mu_k|\log\varepsilon|^2}\right) \\
  = & -a(\xi_j)a_{k0}^\xi G(\xi_j,\xi_k)\int_{\partial A_j} \frac{\partial Z_{j0}}{\partial n}d\sigma +O\left(\frac1{\mu_j\mu_k|\log\varepsilon|^2}\right) \\
  = & O\left(\frac1{\mu_j\mu_k|\log\varepsilon|^2}\right).
\end{align*}
For $l\not=k,j,0$, we have the following estimates hold on $A_l$:
\[|y-\xi_j'|\geq \frac1{2\varepsilon|\log\varepsilon|^\kappa}, \qquad |y-\xi_k'|\geq \frac1{2\varepsilon|\log\varepsilon|^\kappa}.\]
Then we also have $\hat{Z}_{j0}=O\left(\frac{\log|\log\varepsilon|}{\mu_j|\log\varepsilon|}\right)$.
\begin{align*}
   & -\Delta Z_{k0}+W\left(\frac1{\mu_k}-Z_{k0}\right)-a_{k0}^\xi WG(\varepsilon y, \xi_k)-W(\mathring{Z}_{k0}-Z_{k0})\chi_{\{1\leq k\leq l\}} \\
  = &O\left(\left(\frac{\mu_l^2}{(\mu_l+|y-\xi_l'|)^4}+\varepsilon^2\right)\frac{\log|\log\varepsilon|}{\mu_k |\log\varepsilon|}\right).
\end{align*}
Then
\begin{eqnarray*}
&&\int_{A_l} a(\varepsilon y)(\eta_2^q)^2\left[-\Delta Z_{k0}+W\left(\frac1{\mu_k}-Z_{k0}\right)-a_{k0}^\xi WG(\varepsilon y, \xi_k)-W(\mathring{Z}_{k0}-Z_{k0})\chi_{\{1\leq k\leq l\}}\right]\hat{Z}_{j0}dy \\
&=&O\left(\frac{\log^2|\log\varepsilon|}{\mu_j\mu_k| \log\varepsilon|^2}\right).
\end{eqnarray*}
From the same method, we also get
\begin{eqnarray*}
&&\int_{A_0} a(\varepsilon y)(\eta_2^q)^2\left[-\Delta Z_{k0}+W\left(\frac1{\mu_k}-Z_{k0}\right)-a_{k0}^\xi WG(\varepsilon y, \xi_k)-W(\mathring{Z}_{k0}-Z_{k0})\chi_{\{1\leq k\leq l\}}\right]\hat{Z}_{j0}dy \\
&=&O\left(\frac{\log^2|\log\varepsilon|}{\mu_j\mu_k| \log\varepsilon|^2}\right)
\end{eqnarray*}
and
\begin{eqnarray*}
&&\int_{A_{m+1}} a(\varepsilon y)(\eta_2^q)^2\left[-\Delta Z_{k0}+W\left(\frac1{\mu_k}-Z_{k0}\right)-a_{k0}^\xi WG(\varepsilon y, \xi_k)-W(\mathring{Z}_{k0}-Z_{k0})\chi_{\{1\leq k\leq l\}}\right]\hat{Z}_{j0}dy \\
&=&O\left(\frac{\log^2|\log\varepsilon|}{\mu_j\mu_k| \log\varepsilon|^2}\right).
\end{eqnarray*}
Hence we get \eqref{36}.

\end{proof}

\noindent\textbf{Declaration:}

\textbf{Acknowledge:}  This work was supported by the National Natural Science Foundation of China (no. 12001298) and Inner Mongolia Autonomous Region Natural Science Foundation(no. 2024LHMS01001). The author thank the anonymous reviewers for his/her valuable comments and suggestion.

\textbf{Competing interests:} The author have no competing interests as defined by Springer, or other interests that might be perceived to influence the results and/or discussion reported in this paper.

\textbf{Data availability:} This paper does not report data generation or analysis.


\end{document}